\documentclass[12pt]{article}
\usepackage{tikz}
\usepackage{amsfonts}
\usepackage{amsmath}
\usepackage{amsthm}
\usepackage{amssymb}
\usepackage{mathrsfs}
\usepackage{hyperref}
\hypersetup{
 colorlinks=true,
 citecolor=cyan,
 linkcolor=magenta,
 menucolor=magenta,
}
\usepackage{bm}
\usepackage{bbm}
\usepackage{enumitem}
\usepackage{stmaryrd}
\usetikzlibrary{cd}
\usepackage[left=32truemm,right=32truemm,top=30truemm,bottom=30truemm, includefoot]{geometry}

\theoremstyle{plain}
\newtheorem{thm}{Theorem}
\newtheorem{prop}[thm]{Proposition}
\newtheorem{lem}[thm]{Lemma}
\newtheorem{cor}[thm]{Corollary}

\theoremstyle{definition}

\newtheorem{defi}[thm]{Definition}
\newtheorem{rem}[thm]{Remark}

\newtheorem{clm}[thm]{Claim}

\numberwithin{thm}{section}  

\usepackage{relsize}
\usepackage[bbgreekl]{mathbbol}
\usepackage{amsfonts}

\DeclareSymbolFontAlphabet{\mathbb}{AMSb}
\DeclareSymbolFontAlphabet{\mathbbl}{bbold}

\usepackage{titlesec}
\titleformat*{\section}{\large\bfseries}

\newcommand{\N}{{\Bbb N}}
\newcommand{\Z}{{\Bbb Z}}
\newcommand{\Q}{{\Bbb Q}}

\newcommand{\C}{{\Bbb C}}

\renewcommand{\S}{{\mathbb S}}
\newcommand{\1}{{\mathbbm 1}}
\newcommand{\Ker}{{\mathrm{Ker}}}

\newcommand{\cB}{{\mathfrak B}}

\newcommand{\cO}{{\cal O}}

\newcommand{\ol}{\overline}
\newcommand{\ul}{\underline}
\newcommand{\wh}{\widehat}
\newcommand{\wt}{\widetilde}

\title{$q$-de Rham complexes of higher level}
\author{Kimihiko Li}
\date{}

\begin{document}

\maketitle

Address: The University of Tokyo, Tokyo, Japan.

kimihiko@g.ecc.u-tokyo.ac.jp 

\begin{abstract}
In this article, we construct two kinds of de Rham-like complexes which compute the cohomology of complete crystals on the higher-level $q$-crystalline site, which was introduced in a previous article of the author. One complex is the $q$-analog of the higher de Rham complex constructed by Miyatani, and another complex is the $q$-analog of the jet complex constructed by Le Stum-Quir\'os. The complexes we construct can also be regarded as the higher-level analogs of the $q$-de Rham complex. 
\end{abstract} 

\vspace{0.1in}

\vspace{0.1in}

\tableofcontents

\phantomsection
\section*{Introduction}
\addcontentsline{toc}{section}{Introduction}

This article is the continuation of \cite{Li23}. It is devoted to giving certain $q$-analogs of the de Rham complexes that compute the cohomology of complete crystals on the higher-level $q$-crystalline site.

We fix a prime $p$. In \cite{Ber74}, Berthelot constructed a fundamental $p$-adic cohomology theory called crystalline cohomology. It is defined for a scheme $X$ over a $p$-adic formal PD-scheme $S$, and when $X$ is embeddable into a smooth $p$-adic formal scheme over $S$, it is quasi-isomorphic to a certain de Rham complex by the Poincar\'e lemma. Recently, Bhatt and Scholze introduced a kind of generalization of crystalline cohomology, called $q$-crystalline cohomology. It is defined for a $p$-completely smooth affine formal scheme 
$X$ over $R/I_R$, where $(R,I_R)$ is a $q$-PD pair. 
When $X$ is embeddable smoothly over $R$ in a suitable sense, they constructed a certain $q$-de Rham complex which computes the $q$-crystalline cohomology. When $q=1$, these recover the classical crystalline theory. 

On the other hand, for a non-negative integer $m$, Berthelot also introduced crystalline cohomology of level $m$ in \cite{Ber96} to consider the crystalline cohomology theory in a more general situation. When $m=0$, it recovers the usual crystalline cohomology. In \cite{LQ01}, Le Stum and Quir\'{o}s constructed a certain complex called the jet complex, and showed that it computes the $m$-crystalline cohomology by considering the corresponding Poincar\'{e} lemma. However, the jet complex may not be bounded above in general, and their proof of the local freeness of each term is not correct. To overcome these difficulties, Miyatani constructed in \cite{Miy15} another complex, called the higher de Rham complex, and showed that it computes a direct sum of finitely many copies of $m$-crystalline cohomology by considering the corresponding Poincar\'{e} lemma. The higher de Rham complex is bounded and locally free, and these properties were used to prove the finiteness of $m$-crystalline cohomology. However, this complex is not accurate enough to compute only one copy of $m$-crystalline cohomology.

In a previous article \cite{Li23}, we constructed the level-$m$ $q$-crystalline site $(X/R)_{m\text{-}q\text{-crys}}$ for a $q$-PD pair $(R,I_R)$ and a smooth and separated $p$-adic formal scheme $X$ over $R/J_R$, where $J_R := (\phi^m)^{-1}(I_R)$ is defined by using the Frobenius lift $\phi$ of $R$. This can be considered as a generalization of both the $q$-crystalline site and the $m$-crystalline site. We also proved that a certain category of crystals on the $m$-$q$-crystalline site $(X/R)_{m\text{-}q\text{-crys}}$ is equivalent to that on the $q$-crystalline site $(X'/R)_{q\text{-crys}}$, where $X' := X \widehat{\times}_{\text{Spf}(R/J_R),(\phi^m)^*} \text{Spf}(R/I_R)$ is the pullback of $X$ by the $m$-fold iteration $\phi^m$ of the Frobenius lift.

The purposes of the present article are to develop the level-$m$ $q$-crystalline cohomology theory, and to construct the $q$-analogs of the complexes in \cite{Miy15} and \cite{LQ01} that compute the cohomology of complete crystals on the level-$m$ $q$-crystalline site. When $q = 1$, these recover the usual level-$m$ crystalline theory, and when $m = 0$, the $q$-analogs of these complexes coincide with the $q$-de Rham complex in \cite{BS22} in the situation we will consider.

Let us explain the content of each section. In Section \ref{qanbc}, we define the $q$-analog of the binomial coefficients of higher level (Definition \ref{qhbc}), which will frequently appear in the formulas of differential calculus related to the level-$m$ $q$-crystalline theory. When $q=1$, these recover the usual binomial coefficients of higher level in \cite{Ber96}.

In Section \ref{mqanc}, we first calculate a certain $q^{p^m}$-PD envelope (Theorem \ref{qpde}), which is crucial in the level-$m$ $q$-crystalline cohomology theory. We also construct a natural $q$-analog of the $m$-PD polynomial algebra, and show that in a certain situation, the $q^{p^m}$-PD envelope above coincides with this $m$-$q$-PD polynomial algebra (Proposition \ref{rg}), which is much better behaved. Based on these results, we develop the theory of the $m$-$q^{p^m}$-crystalline site $(\ol{A}/R)_{m\text{-}q^{p^m}\text{-crys}}$ for a morphism of $q^{p^m}$-PD pairs \linebreak $(R, I_R) \to (A, I_A)$ with fixed rank one \'{e}tale coordinates $\ul{x} = (x_1, \dots, x_d)$ in $A$ which satisfies  $
I_{A} = \ol{I_R A}^{\rm cl}$, where $\ol{A} := A/J_A$ and $\ol{I_R A}^{\rm cl}$ is the closure of $I_R A$ in $A$ for the $(p, (p)_{q^{p^m}})$-adic topology. We generally consider the complete crystals (Definition \ref{cpfcr}) developed by Tian \cite{Tia23} as coefficients. Then we can show the higher-level $q$-analogs of the results related to the linearizations and stratifications in the classical theory.

In Section \ref{hple}, we first do some differential calculus related to the $m$-$q$-PD polynomial algebra $\wh{A \langle \ul{\xi} \rangle}_{(m), q, \ul{x}}$. Next, we construct the higher $q$-de Rham complex (Definition \ref{mqdr}). Finally, we prove the Poincar\'{e} lemma for the higher $q$-de Rham complex (Corollary \ref{gspl}), which states that this complex computes a direct sum of $p^{md}$ copies of $m$-$q^{p^m}$-crystalline cohomology of complete crystals. These are essentially the $q$-analogs of the results in \cite{Miy15}. But we may need some non-trivial $q$-analogs of the constructions and calculations since the $q$-analogs of some usual calculations, such as the binomial formula, may not work in complete generality.

In Section \ref{qanj}, we first construct the $q$-jet complex (Definition \ref{mqjc}). Next, we construct the homotopy map $h$ on the $q$-jet complex (Proposition \ref{holo}), called the integration of differential forms. Finally, we use this homotopy map $h$ to show the Poincar\'{e} lemma for the $q$-jet complex (Corollary \ref{jcgs}), which states that this complex computes $m$-$q^{p^m}$-crystalline cohomology of complete crystals. These are essentially the $q$-analogs of the results in \cite{LQ01}. In contrast to the difficulty of the description of the homotopy map $h$, these $q$-analogs can be considered by simply replacing the basis of the $m$-PD polynomial algebra $\ul{\xi}^{{ \{ \ul{k} \} }_{(m)}}$ and the binomial coefficient of higher level $\left\langle
\begin{array}{@{}c@{}}
\ul{k} \\ \ul{k}'
\end{array}
\right\rangle_{(m)}$ with their $q$-analogs $\ul{\xi}^{ \{ \ul{k} \}_{(m), q}}$ and $ \left\langle
\begin{array}{@{}c@{}}
\ul{k} \\ \ul{k}'
\end{array}
\right\rangle_{(m), q }$ respectively. However, some of the cases in the proofs of \cite{LQ01} were not fully considered, so we give detailed proofs for the completeness of the present article.

Finally, the author would like to express his sincere gratitude to his supervisor \linebreak Atsushi Shiho who introduced him to the field of higher-level theory, patiently answered many questions and carefully read the drafts of this article. The author is also grateful to Bernard Le Stum for useful conversations and helpful advice.
The author was partially supported by WINGS-FMSP (World-leading Innovative Graduate Study for Frontiers of Mathematical Sciences and Physics) program at the Graduate School of Mathematical Sciences, the University of Tokyo.
This work was supported by the Grant-in-Aid for Scientific Research (KAKENHI No. 22J10387) and the Grant-in-Aid for JSPS DC2 fellowship.

Throughout this article, the set of natural numbers $\N$ means the set of non-negative integers.

\section{$q$-analog of the binomial coefficients of higher level}
\label{qanbc}


In this section, we define a $q$-analog of the binomial coefficients of higher level, following \cite{Ber96}. First, we review the usual binomial coefficients of higher level. Fix a non-negative integer $m$ in order to consider the level-$m$ theory. 

\begin{defi}
\label{hbc}
Let $k, k', k'' \in \N$  such that $k = k'+k''$, and let $r,r',r'',s,s',s'' \in \N$ such that
\begin{align*}
&k = p^mr+s \ \ \ \ \ \ (0 \le s < p^m), \\
&k' = p^mr'+s' \ \ \ \ (0 \le s' < p^m), \\
&k'' = p^mr''+s'' \ \ (0 \le s'' < p^m).
\end{align*}
We sometimes denote the integer $r$ by $\left\lfloor \frac{k}{p^m} \right\rfloor$ and call it the integer part of $\frac{k}{p^m}$.
Recall that the usual \textit{binomial coefficient} is defined as follows:
\[
\begin{pmatrix} k \\ k' \end{pmatrix} := \frac{k!}{k'!k''!}.
\]
Following \cite{Ber96}, we define the \textit{binomial coefficients of higher level} as follows:
\[
\begin{Bmatrix} k \\ k' \end{Bmatrix}_{(m)} := \frac{r!}{r'!r''!} \ \ \text{and} \ \ 
\left\langle
\begin{array}{@{}c@{}}
k \\
k'
\end{array}
\right\rangle_{(m)}
 := \begin{pmatrix} k \\ k' \end{pmatrix} \begin{Bmatrix} k \\ k' \end{Bmatrix} ^{-1}_{(m)} .
\]
\end{defi}

We have the following lemma.

\begin{lem}
\label{hb}
{\rm (\cite[Lemme 1.1.3]{Ber96})} For all $m,k,k' \in \N$ such that $k' \le k$, we have
\[
\begin{Bmatrix} k \\ k' \end{Bmatrix}_{(m)} \in \N, \ \ \ 
\left\langle
\begin{array}{@{}c@{}}
k \\
k'
\end{array}
\right\rangle_{(m)}
\in \Z_{(p)}.
\]
\end{lem}

We also use the usual conventions on multi-indices: for a non-negative integer $d$ and for $\ul{k} = (k_1, \dots, k_d)$ and $\ul{k}' = (k'_1, \dots, k'_d)$ in $\N^d$ satisfying $\ul{k}' \le \ul{k}$, i.e., $k'_i \le k_i$ for all $i = 1, \dots, d$, we define
\[
\begin{pmatrix} \ul{k} \\ \ul{k}' \end{pmatrix} := \prod_{i=1}^{d} \begin{pmatrix} k_i \\ k'_i \end{pmatrix}, \  \begin{Bmatrix} \ul{k} \\ \ul{k}' \end{Bmatrix}_{(m)} := \prod_{i=1}^{d}  \begin{Bmatrix} k_i \\ k'_i \end{Bmatrix}_{(m)} \ \ \text{and} \ \ 
\left\langle
\begin{array}{@{}c@{}}
\ul{k} \\ \ul{k}'
\end{array}
\right\rangle_{(m)}
 := \prod_{i=1}^{d}  \left\langle
\begin{array}{@{}c@{}}
k_i \\ k'_i
\end{array}
\right\rangle_{(m)}.
\]

On the other hand, we may consider various $q$-analogs of the usual concepts on integers, such as $q$-integers and $q$-binomial coefficients, which recover the usual concepts when $q = 1$. A useful reference on this subject is \cite{LQ15}.

We fix an associative ring $R$ and an element $q$ in $R$. If $n \in \N$, the $q$-analog of $n$ is
\[
(n)_{q} := 1+ q + \dots + q^{n-1}.
\]
We will also say that $(n)_q$ is a \textit{$q$-integer} of $R$. When $R = \Z$ is the ring of integers and $q=1$, we recover the usual natural number $n$. 

The following formula is frequently used:

\begin{prop}
\label{nnp}
For all $n,n' \in \N$, we have
\[
(nn')_{q} = (n)_{q}(n')_{q^n}.
\]
\end{prop}
\begin{proof}
We can check it directly from the definition:
\[
\sum_{j=0}^{nn'-1} q^j = \left( \sum_{j'=0}^{n-1} q^{j'} \right) \left( \sum_{j''=0}^{n'-1} q^{nj''} \right) . \qedhere
\] 
\end{proof}

Note that in general, for $j \in \N$, we can consider the $q^{p^j}$-integer $(n)_{q^{p^j}}$ (that is, the $q$-integer when the element $q$ is $q^{p^j}$). Then the following lemma is frequently used:

\begin{lem}
\label{gcdj}
Let $n, j \in \N$ such that {\rm gcd}$(n,p) = 1$, then $(n)_{q^{p^j}}$ is invertible in $\Z [ q ]_{(p,q-1)}$. 
\end{lem}
\begin{proof}
We always have $ (n)_{q^{p^j}} \equiv n$ modulo $q-1$. Since $n$ is invertible in $\Z_{(p)}$, the result follows.
\end{proof}

By using $q$-integers, we can define the \textit{$q$-analog of the factorial} of $n \in \N$:
\[
(n)_{q}! := \prod_{j=0}^{n-1} (n-j)_{q}.
\]

The $q$-analog of the binomial coefficients can be defined inductively as follows:

\begin{defi}
\label{qbco}
The \textit{$q$-binomial coefficients} are defined by induction for $k,k' \in \N$ via Pascal identities
\[
\begin{pmatrix} k \\ k' \end{pmatrix}_{q} = \begin{pmatrix} k-1 \\ k'-1 \end{pmatrix}_{q} + q^{k'} \begin{pmatrix} k-1 \\ k' \end{pmatrix}_{q}
\]
with
\[
\begin{pmatrix} 0 \\ k' \end{pmatrix}_{q} = 
\begin{cases}
1 \ \ \ \text{if} \ k'=0, \\
0 \ \ \ \text{otherwise}.
\end{cases}
\]
\end{defi}

Note that if $R = \Z[q]$, we see that all $q$-binomial coefficients belong to $\Z [ q ]$ by induction. If $R = \Q ( q )$, then we have the following $q$-analog of the familiar expression, which is frequently used in the computation of $q$-binomial coefficients:

\begin{prop}
\label{qiff}
{\rm (\cite[Proposition 2.6]{LQ15})} If $R = \Q ( q )$, then for all $k,k' \in \N $ such that $k' \le k$,
\[
\begin{pmatrix} k \\ k' \end{pmatrix}_{q} = \frac{(k)_{q}!}{(k')_{q}!(k-k')_{q}!}.
\]
\end{prop}

Now we want to define the $q$-analogs of 
\[
\begin{Bmatrix} k \\ k' \end{Bmatrix}_{(m)} \ \text{and} \ \left\langle
\begin{array}{@{}c@{}}
k \\
k'
\end{array}
\right\rangle_{(m)}.
\]
By taking into account Definition \ref{hbc} and Proposition \ref{nnp}, we may make the following definition:

\begin{defi}
\label{qhbc}
Let $k, k', k'' \in \N$ such that $k = k'+k''$, and let $r,r',r'',s,s',s'' \in \N$ such that
\begin{align*}
&k = p^mr+s \ \ \ \ \ \ (0 \le s < p^m), \\
&k' = p^mr'+s' \ \ \ \ (0 \le s' < p^m), \\
&k'' = p^mr''+s'' \ \ (0 \le s'' < p^m).
\end{align*}
We assume that $R = \Q ( q )$. Then we define the \textit{$q$-analog of  the binomial coefficients of higher level} as follows:
\[
\begin{Bmatrix} k \\ k' \end{Bmatrix}_{(m),q} := \frac{(r)_{q^{p^m}}!}{(r')_{q^{p^m}}!(r'')_{q^{p^m}}!} \ \ \text{and} \ \ 
\left\langle
\begin{array}{@{}c@{}}
k \\
k'
\end{array}
\right\rangle_{(m),q}
 := \begin{pmatrix} k \\ k' \end{pmatrix}_{q} \begin{Bmatrix} k \\ k' \end{Bmatrix} ^{-1}_{(m),q}.
\]
We can also consider the multi-indices case: if $\ul{k} = (k_1, \dots, k_d)$ and $\ul{k}' = (k'_1, \dots, k'_d)$ in $\N^d$ satisfy $\ul{k}' \le \ul{k}$, i.e., if $k'_i \le k_i$ for all $i = 1, \dots, d$, then we define
\[
\begin{pmatrix} \ul{k} \\ \ul{k}' \end{pmatrix}_{ q } := \prod_{i=1}^{d} \begin{pmatrix} k_i \\ k'_i \end{pmatrix}_{q}, \  \begin{Bmatrix} \ul{k} \\ \ul{k}' \end{Bmatrix}_{(m), q } := \prod_{i=1}^{d}  \begin{Bmatrix} k_i \\ k'_i \end{Bmatrix}_{(m),q} \ \ \text{and} \ \ 
\left\langle
\begin{array}{@{}c@{}}
\ul{k} \\ \ul{k}'
\end{array}
\right\rangle_{(m), q }
 := \prod_{i=1}^{d}  \left\langle
\begin{array}{@{}c@{}}
k_i \\ k'_i
\end{array}
\right\rangle_{(m), q}.
\]

\end{defi}

Then we have the following lemma, which is the $q$-analog of Lemma \ref{hb}:

\begin{lem}
\label{qhb}
Let $m \in \N$, and let $\ul{k},\ul{k}' \in \N^d$ such that $\ul{k}' \le \ul{k}$. Then we have
\[
\begin{Bmatrix} \ul{k} \\ \ul{k}' \end{Bmatrix}_{(m), q } \in \Z [ q ], \ \ \ 
\left\langle
\begin{array}{@{}c@{}}
\ul{k} \\ \ul{k}'
\end{array}
\right\rangle_{(m), q }
\in \Z [ q ]_{(p,q-1)}.
\]
In particular, $\begin{Bmatrix} \ul{k} \\ \ul{k}' \end{Bmatrix}_{(m), q }$ and $\left\langle
\begin{array}{@{}c@{}}
\ul{k} \\ \ul{k}'
\end{array}
\right\rangle_{(m), q }$ are well-defined as elements of any $\Z [ q ]_{(p,q-1)}$-algebra $R$.
\end{lem}

\begin{proof}
Since the multi-indices case of the binomial coefficients is defined as the products of the one-index case, it is sufficient to show the assertion for  non-negative integers $k, k'$. We follow the notations in Definition \ref{qhbc}, so let $k''$ be the non-negative integer with $k = k'+k''$. Then there are two cases we need to consider: the case $r= r' + r''$ (where $s= s' + s''$), and the case $r= r' + r'' + 1$ (where $p^m + s = s' + s''$).

First we consider the element $\begin{Bmatrix} \ul{k} \\ \ul{k}' \end{Bmatrix}_{(m), q }$. In the case $r= r' + r''$, we see from the definition that
\[
\begin{Bmatrix} k \\ k' \end{Bmatrix}_{(m), q } = \begin{pmatrix} r \\ r' \end{pmatrix}_{q^{p^m}}.
\]
In the case $r= r' + r'' + 1$, we can check that
\[
\begin{Bmatrix} k \\ k' \end{Bmatrix}_{(m), q } = \frac{(r' + r'')_{q^{p^m}}!(r' + r''+1)_{q^{p^m}}}{(r')_{q^{p^m}}!(r'')_{q^{p^m}}!} = \begin{pmatrix} r' + r'' \\ r' \end{pmatrix}_{q^{p^m}} \cdot (r' + r''+1)_{q^{p^m}}.
\]
It follows by induction that all the $q^{p^m}$-binomial coefficients as in Definition \ref{qbco} belong to $\Z [ q^{p^m} ] \subset \Z [ q ]$. So we see that
\[
\begin{Bmatrix} k \\ k' \end{Bmatrix}_{(m), q } \in \Z [ q ].
\]

To prove the rest, thanks to Lemma \ref{gcdj}, we only need to consider the terms $(pl)_{q^{p^j}}$ in the definition of $\left\langle
\begin{array}{@{}c@{}}
\ul{k} \\ \ul{k}'
\end{array}
\right\rangle_{(m), q }$, where $l$ is a non-negative integer. For the calculations below, the next claim is useful:

\begin{clm}
\label{clmus}
Let $r \in \N$, and let $s$ be an integer such that $0 < s < p^m$. Then, in the ring $\Z [ q ]_{(p,q-1)}$, we have
\[
(p^m r + s)_q = u \cdot (s)_q
\]
for a unit $u \in \Z [ q ]_{(p,q-1)}$.
\end{clm}

We prove the claim. Let $s = p^n l$, where $l,n$ are non-negative integers and gcd$(l,p) = 1$. Then by Proposition \ref{nnp}, we have:
\begin{align*}
(p^m r + s)_q &= (p^n(p^{m-n} r + l))_q = (p^n)_q(p^{m-n} r + l)_{q^{p^n}}, \\
(s)_q &= (p^n l)_q = (p^n)_q (l)_{q^{p^n}}.
\end{align*}
As $m-n > 0$, we see that gcd$(p^{m-n} r + l,p) = 1$. So  $(p^{m-n} r + l)_{q^{p^n}}$ is a unit. As $(l)_{q^{p^n}}$ is also a unit by Lemma \ref{gcdj}, the claim follows.

Now we consider the element $\left\langle
\begin{array}{@{}c@{}}
\ul{k} \\ \ul{k}'
\end{array}
\right\rangle_{(m), q }$. We first consider the case $r= r' + r''$. By Proposition \ref{qiff}, we can check that 
\begin{align*}
\begin{pmatrix} k \\ k' \end{pmatrix}_{q} &= \frac{(p^m r + s)_{q}!}{(p^m r' + s')_{q}!(p^m r'' + s'')_{q}!}  \\
&= \frac{(p^m r)_{q}!}{(p^m r')_{q}!(p^m r'')_{q}!} \cdot \frac{\prod_{j=1}^s(p^m r+j)_{q}}{\prod_{j'=1}^{s'}(p^m r'+j')_{q}\prod_{j''=1}^{s''}(p^m r''+j'')_{q}}.
\end{align*}
Since $0 < j,j',j'' < p^m$, by using Claim \ref{clmus}, we see that up to a unit in $\Z [ q ]_{(p,q-1)}$, the second term of the right hand side is equal to
\[
\frac{\prod_{j=1}^s(j)_{q}}{\prod_{j'=1}^{s'}(j')_{q}\prod_{j''=1}^{s''}(j'')_{q}} = \begin{pmatrix} s \\ s' \end{pmatrix}_{q} \in \Z [ q ].
\]
So it suffices to consider the first term of the right hand side, and we are reduced to proving the statement for $k = p^mr, k' = p^mr', k'' = p^mr''$. Then we have
\[
\begin{Bmatrix} k \\ k' \end{Bmatrix}_{(m),q} = \frac{(r)_{q^{p^m}}!}{(r')_{q^{p^m}}!(r'')_{q^{p^m}}!}.
\]
On the other hand, by Proposition \ref{nnp}, for all non-negative integer $l$, we have
\[
(p^m l)_q = (p^m)_q (l)_{q^{p^m}}.
\]
So, by cancelling the terms $(l)_{q^{p^m}}$ in $\begin{Bmatrix} k \\ k' \end{Bmatrix}_{(m),q}$ with the terms $(p^m l)_q$ in $\begin{pmatrix} k \\ k' \end{pmatrix}_{q}$, we have
{\fontsize{9pt}{10pt}\selectfont
\begin{align*}
\left\langle
\begin{array}{@{}c@{}}
k \\ k'
\end{array}
\right\rangle_{(m), q } &= \frac{\prod_{l=0}^{r-1} \left( (p^m)_q \cdot \prod_{j=1}^{p^m-1}(p^m l+j)_{q} \right) }{\left( \prod_{l'=0}^{r'-1} \left( (p^m)_q \cdot \prod_{j'=1}^{p^m-1}(p^m l'+j')_{q} \right) \right) \left( \prod_{l''=0}^{r''-1} \left( (p^m)_q \cdot \prod_{j''=1}^{p^m-1}(p^m l''+j'')_{q} \right) \right) } \\
&= \frac{((p^m)_q)^r \prod_{l=0}^{r-1} \prod_{j=1}^{p^m-1}(p^m l+j)_{q}}{((p^m)_q)^{r'+r''}\left( \prod_{l'=0}^{r'-1}  \prod_{j'=1}^{p^m-1}(p^m l'+j')_{q} \right) \left( \prod_{l''=0}^{r''-1} \prod_{j''=1}^{p^m-1}(p^m l''+j'')_{q} \right) } \\
&= \frac{\prod_{l=0}^{r-1} \prod_{j=1}^{p^m-1}(p^m l+j)_{q}}{\left( \prod_{l'=0}^{r'-1}  \prod_{j'=1}^{p^m-1}(p^m l'+j')_{q} \right) \left( \prod_{l''=0}^{r''-1} \prod_{j''=1}^{p^m-1}(p^m l''+j'')_{q} \right) }.
\end{align*}}

\noindent
By using Claim \ref{clmus} again, we see that up to a unit in $\Z [ q ]_{(p,q-1)}$, this is equal to
\[
\frac{ \left( \prod_{j=1}^{p^m-1}(j)_{q} \right)^r }{\left(  \prod_{j'=1}^{p^m-1}(j')_{q} \right)^{r'} \left(  \prod_{j''=1}^{p^m-1}(j'')_{q} \right)^{r''} } = 1.
\]

It remains to consider the case $r= r' + r'' + 1$. By Proposition \ref{qiff}, we can check that 
{ \allowdisplaybreaks
\begin{align*}
\begin{pmatrix} k \\ k' \end{pmatrix}_{q} &= \frac{(p^m r + s)_{q}!}{(p^m r' + s')_{q}!(p^m r'' + s'')_{q}!}  \\
&= \frac{(p^m (r'+r''))_{q}!}{(p^m r')_{q}!(p^m r'')_{q}!} \cdot \frac{\prod_{j=1}^{p^m+s}(p^m (r'+r'')+j)_{q}}{\left( \prod_{j'=1}^{s'}(p^m r'+j')_{q} \right) \left( \prod_{j''=1}^{s''}(p^m r''+j'')_{q} \right) }  \\
&= \begin{pmatrix} p^m (r'+r'') \\ p^m r' \end{pmatrix}_{q} \cdot \frac{\prod_{j=1}^{p^m+s}(p^m (r'+r'')+j)_{q}}{\left( \prod_{j'=1}^{s'}(p^m r'+j')_{q} \right) \left( \prod_{j''=1}^{s''}(p^m r''+j'')_{q} \right) }.
\end{align*}
We also have
\begin{align*}
\begin{Bmatrix} k \\ k' \end{Bmatrix}_{(m),q} &= \frac{(r'+r'')_{q^{p^m}}!}{(r')_{q^{p^m}}!(r'')_{q^{p^m}}!} \cdot (r'+r''+1)_{q^{p^m}} \\
&= \begin{Bmatrix} p^m (r'+r'') \\ p^m r' \end{Bmatrix}_{(m),q} \cdot (r'+r''+1)_{q^{p^m}}.
\end{align*}}
So, the element $\left\langle
\begin{array}{@{}c@{}}
k \\ k'
\end{array}
\right\rangle_{(m), q }$ is equal to
{\fontsize{9pt}{10pt}\selectfont
\begin{align*}
\left\langle
\begin{array}{@{}c@{}}
p^m (r'+r'') \\ p^m r'
\end{array}
\right\rangle_{(m), q } \cdot \frac{\left( \prod_{j=1}^{p^m-1}(p^m (r'+r'')+j)_{q} \right) \left( \prod_{j'''=p^m+1}^{p^m+s}(p^m (r'+r'')+j''')_{q} \right)}{\left( \prod_{j'=1}^{s'}(p^m r'+j')_{q} \right) \left( \prod_{j''=1}^{s''}(p^m r''+j'')_{q} \right) }  \cdot \frac{(p^m (r'+r''+1))_{q}}{(r'+r''+1)_{q^{p^m}}}.
\end{align*}}

\noindent
The first term belongs to $\Z [ q ]_{(p,q-1)}$ by the case $r= r' + r''$. For the third term, by Proposition \ref{nnp}, we see that 
\[
(p^m (r'+r''+1))_{q} = (p^m)_{q}(r'+r''+1)_{q^{p^m}}.
\]
So the third term is equal to $(p^m)_{q}$. Now, we consider the second term. By using Claim \ref{clmus}, we see that up to a unit in $\Z [ q ]_{(p,q-1)}$, the second term is equal to
\[
\frac{\left( \prod_{j=1}^{p^m-1}(j)_{q} \right) \left( \prod_{j'''=p^m+1}^{p^m+s}(j''')_{q} \right)}{\left( \prod_{j'=1}^{s'}(j')_{q} \right) \left( \prod_{j''=1}^{s''}(j'')_{q} \right) }.
\] After multiplying the third term $(p^m)_{q}$, we get $\begin{pmatrix} p^m+s \\ s' \end{pmatrix}_{q} \in \Z [ q ]$, so the result follows.
\end{proof}


\section{$q$-analog of the crystalline site of higher level}
\label{mqanc}

A $q$-analog of the crystalline cohomology theory called the $q$-crystalline cohomology theory was introduced by Bhatt and Scholze \cite{BS22}. On the other hand, the higher-level crystalline cohomology theory was introduced by Berthelot in \cite{Ber96} and it was developed by Le Stum-Quir\'{o}s and Miyatani in \cite{LQ01}, \cite{Miy15}. The author constructed in \cite{Li23} the higher-level $q$-crystalline cohomology theory which was compatible with these two theories.

In this section, we introduce some more tools about higher-level $q$-crystalline theory, which will later be used in order to construct the complexes that compute the cohomology.


First, we review the theory of twisted divided polynomials. Our setting is based on that of Chapter 3 in \cite{Hou23}, which is the higher dimensional version of \cite{GLQ23b}. As before, fix a non-negative integer $d$. Let $R$ be a $\delta$-$\Z [ q ]_{(p,q-1)}$-algebra, where $q$ is seen as a rank one element of $R$, i.e., $\delta(q) = 0$. We endow all modules with their $(p,q-1)$-adic topology. When we consider (classical or derived) completion, it will always be with respect to the $(p,q-1)$-adic topology. Let $A$ be a classically $(p,q-1)$-complete $R$-algebra (in particular, $A$ is derived $(p,q-1)$-complete) with fixed  \'{e}tale coordinates $\ul{x} = (x_1, \dots, x_d)$ (i.e., the $x_i$'s are the images of the indeterminates $X_i$'s under some \'{e}tale map $R[ \ul{X} ] \to A$). By Lemma 2.18 of \cite{BS22}, there exists a unique structure of $\delta$-$R$-algebra on $A$ such that $x_1, \dots, x_d$ are rank one elements. 

Note that, as a $\delta$-ring, $R$ is endowed with a Frobenius lift $\phi$ which satisfies \linebreak $\phi(q) = q^p$.  As a $\delta$-$R$-algebra, $A$ is endowed with a Frobenius lift $\phi$ which is semilinear with respect to the Frobenius lift $\phi$ of $R$, and satisfies $\phi(x_i) = x_i^p$ for $i = 1, \dots , d$. 

To consider more general $\delta$-$R$-algebras than $A$ above, let $B$ be a $\delta$-$R$-algebra endowed with a sequence of elements $\ul{x} = (x_1, \dots, x_d)$ which are not necessarily \'{e}tale coordinates nor rank one elements.
We will make several definitions for this general $\delta$-$R$-algebra $B$, which have additional explicit descriptions and structures when the $x_i$'s are all rank one elements (in particular when there exists a morphism of $\delta$-$R$-algebras $A \to B$ such that the $x_i$'s in $B$ are the images of the $x_i$'s in $A$).

We consider the polynomial ring $B[ \ul{\xi} ] := B[ \xi_1, \dots, \xi_d ]$ endowed with the unique  $\delta$-structure over $B$ such that, for $i = 1, \dots , d$, the element $x_i + \xi_i$ has rank one. We call this $\delta$-structure  the \textit{symmetric $\delta$-structure}. When the $x_i$'s in $B$ are all rank one elements, the $\delta$-structure is given by:
\begin{equation}
\delta(\xi_i) = \sum_{j=1}^{p-1} \frac{1}{p} \begin{pmatrix} p \\ j \end{pmatrix} x_i^{p-j} \xi_i^j.
\end{equation}
This corresponds to the following Frobenius lift:
\begin{equation}
\label{lof}
\phi(\xi_i) = (\xi_i + x_i)^p - x_i^p.
\end{equation}




Recall that in Section 1 of \cite{GLQ22b}, Gros-Le Stum-Quir\'{o}s introduced the notion of \textit{twisted powers} in one dimension. We can naturally extend it to the higher dimensional case. First, for $\ul{y} = (y_1, \dots, y_d) \in B^d$ and $k \in \N$, we set
\[
\xi_i^{(k)_{q, y_i}} := \prod_{j=0}^{k-1} (\xi_i + (j)_{q} y_i) \in B[ \ul{\xi} ].
\]
Then we can consider the multi-indices version of the twisted powers: for $\ul{k} \in \N^d$, we set 
\[
\ul{\xi}^{(\ul{k})_{q, \ul{y}}} := \prod_{i=1}^d \xi_i^{(k_i)_{q, y_i}}.
\]
Following Section 2 of \cite{GLQ23b}, if $y_i := (1-q) x_i$ for $i = 1, \dots , d$, then we drop the indices $y_i, \ul{y}$ and simply denote them by $\xi_i^{(k)_{q}},  \ul{\xi}^{(\ul{k})_{q}}$ respectively:
{\fontsize{11pt}{10pt}\selectfont
\begin{align*}
\xi_i^{(k)_{q}} := \xi_i^{(k)_{q, (1-q) x_i}} := \prod_{j=0}^{k-1} (\xi_i + (j)_{q} (1-q) x_i) &= \prod_{j=0}^{k-1} (\xi_i + (1-q^j) x_i) = \prod_{j=0}^{k-1} (\xi_i + x_i -q^j x_i), \\
\ul{\xi}^{(\ul{k})_{q}} &:= \prod_{i=1}^d \xi_i^{(k_i)_{q}}.
\end{align*}}

\noindent
Since each $\xi_i^{(k_i)_{q}}$ is a monic polynomial in $\xi_i$, we see that the $\ul{\xi}^{(\ul{k})_{q}}$'s for $\ul{k} \in \N^d$ form a basis of the $B$-module $B[ \ul{\xi} ]$. We will need later the following multiplication rule with respect to this basis, which was shown in Lemma 1.2 of \cite{GLQ22b}:
\begin{equation}
\label{mrul}
\xi_i^{(k)_{q}}\xi_i^{(k')_{q}} = \sum_{0 \le j \le k,k'} (j)_q! q^{\frac{j(j-1)}{2}}\begin{pmatrix} k \\ j \end{pmatrix}_{q} \begin{pmatrix} k' \\ j \end{pmatrix}_{q} (q-1)^j x_i^j \xi_i^{(k+k'-j)_{q}}.
\end{equation}

In Section 2 of \cite{GLQ22b}, Gros-Le Stum-Quir\'{o}s introduced the ring of \textit{twisted divided polynomials} in one dimension. We define the higher dimensional version $B \langle \ul{\xi} \rangle_{q, \ul{x}}$ of it as follows: as a $B$-module, it is free on abstract generators $\ul{\xi}^{[\ul{k}]_{q}}$ indexed by $\ul{k} \in \N^d$. We call the $\ul{\xi}^{[\ul{k}]_{q}}$'s the \textit{twisted divided powers}. By Proposition 2.1 and 2.2 of \cite{GLQ22b}, we can endow $B \langle \ul{\xi} \rangle_{q, \ul{x}}$ with a ring structure using the multiplication rule
\begin{equation}
\xi_i^{[k]_{q}}\xi_i^{[k']_{q}} = \sum_{0 \le j \le k,k'} q^{\frac{j(j-1)}{2}}\begin{pmatrix} k+k'-j \\ k \end{pmatrix}_{q} \begin{pmatrix} k \\ j \end{pmatrix}_{q} (q-1)^j x_i^j \xi_i^{[k+k'-j]_{q}}.
\end{equation}
Using (\ref{mrul}), we can check that the $B$-linear map
\begin{align}
\label{tdpb}
B[ \ul{\xi} ] &\to B \langle \ul{\xi} \rangle_{q, \ul{x}} \\
\xi_i^{(k)_{q}} &\mapsto (k)_{q}! \xi_i^{[k]_{q}} \notag
\end{align}
is a morphism of $B$-algebras. We see from this map that the $\ul{\xi}^{[\ul{k}]_{q}}$'s recover the usual divided powers when $q = 1$. 
We also denote by $I^{[1]}$ the augmentation ideal generated by all $\ul{\xi}^{[\ul{k}]_{q}}$ with  $| \ul{k} | := k_1+ \cdots+ k_d  \geq 1$. 

When the $x_i$'s in $B$ are all rank one elements, we can endow $B \langle \ul{\xi} \rangle_{q, \ul{x}}$ with the unique $\delta$-structure such that (\ref{tdpb}) is a morphism of $\delta$-$B$-algebras: we can define the Frobenius lift $\phi$ on $B \langle \ul{\xi} \rangle_{q, \ul{x}}$ by using Definition 7.10 of \cite{GLQ22b} and the arguments in Section 2 of \cite{GLQ23b} for several variables. When $B$ is $p$-torsion-free (in particular when $R = \Z [ q ]_{(p,q-1)}$ and $B = \Z [ q ]_{(p,q-1)} [\ul{x}]$),  this corresponds to a unique $\delta$-structure on $B \langle \ul{\xi} \rangle_{q, \ul{x}}$. In general, we can use the $\delta$-structure on $\Z [ q ]_{(p,q-1)} [\ul{x}] \langle \ul{\xi} \rangle_{q, \ul{x}}$ to define the $\delta$-structure on $B \langle \ul{\xi} \rangle_{q, \ul{x}}$ by using the isomorphism
\[
B \langle \ul{\xi} \rangle_{q, \ul{x}} \simeq B \otimes_{\Z [ q ]_{(p,q-1)} [\ul{x}]} \Z [ q ]_{(p,q-1)} [\ul{x}] \langle \ul{\xi} \rangle_{q, \ul{x}},
\]
where the $\delta$-structure on the right hand side is defined by using Remark 2.7 of \cite{BS22}.

Next, we recall the notion of $q$-PD pair. Our setting is based on that of  Section 7 in \cite{GLQ23b}.

\begin{defi}
\label{dfqpd}
A \textit{$q$-PD pair} is a derived $(p,(p)_q)$-complete $\delta$-pair $(D,I_D)$ consisting of a $\delta$-$R$-algebra $D$ and an ideal $I_D$ which satisfy the following conditions:
\begin{enumerate}
\item For any $f \in I_D$, $\phi (f) - (p)_q \delta(f) \in (p)_qI_D$.
\item The ring $D$ is bounded, i.e., $D$ is $(p)_q$-torsion-free and $D/((p)_q)$ has bounded $p^\infty$-torsion.
\item $D/I_D$ is classically $p$-complete.
\end{enumerate}
By condition 1 and the $(p)_q$-torsion-freeness of $D$ in condition 2, we can introduce the map
\begin{align*}
\gamma: I_D &\to D  \\
f &\mapsto \frac{\phi(f)}{(p)_q} - \delta(f).
\end{align*}
Then, condition 1 means that $\gamma (I_D) \subset I_D$.
\end{defi}

\begin{rem}
\label{rqpd}
\begin{enumerate}
\item For $m \in \N$, we have the congruence:
\[
(p)_{q^{p^m}} \equiv p \  \text{mod} \  q-1.
\]
We can also check that
\[
(p)_{q^{p^m}}(q-1)^{p^m} \equiv (p)_{q^{p^m}}(q^{p^m}-1) \equiv q^{p^{m+1}}-1 \equiv (q-1)^{p^{m+1}} \ \text{mod} \ p,
\]
so we have the following congruence:
\[
(p)_{q^{p^m}} \equiv (q-1)^{p^{m}(p-1)} \ \text{mod} \ p.
\]
It follows that $(p)_{q^{p^m}} \in (p, q-1)$ and $(q-1)^{p^{m}(p-1)} \in (p, (p)_{q^{p^m}})$. Therefore, the $(p, q-1)$-adic topology, the $(p, (p)_q)$-adic topology, the $(p, (p)_{q^{p^m}})$-adic topology and the $(p, q^{p^m}-1)$-adic topology coincide for $m \in \N$.
\item By Lemma 3.7 of \cite{BS22}, a $q$-PD pair $(D,I_D)$ is actually classically $(p, (p)_q)$-complete. So, by Remark \ref{rqpd}.1, we see that the completeness of the $q$-PD pair $(D,I_D)$ coincides with the classical $(p, q-1)$-adic completeness assumption on $A$ we made at the beginning of this section.
\end{enumerate}
\end{rem}

For any $\delta$-ring $C$, we can consider the $\delta$-envelope of a $C$-algebra (see Definition 1.1 in \cite{GLQ23b}). In this article, we mainly consider the following two cases: one is the $\delta$-envelope $C[\ul{X}]^{\delta}$ of the polynomial ring $C[\ul{X}]$, which can be identified with the free $\delta$-ring on generators $X_1, \dots, X_d$. For the elements $\ul{f}= (f_1, \dots, f_d) \in C^d$ and $g \in C$, we can also consider the $\delta$-envelope $C[\ul{f}/g]^{\delta}$ of the $C$-algebra $C[\ul{f}/g]$. This can be identified with $C \otimes_{C[\ul{X}]^{\delta}} C[\ul{Y}]^{\delta}$, 
where the structural maps on the left and on the right are defined by $X_i \mapsto f_i,$ $X_i \mapsto gY_i$  respectively. 

We can also consider the $q$-PD envelope (for more details, see Lemma 16.10 in \cite{BS22}). 

\begin{defi}
\label{dfqpd2}
{\rm (cf. \cite[Definition 3.3, Definition 4.8]{GLQ23b})}
Let $(C, I_C)$ be a $\delta$-pair. Then (if it exists) its $q$-PD envelope $(C^{[ \ ]_q}, I_C^{[ \ ]_q})$ is a $q$-PD pair that is universal for morphisms to  $q$-PD pairs: there exists a morphism of $\delta$-pairs $(C, I_C) \to (C^{[ \ ]_q}, I_C^{[ \ ]_q})$ such that any morphism $(C, I_C) \to (C', I_{C'})$ to a $q$-PD pair extends uniquely to $(C^{[ \ ]_q}, I_C^{[ \ ]_q})$.
\end{defi}

Now, we can give an explicit description of the $q$-PD envelope in a specific case, which will be useful later. 

\begin{thm}
\label{qpdb}
{\rm (cf. \cite[Theorem 3.5, Proposition 4.9]{GLQ23b})} If $B$ is a bounded $\delta$-$R$-algebra with elements $\ul{x} = (x_1, \dots, x_d) \in B^d$ and $B [ \ul{\xi}]$ is endowed with the symmetric $\delta$-structure, then $\left(B [\ul{\xi}] \left[ \frac{\phi(\ul{\xi})}{(p)_q} \right]^{\delta, \wedge}, K\right)$ is the $q$-PD envelope of $(B [ \ul{\xi}], (\xi_1, \dots, \xi_d))$, where $K$ is the minimal $(p, (p)_q)$-complete ideal containing $(\xi_1, \dots, \xi_d)$ and stable under the operation $\gamma$. Moreover, if the $x_i$'s in $B$ are all rank one elements, then we have an isomorphism of $\delta$-$R$-algebras
\[
\wh{B \langle \ul{\xi} \rangle}_{q, \ul{x}} \simeq B [\ul{\xi}] \left[\frac{\phi(\ul{\xi})}{(p)_q} \right]^{\delta, \wedge}
\]
so that $(\wh{B \langle \ul{\xi} \rangle}_{q, \ul{x}}, \ol{I^{[1]}}^{\rm cl})$ is also the $q$-PD envelope of $(B [ \ul{\xi}], (\xi_1, \dots, \xi_d))$, where $\ol{I^{[1]}}^{\rm cl}$ denotes the closure of $I^{[1]} \wh{B \langle \ul{\xi} \rangle}_{q, \ul{x}}$ in $\wh{B \langle \ul{\xi} \rangle}_{q, \ul{x}}$ for the $(p,(p)_{q})$-adic topology. 
\end{thm}
\begin{proof}
The first assertion follows from the assumptions and the universal properties of $\left(B [\ul{\xi}] \left[\frac{\phi(\ul{\xi})}{(p)_q} \right]^{\delta, \wedge}, K\right)$. The second assertion follows from Th\'{e}or\`{e}me 3.3.5 of \cite{Hou23}.
\end{proof}

Let $B$ be a bounded $\delta$-$R$-algebra and let $A \to B$ be a morphism of $R$-algebras (not necessarily a morphism of $\delta$-rings). By abuse of notation, we shall denote the images of the $x_i$'s in $B$ by the same symbols. Note that the $x_i$'s in $B$ are not necessarily rank one elements. There is a unique $\delta$-structure on the ring $B \wh{\otimes}_R A$ (where $\wh{\otimes}$ here means the classical $(p, q-1)$-adic completion) compatible with the $\delta$-structure on $B$ and that on $A$ by Remark 2.7 and Lemma 2.17 of \cite{BS22}. We can consider a morphism of $\delta$-$B$-algebras
\begin{align*}
B[ \ul{\xi} ] &\to B \wh{\otimes}_R A  \\
\xi_i &\mapsto 1 \otimes x_i - x_i \otimes 1.
\end{align*}
Let $J$ be the kernel of multiplication $B \wh{\otimes}_R A \twoheadrightarrow B$. Then we have the following lemma:

\begin{lem}
\label{qpdc}
{\rm (cf. \cite[Lemma 7.2]{GLQ23b})} If two morphisms of $\delta$-pairs \[ u_1, u_2 : (B \wh{\otimes}_R A, J) \to (D, I_D) \] to a $q$-PD pair coincide when restricted to $B[ \ul{\xi} ]$, then they are equal.
\end{lem}
\begin{proof}
We follow the proof of \cite{GLQ23b}. It is sufficient to show that any morphism of $\delta$-pairs
\[
(B \wh{\otimes}_R A, J) \otimes_{(B[ \ul{\xi} ], (\ul{\xi}))} (B \wh{\otimes}_R A, J) \to (D, I_D)
\]
factors through the multiplication map
\[
(B \wh{\otimes}_R A, J) \otimes_{(B[ \ul{\xi} ], (\ul{\xi}))} (B \wh{\otimes}_R A, J) \to (B \wh{\otimes}_R A, J).
\]
Let $N$ be the kernel of the multiplication from $Q:= A \wh{\otimes}_{R[ \ul{x} ]} A$ to $A$. Since $\ul{x}$ are  \'{e}tale coordinates on $A$, by Tag 02FL of \cite{Sta}, we have an isomorphism $Q \simeq N \times A$ of rings such that the maps $N \to Q$ and $Q \to A$ correspond to the inclusion and the projection respectively. By Lemma 1.3 of \cite{GLQ23b}, we see that $N$ is a $\delta$-ring. So if we let $M := B \wh{\otimes}_R N$, then there exists a sequence of isomorphisms of $\delta$-rings
\[
(B \wh{\otimes}_R A)\wh{\otimes}_{B[ \ul{\xi} ]}(B \wh{\otimes}_R A) \simeq B \wh{\otimes}_{R} Q \simeq B \wh{\otimes}_{R} (N \times A) \simeq M \times (B \wh{\otimes}_R A),
\]
which fits into the following isomorphism of exact sequences
\[
\begin{tikzcd}[column sep=0.6cm]
0 \arrow[r] & (B \wh{\otimes}_R A) \wh{\otimes}_{B[ \ul{\xi} ]} J + J \wh{\otimes}_{B[ \ul{\xi} ]} (B \wh{\otimes}_R A) \arrow[r]\arrow[d,"\simeq"] & (B \wh{\otimes}_R A)\wh{\otimes}_{B[ \ul{\xi} ]}(B \wh{\otimes}_R A) \arrow[r]\arrow[d,"\simeq"] & B \arrow[r]\arrow[d,"\simeq"] & 0 \\
0 \arrow[r] & M \times J \arrow[r] & M \times (B \wh{\otimes}_R A) \arrow[r] & B \arrow[r] & 0.
\end{tikzcd}
\]
Here the upper right map is induced from the multiplication map, and the lower right map is the composition of the projection to $B \wh{\otimes}_R A$ and the multiplication map. Therefore, we have an isomorphism of $\delta$-pairs
\begin{align*}
&(B \wh{\otimes}_R A, J) \otimes_{(B[ \ul{\xi} ], (\ul{\xi}))} (B \wh{\otimes}_R A, J) \\
= &((B \wh{\otimes}_R A)\wh{\otimes}_{B[ \ul{\xi} ]}(B \wh{\otimes}_R A), (B \wh{\otimes}_R A) \wh{\otimes}_{B[ \ul{\xi} ]} J + J \wh{\otimes}_{B[ \ul{\xi} ]} (B \wh{\otimes}_R A)) \\
\simeq &(M \times (B \wh{\otimes}_R A), M \times J) \\
= & (M,M) \times (B \wh{\otimes}_R A, J).
\end{align*}
It remains to show that any morphism of $\delta$-pairs
\[
u : (M,M) \times (B \wh{\otimes}_R A, J) \to (D, I_D)
\]
factors through $(B \wh{\otimes}_R A, J)$. If we denote the image of $(1,0), (0,1)$ under $u$ by $e_1, e_2$ and put $(D_i, I_{D_i}) = (e_iD ,e_iI_D)$ for $i = 1,2$, then the morphism $u$ can be written as the product  $u' \times u''$ with 
\begin{align*}
&u' : (M,M) \to (D_1, I_{D_1}), \\
&u'' : (B \wh{\otimes}_R A, J) \to (D_2, I_{D_2}),
\end{align*}
and by Lemma 1.3 of \cite{GLQ23b}, these are morphisms of $\delta$-pairs. The map $u'$ factors through the $q$-PD envelope of $(M,M)$, which is the zero ring by Examples 2 after the Lemma 4.10 of \cite{GLQ23b}.
\end{proof}

Keep the assumption that $B$ is bounded and that there exists a morphism of $R$-algebras $f: A \to B$. We can consider the morphism of $\delta$-$R$-algebras
\begin{align}
\label{eten}
R[ \ul{X} ] &\to B [\ul{\xi}] \left[\frac{\phi(\ul{\xi})}{(p)_q} \right]^{\delta, \wedge} \\
X_i &\mapsto \xi_i + x_i. \notag
\end{align}
Since $R[ \ul{X} ] \to A$ is \'{e}tale (in particular, it is $(p, (p)_q)$-completely \'{e}tale), and since $\xi_i^p = (p)_q \cdot \frac{\phi(\xi_i)}{(p)_q}-p\delta(\xi_i)$, there exists a unique morphism $\theta$ making the following  diagram commute:
\[
\begin{tikzcd}
A \arrow[dr,dashed,"\theta"]\arrow[r,"f"] & B [\ul{\xi}] \left[\frac{\phi(\ul{\xi})}{(p)_q} \right]^{\delta, \wedge}/(\ul{\xi})  \\
R[ \ul{X} ]  \arrow[u]\arrow[r,"(\ref{eten})"] & B [\ul{\xi}] \left[\frac{\phi(\ul{\xi})}{(p)_q} \right]^{\delta, \wedge}. \arrow[u,twoheadrightarrow]
\end{tikzcd}
\]
It follows from Lemma 3.4 of \cite{GLQ23a} that $\theta$ is a morphism of $\delta$-$A$-algebras. By the linear extension of $\theta$, we get a morphism of $\delta$-rings $\wt{\theta} : B \wh{\otimes}_R A \to B [\ul{\xi}] \left[\frac{\phi(\ul{\xi})}{(p)_q} \right]^{\delta, \wedge}$. 
For an element $\sum_{j=1}^{r} b_j \otimes a_j \in J$, by the constructions of $J$ and $\wt{\theta}$, we can check that $\wt{\theta}(\sum_{j=1}^{r} b_j \otimes a_j) = \sum_{j=1}^{r} b_j( f(a_j)+y_j) = \sum_{j=1}^{r} b_j y_j \in (\ul{\xi})$ for some  $y_j$'s in $(\ul{\xi})$. Therefore, the morphism of $\delta$-rings $\wt{\theta}$ uniquely extends to a morphism of $\delta$-pairs $\wt{\theta} : (B \wh{\otimes}_R A, J) \to \left(B [\ul{\xi}] \left[\frac{\phi(\ul{\xi})}{(p)_q} \right]^{\delta, \wedge}, K\right)$. 
When the $x_i$'s in $B$ are all rank one elements, the morphism $\wt{\theta}$ coincides with the linear extension and the completion of the morphism $A \otimes_R A \to \wh{A \langle \ul{\xi} \rangle}_{q, \ul{x}}$ in D\'{e}finition 3.4.2 of \cite{Hou23}, since these are uniquely characterized by the requirement $1 \otimes x_i - x_i \otimes 1 \mapsto \xi_i$.
Now we have the higher dimensional version of Theorem 7.3 of \cite{GLQ23b} in this more general setting:

\begin{thm}
\label{qpd}
{\rm (cf. \cite[Theorem 7.3]{GLQ23b})} If $B$ is a bounded $\delta$-$R$-algebra and $A \to B$ is a morphism of $R$-algebras, then $\left(B [\ul{\xi}] \left[\frac{\phi(\ul{\xi})}{(p)_q} \right]^{\delta, \wedge}, K\right)$ is the $q$-PD envelope of $(B \wh{\otimes}_R A, J)$, where $K$ is the minimal $(p, (p)_q)$-complete ideal containing $(\xi_1, \dots, \xi_d)$ and stable under the operation $\gamma$. Moreover, if the $x_i$'s in $B$ are all rank one elements, then  $(\wh{B \langle \ul{\xi} \rangle}_{q, \ul{x}}, \ol{I^{[1]}}^{\rm cl})$ is also the $q$-PD envelope of $(B \wh{\otimes}_R A, J)$.
\end{thm}
\begin{proof}
Let $u : (B \wh{\otimes}_R A, J) \to (D, I_D)$ be a morphism to a $q$-PD pair. The restriction of $u$ to $(B [\ul{\xi}], (\xi_1, \dots, \xi_d))$ extends uniquely to a morphism \[ v : \left(B [\ul{\xi}] \left[\frac{\phi(\ul{\xi})}{(p)_q} \right]^{\delta, \wedge}, K\right) \to (D, I_D). \] Since the diagram
\[
\begin{tikzcd}
(B [\ul{\xi}], (\xi_1, \dots, \xi_d)) \arrow[d]\arrow[dr] & \\
(B \wh{\otimes}_R A, J) \arrow[r,"\wt{\theta}"] & \left(B [\ul{\xi}] \left[\frac{\phi(\ul{\xi})}{(p)_q} \right]^{\delta, \wedge}, K\right)
\end{tikzcd}
\]
is commutative, we can apply Lemma \ref{qpdc} to the maps $u_1 = u$ and
\[
u_2 : (B \wh{\otimes}_R A, J) \xrightarrow{\wt{\theta}} \left(B [\ul{\xi}] \left[\frac{\phi(\ul{\xi})}{(p)_q} \right]^{\delta, \wedge}, K\right) \xrightarrow{v} (D, I_D).
\]
The second assertion follows from the description in Theorem \ref{qpdb}.
\end{proof}

Fix a non-negative integer $m$. To develop the level-$m$ $q^{p^m}$-crystalline cohomology theory (that is, the level-$m$ $q$-crystalline cohomology theory when the element $q$ is $q^{p^m}$), we first make the following definitions, which generalize Definitions \ref{dfqpd} and \ref{dfqpd2} (the case $m = 0$) by simply replacing $q$ by $q^{p^m}$:

\begin{defi}
\label{dfqpdm}
A \textit{$q^{p^m}$-PD pair} is a derived $(p,(p)_{q^{p^m}})$-complete $\delta$-pair $(D,I_D)$ consisting of a $\delta$-$R$-algebra $D$ and an ideal $I_D$ which satisfy the following conditions:
\begin{enumerate}
\item For any $f \in I_D$, $\phi (f) - (p)_{q^{p^m}} \delta(f) \in (p)_{q^{p^m}}I_D$.
\item The ring $D$ is $(p)_{q^{p^m}}$-bounded, i.e., $D$ is $(p)_{q^{p^m}}$-torsion-free and $D/((p)_{q^{p^m}})$ has bounded $p^\infty$-torsion.
\item $D/I_D$ is classically $p$-complete.
\end{enumerate}
\end{defi}

\begin{defi}
\label{pqpmpdenv}
Let $(C, I_C)$ be a $\delta$-pair. Then (if it exists) its $q^{p^m}$-PD envelope $(C^{[ \ ]_{q^{p^m}}}, I_C^{[ \ ]_{q^{p^m}}})$ is a $q^{p^m}$-PD pair that is universal for morphisms to  $q^{p^m}$-PD pairs: there exists a morphism of $\delta$-pairs $(C, I_C) \to (C^{[ \ ]_{q^{p^m}}}, I_C^{[ \ ]_{q^{p^m}}})$ such that any morphism $(C, I_C) \to (C', I_{C'})$ to a $q^{p^m}$-PD pair extends uniquely to $(C^{[ \ ]_{q^{p^m}}}, I_C^{[ \ ]_{q^{p^m}}})$.
\end{defi}

\begin{rem}
\label{pqpmrem}
Thanks to Remark \ref{rqpd}.1, the $(p, (p)_{q^{p^m}})$-adic topology we consider coincides with the $(p,q-1)$-adic topology.
\end{rem}

For the arguments later about the level-$m$ $q^{p^m}$-crystalline cohomology theory, we are interested in what is the $q^{p^m}$-PD envelope of $(B \wh{\otimes}_R A, \phi^m(J)(B \wh{\otimes}_R A))$. 
We can calculate it explicitly by using the base change of Theorem \ref{qpdb}. First note that the element $\phi^{m}(x_i) + \phi^{m}(\xi_i)$ has rank one:
\[
\phi(\phi^{m}(x_i) + \phi^{m}(\xi_i)) = \phi^{m+1}(x_i + \xi_i) = (x_i + \xi_i)^{p^{m+1}} = (\phi^{m}(x_i + \xi_i))^p.
\]
In particular, when the $x_i$'s in $B$ are all rank one elements, we see that the element $x_i^{p^m} + \phi^{m}(\xi_i)$ has rank one.
So, if we now take the elements $q, \ul{x}, \ul{\xi}$ which were used in the construction of the twisted divided polynomials to be the rank one elements $q^{p^m},\ul{x}^{p^m}$ and the elements $\phi^{m}(\ul{\xi})$, we can construct the ring $\Z [ q ]_{(p,q-1)}[\ul{x}^{p^m}]\langle \phi^m(\ul{\xi}) \rangle_{q^{p^m}, \ul{x}^{p^m}}$.
Let 
$B[ \ul{\xi} ] \langle \phi^m(\ul{\xi}) \rangle_{q^{p^m}, \ul{x}^{p^m}}$ be the base change of $\Z [ q ]_{(p,q-1)}[\ul{x}^{p^m}]\langle \phi^m(\ul{\xi}) \rangle_{q^{p^m}, \ul{x}^{p^m}}$ along the map of rings
\begin{align}
\label{bcpm}
\Z [ q ]_{(p,q-1)}[\ul{x}^{p^m}][\phi^m(\ul{\xi})] \to B[ \ul{\xi} ].
\end{align}
As a $B[ \ul{\xi} ]$-module, $B[ \ul{\xi} ] \langle \phi^m(\ul{\xi}) \rangle_{q^{p^m}, \ul{x}^{p^m}}$ is generated by $(\phi^m(\ul{\xi}))^{[\ul{k}]_{q^{p^m}, (1-q^{p^m})\ul{x}^{p^m}}}$ indexed by $\ul{k} \in \N^d$, which are heuristically the images of $\ul{\xi}^{[\ul{k}]_{q}}$ under $\phi^m$. When the $x_i$'s in $B$ are all rank one elements, the morphism (\ref{bcpm}) is a map of $\delta$-rings, so we can endow $B[ \ul{\xi} ] \langle \phi^m(\ul{\xi}) \rangle_{q^{p^m}, \ul{x}^{p^m}}$ with a $\delta$-structure.

The following theorem is the higher-level version of Theorem \ref{qpdb}:

\begin{thm}
\label{qpdd}
If $B$ is a $(p)_{q^{p^m}}$-bounded $\delta$-$R$-algebra with elements $\ul{x} = (x_1, \dots, x_d) \in B^d$  and $B [ \ul{\xi}]$ is endowed with the symmetric $\delta$-structure, then $\left(B [\ul{\xi}] \left[\frac{\phi^{m+1}(\ul{\xi})}{(p)_{q^{p^m}}} \right]^{\delta, \wedge}, K_{(m)}\right)$ is the $q^{p^m}$-PD envelope of $(B [ \ul{\xi}], (\phi^m(\xi_1), \dots, \phi^m(\xi_d)))$, where $K_{(m)}$ is the minimal $(p, (p)_{q^{p^m}})$-complete ideal containing $(\phi^m(\xi_1), \dots, \phi^m(\xi_d))$ and stable under the operation $\gamma$. Moreover, if the $x_i$'s in $B$ are all rank one elements, then we have an isomorphism of $\delta$-$R$-algebras
\[
B[ \ul{\xi} ] \langle \phi^m(\ul{\xi}) \rangle_{q^{p^m}, \ul{x}^{p^m}}^{\wedge} \simeq B [\ul{\xi}] \left[\frac{\phi^{m+1}(\ul{\xi})}{(p)_{q^{p^m}}} \right]^{\delta, \wedge}
\]
so that \[ (B[ \ul{\xi} ] \langle \phi^m(\ul{\xi}) \rangle_{q^{p^m}, \ul{x}^{p^m}}^{\wedge}, \ol{I^{[p^m]}}^{\rm cl}) \] is also the $q^{p^m}$-PD envelope of $(B [ \ul{\xi}], (\phi^m(\xi_1), \dots, \phi^m(\xi_d)))$, where $I^{[p^m]}$ is the ideal of $B[ \ul{\xi} ] \langle \phi^m(\ul{\xi}) \rangle_{q^{p^m}, \ul{x}^{p^m}}$ generated by all $(\phi^m(\ul{\xi}))^{[\ul{k}]_{q^{p^m}, (1-q^{p^m})\ul{x}^{p^m}}}$ for $| \ul{k} |  \geq 1$.
\end{thm}
\begin{proof}
The first assertion follows from the assumptions and the universal properties of $\left(B [\ul{\xi}] \left[\frac{\phi^{m+1}(\ul{\xi})}{(p)_{q^{p^m}}} \right]^{\delta, \wedge}, K_{(m)}\right)$. We prove the second assertion. By the definition, the formation of $B[ \ul{\xi} ] \langle \phi^m(\ul{\xi}) \rangle_{q^{p^m}, \ul{x}^{p^m}}$ is stable under the base change of a $\Z [ q ]_{(p,q-1)}[\ul{x}^{p^m}][\phi^m(\ul{\xi})]$-algebra. Since the formation of $q^{p^m}$-PD envelopes is also stable under the base change, it is enough to consider the $q^{p^m}$-PD envelope of \[ (\Z [ q ]_{(p,q-1)}[\ul{x}^{p^m}][\phi^m(\ul{\xi})], (\phi^m(\xi_1), \dots, \phi^m(\xi_d))). \] If we now take the elements $q, \ul{x}, \ul{\xi}$ in Theorem \ref{qpdb} to be respectively $q^{p^m},\ul{x}^{p^m},\phi^{m}(\ul{\xi})$, 
 we get the $q^{p^m}$-PD envelope $(\Z [ q ]_{(p,q-1)}[\ul{x}^{p^m}]\langle \phi^m(\ul{\xi}) \rangle_{q^{p^m}, \ul{x}^{p^m}}^{\wedge}, \ol{I^{[p^m]}}^{\rm cl})$.
\end{proof}

Keep the assumption that $B$ is $(p)_{q^{p^m}}$-bounded, and let $A \to B$ be a morphism of $R$-algebras. In the same way as the level-$0$ case, we can consider the morphism of $\delta$-$R$-algebras
\begin{align}
\label{etena}
R[ \ul{X} ] &\to B [\ul{\xi}] \left[\frac{\phi^{m+1}(\ul{\xi})}{(p)_{q^{p^m}}} \right]^{\delta, \wedge} \\
X_i &\mapsto \xi_i + x_i. \notag
\end{align}
Since $R[ \ul{X} ] \to A$ is \'{e}tale (in particular, it is $(p, (p)_{q^{p^m}})$-completely \'{e}tale), and since $\xi_i^{p^{m+1}} = (p)_{q^{p^m}} \cdot \frac{\phi^{m+1}(\xi_i)}{(p)_{q^{p^m}}} - pz_i$ for some elements $z_i$'s in $B [\ul{\xi}]$ by Remark 2.13 of \cite{BS22}, 
there exists a unique morphism $\theta_{(m)}$ making the following diagram commute:
\[
\begin{tikzcd}
A \arrow[dr,dashed,near end,"\theta_{(m)}"]\arrow[r] & B [\ul{\xi}] \left[\frac{\phi^{m+1}(\ul{\xi})}{(p)_{q^{p^m}}} \right]^{\delta, \wedge}/(\ul{\xi})  \\
R[ \ul{X} ]  \arrow[u]\arrow[r,"(\ref{etena})"] & B [\ul{\xi}] \left[\frac{\phi^{m+1}(\ul{\xi})}{(p)_{q^{p^m}}} \right]^{\delta, \wedge}. \arrow[u,twoheadrightarrow]
\end{tikzcd}
\]
It follows from Lemma 3.4 of \cite{GLQ23a} that $\theta_{(m)}$ is a morphism of $\delta$-$A$-algebras.
We call $\theta_{(m)}$ the \textit{$q$-Taylor map of level $m$}. By the linear extension of $\theta_{(m)}$, we get a morphism of $\delta$-rings
$\wt{\theta}_{(m)} : B \wh{\otimes}_R A \to B [\ul{\xi}] \left[\frac{\phi^{m+1}(\ul{\xi})}{(p)_{q^{p^m}}} \right]^{\delta, \wedge}$. As we did before Theorem \ref{qpd} in the level-0 case, we can check that $\wt{\theta}_{(m)}(J) \subset (\ul{\xi})$. Therefore, the morphism of $\delta$-rings $\wt{\theta}_{(m)}$ uniquely extends to a morphism of $\delta$-pairs \[ \wt{\theta}_{(m)} : (B \wh{\otimes}_R A, \phi^m(J)(B \wh{\otimes}_R A)) \to \left(B [\ul{\xi}] \left[\frac{\phi^{m+1}(\ul{\xi})}{(p)_{q^{p^m}}} \right]^{\delta, \wedge}, K_{(m)}\right). \] We have the higher-level version of Lemma \ref{qpdc}:

\begin{lem}
\label{qpdf}
If two morphisms of $\delta$-pairs $u_1, u_2 : (B \wh{\otimes}_R A, \phi^m(J)(B \wh{\otimes}_R A)) \to (D, I_D)$ to a $q^{p^m}$-PD pair coincide when restricted to $B[ \ul{\xi} ]$, then they are equal.
\end{lem}
\begin{proof}
By the proof of Lemma \ref{qpdc}, we have an isomorphism of $\delta$-pairs
\[
(B \wh{\otimes}_R A, J) \otimes_{(B[ \ul{\xi} ], (\ul{\xi}))} (B \wh{\otimes}_R A, J) \simeq (M,M) \times (B \wh{\otimes}_R A, J).
\]
Therefore, we also have an isomorphism of $\delta$-pairs
\begin{align*}
&(B \wh{\otimes}_R A, \phi^m(J)(B \wh{\otimes}_R A)) \otimes_{(B[ \ul{\xi} ], (\phi^m(\ul{\xi})))} (B \wh{\otimes}_R A, \phi^m(J)(B \wh{\otimes}_R A)) \\
\simeq &(M,\phi^m(M)M) \times (B \wh{\otimes}_R A, \phi^m(J)(B \wh{\otimes}_R A)) \\
\simeq &(M,M) \times (B \wh{\otimes}_R A, \phi^m(J)(B \wh{\otimes}_R A)).
\end{align*}
So it is enough to show that any morphism of $\delta$-pairs
\[
u: (M,M) \times (B \wh{\otimes}_R A, \phi^m(J)(B \wh{\otimes}_R A)) \to (D, I_D)
\]
factors through $(B \wh{\otimes}_R A, \phi^m(J)(B \wh{\otimes}_R A))$. As in the proof of Lemma \ref{qpdc}, this follows immediately from the fact that the $q^{p^m}$-PD envelope of $(M,M)$ is the zero ring.
\end{proof}


Now we have the higher-level version of Theorem \ref{qpd}.

\begin{thm}
\label{qpde}
If $B$ is a $(p)_{q^{p^m}}$-bounded $\delta$-$R$-algebra and $A \to B$ is a morphism of $R$-algebras, then \[ \left(B [\ul{\xi}] \left[\frac{\phi^{m+1}(\ul{\xi})}{(p)_{q^{p^m}}} \right]^{\delta, \wedge}, K_{(m)}\right) \] is the $q^{p^m}$-PD envelope of $(B \wh{\otimes}_R A, \phi^m(J)(B \wh{\otimes}_R A))$, where $K_{(m)}$ is the minimal \linebreak $(p, (p)_{q^{p^m}})$-complete ideal containing $(\phi^m(\xi_1), \dots, \phi^m(\xi_d))$ and stable under the operation $\gamma$. Moreover, if the $x_i$'s in $B$ are all rank one elements, then \[ (B[ \ul{\xi} ] \langle \phi^m(\ul{\xi}) \rangle_{q^{p^m}, \ul{x}^{p^m}}^{\wedge}, \ol{I^{[p^m]}}^{\rm cl}) \] is also the $q^{p^m}$-PD envelope of $(B \wh{\otimes}_R A, \phi^m(J)(B \wh{\otimes}_R A))$.
\end{thm}
\begin{proof}
Let $u : (B \wh{\otimes}_R A, \phi^m(J)(B \wh{\otimes}_R A)) \to (D, I_D)$ be a morphism to a $q^{p^m}$-PD pair. The restriction of $u$ to $(B [\ul{\xi}], (\phi^m(\xi_1), \dots, \phi^m(\xi_d)))$ extends uniquely to a morphism $v : \left(B [\ul{\xi}] \left[\frac{\phi^{m+1}(\ul{\xi})}{(p)_{q^{p^m}}} \right]^{\delta, \wedge}, K_{(m)}\right) \to (D, I_D)$. Since the diagram
\[
\begin{tikzcd}
(B [\ul{\xi}], (\phi^m(\xi_1), \dots, \phi^m(\xi_d))) \arrow[d]\arrow[dr] & \\
(B \wh{\otimes}_R A, \phi^m(J)(B \wh{\otimes}_R A)) \arrow[r,"\wt{\theta}_{(m)}"] & \left(B [\ul{\xi}] \left[\frac{\phi^{m+1}(\ul{\xi})}{(p)_{q^{p^m}}} \right]^{\delta, \wedge}, K_{(m)}\right)
\end{tikzcd}
\]
is commutative, we can apply Lemma \ref{qpdf} to the maps $u_1 = u$ and
\[
u_2 : (B \wh{\otimes}_R A, \phi^m(J)(B \wh{\otimes}_R A)) \xrightarrow{\wt{\theta}_{(m)}} \left(B [\ul{\xi}] \left[\frac{\phi^{m+1}(\ul{\xi})}{(p)_{q^{p^m}}} \right]^{\delta, \wedge}, K_{(m)}\right) \xrightarrow{v} (D, I_D).
\]
The second assertion follows from the description in Theorem \ref{qpdd}.
\end{proof}

On the other hand, in the higher-level crystalline cohomology theory, we can  consider the higher-level analog of the usual PD polynomial algebra. We can consider the \textit{divided power of higher level}  $\ul{\xi}^{{ \{ \ul{k} \} }_{(m)}}$ which satisfies
\[
\prod_{i=1}^d \left\lfloor \frac{k_i}{p^m} \right\rfloor! \xi_i^{ \{ k_i \}_{(m)}}= \ul{\xi}^{  \ul{k}  },
\]
and the \textit{$m$-PD polynomial algebra} which  is free as a module on the generators $\ul{\xi}^{{ \{ \ul{k} \} }_{(m)}}$ indexed by $\ul{k} \in \N^d$. 

Now, for a $\delta$-$R$-algebra $B$ (that need not be $(p)_{q^{p^m}}$-bounded, nor have a morphism of $R$-algebras $A \to B$), we can  consider the natural $q$-analog of the $m$-PD polynomial algebra, which we call the \textit{$m$-$q$-PD polynomial algebra} $B \langle \ul{\xi} \rangle_{(m), q, \ul{x}}$, as follows: as a $B$-module, it is free on abstract generators $\ul{\xi}^{ \{ \ul{k} \}_{(m), q}}$ indexed by $\ul{k} \in \N^d$. 
We first check the following, which is a higher-level analog of Proposition 2.1 of \cite{GLQ22b} in our more general setting:

\begin{prop}
\label{film}
{\rm (cf. \cite[Proposition 2.1]{GLQ22b})} There exists a unique morphism of $B$-modules
\begin{align}
\label{filma}
B[ \ul{\xi} ] &\to B \langle \ul{\xi} \rangle_{(m), q, \ul{x}} \\
\ul{\xi}^{(\ul{k})_{q}} &\mapsto \prod_{i=1}^d \left(\left\lfloor \frac{k_i}{p^m} \right\rfloor\right)_{q^{p^m}} ! \xi_i^{ \{ k_i \}_{(m), q}}. \notag
\end{align}
It is an isomorphism if all positive $q$-integers are invertible in $R$.
\end{prop}
\begin{proof}
This follows immediately from the facts that $\ul{\xi}^{(\ul{k})_{q}}$ indexed by $\ul{k} \in \N^d$ form a basis of $B[ \ul{\xi} ]$, and that $\ul{\xi}^{ \{ \ul{k} \}_{(m), q}}$ indexed by $\ul{k} \in \N^d$ form a basis of $B \langle \ul{\xi} \rangle_{(m), q, \ul{x}}$ by definition.
\end{proof}

To endow $B \langle \ul{\xi} \rangle_{(m), q, \ul{x}}$ with a ring structure such that (\ref{filma}) is a morphism of $B$-algebras, we need to check the following quite involved multiplication rule, although it is not necessary for later computations:

\begin{prop}
\label{mr}
{\rm (cf. \cite[Proposition 2.2]{GLQ22b})} Let $k, k'$ be non-negative integers. The multiplication rule
{\fontsize{9pt}{10pt}\selectfont
\[
\xi_i^{ \{ k \}_{(m), q}}\xi_i^{ \{ k' \}_{(m), q}} = \sum_{0 \le j \le k,k'} \frac{(j)_q!}{\left( \left\lfloor \frac{j}{p^m} \right\rfloor \right)_{q^{p^m}}!} q^{\frac{j(j-1)}{2}}\begin{Bmatrix} k+k'-j \\ k \end{Bmatrix}_{(m),q} \begin{pmatrix} k \\ j \end{pmatrix}_{q} 
\left\langle
\begin{array}{@{}c@{}}
k' \\
j
\end{array}
\right\rangle_{(m),q}
 (q-1)^j x_i^j \xi_i^{ \{ k+k'-j \}_{(m), q}}
\]
}

\noindent
defines a structure of a commutative $B$-algebra on $B \langle \ul{\xi} \rangle_{(m), q, \ul{x}}$ such that the linear map {\rm (\ref{filma})} is a morphism of $B$-algebras.
\end{prop}
Note that, as in the proof of Lemma \ref{qhb}, by using the formula $(p^ml)_q = (p^m)_q (l)_{q^{p^m}}$, we can cancel the terms $(l)_{q^{p^m}}$ in $\left(\left\lfloor \frac{j}{p^m} \right\rfloor\right)_{q^{p^m}}!$ with the terms $(p^m l)_q$ in $(j)_q!$ for $1 \le l \le \left\lfloor \frac{j}{p^m} \right\rfloor$ and show that $\frac{(j)_q!}{\left(\left\lfloor \frac{j}{p^m} \right\rfloor\right)_{q^{p^m}}!} \in \Z[q]$. So this multiplication rule is well-defined over a $\Z [ q ]_{(p,q-1)}$-algebra.
\begin{proof}
As the multi-indices case of the twisted powers is defined as the products of the one-index case, it is enough to show the assertion in the case $d=1$.
By using the maps
\begin{align*}
\Q(q) \hookleftarrow &\Z [ q ]_{(p,q-1)} \to R \\
\Q(q)[x] \hookleftarrow &\Z [ q ]_{(p,q-1)}[x] \to B,
\end{align*}
we may assume that $R=\Q(q)$ and $B = \Q(q)[x]$. In particular, all positive $q$-integers are invertible in $B$, so the map (\ref{filma}) is an isomorphism. Then, using (\ref{mrul}), we can check that the multiplications on both sides coincide because of the following equality:
\begin{align*}
&\left(\left\lfloor \frac{k+k'-j}{p^m} \right\rfloor\right)_{q^{p^m}}! \left(\left\lfloor \frac{j}{p^m} \right\rfloor\right)_{q^{p^m}}! (j)_q! \begin{pmatrix} k \\ j \end{pmatrix}_{q} \begin{pmatrix} k' \\ j \end{pmatrix}_{q}   \\
= &\left(\left\lfloor \frac{k}{p^m} \right\rfloor\right)_{q^{p^m}}! \left(\left\lfloor \frac{k'}{p^m} \right\rfloor\right)_{q^{p^m}}! (j)_q!
\begin{Bmatrix} k+k'-j \\ k \end{Bmatrix}_{(m),q}  \begin{pmatrix} k \\ j \end{pmatrix}_{q} 
\left\langle
\begin{array}{@{}c@{}}
k' \\
j
\end{array}
\right\rangle_{(m),q}.
\end{align*}
\end{proof}

When the $x_i$'s in $B$ are all rank one elements, we can endow $B \langle \ul{\xi} \rangle_{(m), q, \ul{x}}$ with the unique $\delta$-structure such that (\ref{filma}) is a morphism of $\delta$-$B$-algebras. This will be done in Proposition \ref{rg}. First of all,  we can construct a morphism of $B$-modules
\begin{align}
\label{filmb}
B \langle \ul{\xi} \rangle_{(m), q, \ul{x}} &\to B \langle \ul{\xi} \rangle_{q, \ul{x}} \\
\ul{\xi}^{\{ \ul{k} \}_{(m), q}} &\mapsto \prod_{i=1}^d \frac{(k_i)_q!}{\left(\left\lfloor \frac{k_i}{p^m} \right\rfloor\right)_{q^{p^m}}!}  \xi_i^{ [ k_i ]_{q}}. \notag
\end{align}
By the argument before the proof of Proposition \ref{mr}, we see that $\frac{(k_i)_q!}{\left(\left\lfloor \frac{k_i}{p^m} \right\rfloor\right)_{q^{p^m}}!} \in \Z[q]$, so this map is well-defined. To check that this is a morphism of $B$-algebras, as in the proof of Proposition \ref{mr}, we may assume that $R=\Q(q)$ and $B = \Q(q)[\ul{x}]$. The maps (\ref{tdpb}), (\ref{filma}) and (\ref{filmb}) fit into the following commutative diagram:
\[
\begin{tikzcd}
B [\ul{\xi}] \arrow[r,"(\ref{filma})"]\arrow[dr,"(\ref{tdpb})"'] & B \langle \ul{\xi} \rangle_{(m), q, \ul{x}} \arrow[d,"(\ref{filmb})"] \\
 & B \langle \ul{\xi} \rangle_{q, \ul{x}}.
\end{tikzcd}
\]
Both (\ref{tdpb}) and  (\ref{filma}) are isomorphisms of $B$-algebras, since, by the assumption, all positive $q$-integers are invertible in $R$. Therefore, the same holds for (\ref{filmb}). 

First, we consider the $\delta$-structure when $B = \Z [ q ]_{(p,q-1)}[\ul{x}]$. Then the map (\ref{filmb}) is injective (not necessarily isomorphic, since not all positive $q$-integers are invertible in $\Z [ q ]_{(p,q-1)}[\ul{x}]$). By the arguments in Section 2 of \cite{GLQ23b}, we have the formula of the Frobenius lift $\phi$ on $\Z [ q ]_{(p,q-1)}[\ul{x}] \langle \ul{\xi} \rangle_{q, \ul{x}}$:
\[
\phi(\xi_i^{ [ k ]_{q}}) = \sum_{ j = k}^{pk} (p)_q^k b_{k,j} x_i^{pk-j} \xi_i^{ [ j ]_{q}},
\]
where the detailed expressions of $b_{k,j} \in \Z [q ]$ are quite complicated. To see that $\phi$ restricts to an endomorphism of $\Z [ q ]_{(p,q-1)}[\ul{x}] \langle \ul{\xi} \rangle_{(m), q, \ul{x}}$, it is enough to show that the coefficients of $\frac{(j)_q!}{\left(\left\lfloor \frac{j}{p^m} \right\rfloor\right)_{q^{p^m}}!} \xi_i^{ [ j ]_{q}}$ in the formula
{\fontsize{10pt}{10pt}\selectfont
\[
\phi \left( \frac{(k)_q!}{\left(\left\lfloor \frac{k}{p^m} \right\rfloor\right)_{q^{p^m}}!} \xi_i^{ [ k ]_{q}} \right) = \sum_{ j = k}^{pk} (p)_q^k \cdot \frac{(k)_{q^p}!}{\left(\left\lfloor \frac{k}{p^m} \right\rfloor\right)_{q^{p^{m+1}}}!} \cdot \frac{\left(\left\lfloor \frac{j}{p^m} \right\rfloor\right)_{q^{p^m}}!}{(j)_q!} b_{k,j} x_i^{pk-j} \left( \frac{(j)_q!}{\left(\left\lfloor \frac{j}{p^m} \right\rfloor\right)_{q^{p^m}}!} \xi_i^{ [ j ]_{q}} \right)
\]
}

\noindent
are the elements in $\Z [ q ]_{(p,q-1)}[\ul{x}]$. By the argument before the proof of Proposition \ref{mr}, we see that
\[
\frac{\left(\left\lfloor \frac{j}{p^m} \right\rfloor\right)_{q^{p^m}}!}{(j)_q!} = \frac{1}{\prod_{l=0}^{\left\lfloor \frac{j}{p^m} \right\rfloor-1} \left( (p^m)_q \cdot \prod_{u=1}^{p^m-1}(p^m l+u)_{q} \right)} \cdot \frac{1}{\prod_{u'=1}^{j-p^m \left\lfloor \frac{j}{p^m} \right\rfloor} \left(p^m \left\lfloor \frac{j}{p^m} \right\rfloor+u' \right)_q},
\]
which is equal to $\frac{1}{\prod_{l=0}^{\left\lfloor \frac{j}{p^m} \right\rfloor-1} (p^m)_q}$ up to a unit in $\Z [ q ]_{(p,q-1)}$ by Lemma \ref{gcdj}. Since this denominator is biggest when $j = pk$,
 it suffices to show that when $j = pk$, the element
\[
(p)_q^k \cdot \frac{(k)_{q^p}!}{\left(\left\lfloor \frac{k}{p^m} \right\rfloor\right)_{q^{p^{m+1}}}!} \cdot \frac{\left(\left\lfloor \frac{pk}{p^m} \right\rfloor\right)_{q^{p^m}}!}{(pk)_q!}
\]
belongs to $\Z [ q ]_{(p,q-1)}$.
Thanks again to Lemma \ref{gcdj}, for any non-negative integer $m'$, we can ignore the terms $(n)_{q^{p^{m'}}}$ when $\gcd(n,p) = 1$. Then the result follows immediately from the equations
\begin{align*}
(pl)_q = (p)_q (l)_{q^p} \ \ \ &\text{for} \ l = 1, \dots, k,  \\
(pl')_{q^{p^m}} = (p)_{q^{p^m}} (l')_{q^{p^{m+1}}} \ \ \ &\text{for} \ l' = 1, \dots, \left\lfloor \frac{k}{p^m} \right\rfloor.
\end{align*}
As $\Z [ q ]_{(p,q-1)}[\ul{x}]$ is $p$-torsion-free, the Frobenius lift $\phi$ on $\Z [ q ]_{(p,q-1)}[\ul{x}] \langle \ul{\xi} \rangle_{(m), q, \ul{x}}$ corresponds to a unique $\delta$-structure. Now we consider a general $B$ where the $x_i$'s in $B$ are all rank one elements. Then we can  define the $\delta$-structure on $B \langle \ul{\xi} \rangle_{(m), q, \ul{x}}$ by using the isomorphism
\[
B \langle \ul{\xi} \rangle_{(m), q, \ul{x}} \simeq B \otimes_{\Z [ q ]_{(p,q-1)} [\ul{x}]} \Z [ q ]_{(p,q-1)}[\ul{x}] \langle \ul{\xi} \rangle_{(m), q, \ul{x}}.
\]

The ring $B[ \ul{\xi} ] \langle \phi^m(\ul{\xi}) \rangle_{q^{p^m}, \ul{x}^{p^m}}$ that appears in Theorem \ref{qpde} is the ring we are particularly interested in, but its construction is quite complicated. 
On the other hand, the $m$-$q$-PD polynomial algebra $B \langle \ul{\xi} \rangle_{(m), q, \ul{x}}$ is the natural $q$-analog of the usual $m$-PD polynomial algebra. So we can consider the $q$-analogs of the usual calculations for $B \langle \ul{\xi} \rangle_{(m), q, \ul{x}}$. However, the relation between these rings is not clear from the definitions. Actually, these two rings are isomorphic as $\delta$-$B$-algebras:

\begin{prop}
\label{rg}
Let
\begin{align*}
\Z [ q ]_{(p,q-1)}[\ul{x}][ \ul{\xi} ] \langle \phi^m(\ul{\xi}) \rangle_{q^{p^m}, \ul{x}^{p^m}} &\hookrightarrow \Z_{(p)} ( q )[\ul{x}][ \ul{\xi} ] \\
(\phi^m(\ul{\xi}))^{[\ul{k}]_{q^{p^m}, (1-q^{p^m})\ul{x}^{p^m}}} &\mapsto \prod_{i=1}^d \frac{(\phi^m(\xi_i))^{(k_i)_{q^{p^m}, (1-q^{p^m})\ul{x}^{p^m}}}}{(k_i)_{q^{p^m}}!}
\end{align*}
and
\begin{align*}
\Z [ q ]_{(p,q-1)}[\ul{x}]\langle \ul{\xi} \rangle_{(m), q, \ul{x}} &\hookrightarrow \Z_{(p)} ( q )[\ul{x}][ \ul{\xi} ] \\
\ul{\xi}^{\{ \ul{k} \}_{(m), q}} &\mapsto \prod_{i=1}^d \frac{ \xi_i^{ ( k_i )_{q}}}{\left(\left\lfloor \frac{k_i}{p^m} \right\rfloor\right)_{q^{p^m}}!} \notag
 \notag
\end{align*}
be the natural morphisms of $\delta$-rings. Then,
the ring $\Z [ q ]_{(p,q-1)}[\ul{x}][ \ul{\xi} ] \langle \phi^m(\ul{\xi}) \rangle_{q^{p^m}, \ul{x}^{p^m}}$ is isomorphic to the ring $\Z [ q ]_{(p,q-1)}[\ul{x}]\langle \ul{\xi} \rangle_{(m), q, \ul{x}}$ as $\delta$-subrings of $\Z_{(p)} ( q )[\ul{x}][ \ul{\xi} ]$. In particular, by the base change to a $\delta$-$R$-algebra $B$ with elements $\ul{x} = (x_1, \dots, x_d) \in B^d$, we can construct an isomorphism of $B$-algebras
\[
B[ \ul{\xi} ] \langle \phi^m(\ul{\xi}) \rangle_{q^{p^m}, \ul{x}^{p^m}}  \simeq B \langle \ul{\xi} \rangle_{(m), q, \ul{x}}.
\]
If the $x_i$'s in $B$ are all rank one elements, then this is an isomorphism of $\delta$-$B$-algebras.
\end{prop}
\begin{proof}
Let $\ul{k}, \ul{r}, \ul{s} \in \N^d$ such that $k_i = p^m r_i + s_i \ (0 \le s_i < p^m)$ for all $i = 1, \dots, d$.
 By construction, $\{ \ul{\xi}^{ \{ \ul{k} \}_{(m), q}} \}_{\ul{k} \in \N^d}$ is a basis of $\Z [ q ]_{(p,q-1)}[\ul{x}] \langle \ul{\xi} \rangle_{(m), q, \ul{x}}$ as a $\Z [ q ]_{(p,q-1)}[\ul{x}]$-module. Since we have
{\fontsize{9pt}{10pt}\selectfont
\[
\Z [ q ]_{(p,q-1)}[\ul{x}][ \ul{\xi} ] \langle \phi^m(\ul{\xi}) \rangle_{q^{p^m}, \ul{x}^{p^m}} = \Z [ q ]_{(p,q-1)}[\ul{x}][ \ul{\xi} ] \otimes_{\Z [ q ]_{(p,q-1)}[\ul{x}^{p^m}][\phi^m(\ul{\xi})]} \Z [ q ]_{(p,q-1)}[\ul{x}^{p^m}]\langle \phi^m(\ul{\xi}) \rangle_{q^{p^m}, \ul{x}^{p^m}},
\]
}

\noindent
and $\{ \ul{\xi}^{\ul{s}}(\phi^m(\ul{\xi}))^{(\ul{r})_{q^{p^m}, (1-q^{p^m})\ul{x}^{p^m}}} \}_{\ul{k} \in \N^d}$ is a basis of $\Z [ q ]_{(p,q-1)}[\ul{x}][ \ul{\xi} ]$ as a $\Z [ q ]_{(p,q-1)}[\ul{x}]$-module, by moving $(\phi^m(\ul{\xi}))^{(\ul{r})_{q^{p^m}, (1-q^{p^m})\ul{x}^{p^m}}}$ to the right of the tensor product, we see that $\{ \ul{\xi}^{\ul{s}}(\phi^m(\ul{\xi}))^{[\ul{r}]_{q^{p^m}, (1-q^{p^m})\ul{x}^{p^m}}} \}_{\ul{k} \in \N^d}$ is a basis of $\Z [ q ]_{(p,q-1)}[\ul{x}][ \ul{\xi} ] \langle \phi^m(\ul{\xi}) \rangle_{q^{p^m}, \ul{x}^{p^m}}$ as a $\Z [ q ]_{(p,q-1)}[\ul{x}]$-module.  By considering the images of these bases in $\Z_{(p)} ( q )[\ul{x}][ \ul{\xi} ]$, we see that as polynomials in variables $\ul{\xi}$,  the degrees of $\ul{\xi}^{ \{ \ul{k} \}_{(m), q}}$ and $\ul{\xi}^{\ul{s}}(\phi^m(\ul{\xi}))^{[\ul{r}]_{q^{p^m}, (1-q^{p^m})\ul{x}^{p^m}}}$ are both equal  to $| \ul{k} |$, and have a unique term of the highest degree with the same leading coefficients. So, if we are able to check that for any $\ul{r} \in \N^d$, the element
$(\phi^m(\ul{\xi}))^{[\ul{r}]_{q^{p^m}, (1-q^{p^m})\ul{x}^{p^m}}}$
belongs to $\Z [ q ]_{(p,q-1)}[\ul{x}]\langle \ul{\xi} \rangle_{(m), q, \ul{x}}$, then by induction, we can also write $\ul{\xi}^{ \{ \ul{k} \}_{(m), q}}$ as a linear combination of the $\ul{\xi}^{\ul{s}}(\phi^m(\ul{\xi}))^{[\ul{r}]_{q^{p^m}, (1-q^{p^m})\ul{x}^{p^m}}}$'s, so the result follows.
As the multi-indices case of the twisted powers is defined as the products of the one-index case, it is sufficient to show the assertion in the case $d=1$, i.e., for any $r \in \N$, the element $(\phi^m(\xi))^{[r]_{q^{p^m}, (1-q^{p^m})x^{p^m}}}$
belongs to $\Z [ q ]_{(p,q-1)}[x]\langle \xi \rangle_{(m), q, x}$.
Equivalently, we need to show that
\begin{equation}
\label{pxrq}
(\phi^m(\xi))^{(r)_{q^{p^m}, (1-q^{p^m})x^{p^m}}} \in (r)_{q^{p^m}}!\Z [ q ]_{(p,q-1)}[x]\langle \xi \rangle_{(m), q, x}.
\end{equation}

Thanks to Lemma \ref{gcdj}, in the factorial $(r)_{q^{p^m}}!$, we only need to consider the terms $(p^n l)_{q^{p^m}}$, where $l,n$ are positive integers with $p^n l \le r$ such that gcd$(l,p) = 1$. By Proposition \ref{nnp}, we have
\begin{equation}
\label{pnlq}
(p^n l)_{q^{p^m}}= (p)_{q^{p^m}}(p)_{q^{p^{m+1}}} \cdots (p)_{q^{p^{m+n-1}}} (l)_{q^{p^{m+n}}},
\end{equation}
where the term $(l)_{q^{p^{m+n}}}$ can be ignored by Lemma \ref{gcdj}. Note that $\Z [ q ]_{(p,q-1)}$ is a UFD, and we can check that for $j \in \N$, each $(p)_{q^{p^j}}$ is an irreducible polynomial in $\Z [ q ]_{(p,q-1)}$ that is different from each other: we fix a compatible sequence of $p$-power roots of unity $\{ \zeta_{p^{j+1}} \}_{j \in \N}$ as elements of $\C$, which can be used to analyze the terms $(p)_{q^{p^j}}$. The roots of $(p)_{q^{p^j}}$ in $\C$ in variable $q$ are primitive $p^{j+1}$-th roots of unity, so $(p)_{q^{p^j}}$ is the $p^{j+1}$-th cyclotomic polynomial, which is irreducible by the classical result. Therefore, if we want to show (\ref{pxrq}), it is important to count the number of irreducible factors $(p)_{q^{p^j}}$. For this purpose, the next claim is useful:

\begin{clm}
\label{clpin}
Let $n$ be a positive integer, and let $v$ be a non-negative integer. Then:
\[
\prod_{u=0}^{p^n-1} (\phi^m(\xi)+ (1 - q^{u p^m + v p^{m+n}}) x^{p^m}) - \xi^{{(p^{m+n})}_q} \in (p)_{q^{p^{m+n-1}}} \Z [ q ][x][\xi].
\]
\end{clm}

We prove the claim. It is enough to check that, if we assign $q = \zeta_{p^{m+n}}$ in $\C$, then the element in the claim is equal to 0. Now, the twisted power $\xi^{{(p^{m+n})}_q}$ can be factored as the product
\[
\xi^{{(p^{m+n})}_q}  = \prod_{u'=0}^{p^n-1} \prod_{v'=0}^{p^m-1} (\xi + (1-q^{u' + p^n v'})x).
\]
So it is sufficient to show that for $u \in \N$, if we assign $q = \zeta_{p^{m+n}}$, then the element
\begin{equation}
\label{pmxp}
(\phi^m(\xi)+ (1 - q^{u p^m + v p^{m+n}}) x^{p^m}) - \prod_{v'=0}^{p^m-1} (\xi + (1-q^{u + p^n v'})x) 
\end{equation}
is equal to 0. Since the elements $x, \xi + x$ are rank one elements, we see that:
\[
\phi^m(\xi) + x^{p^m}= \phi^m(\xi+x) = (\xi + x)^{p^m}.
\]
So for the left term in (\ref{pmxp}), we can check that:
\[
\phi^m(\xi)+ (1 - q^{u p^m + v p^{m+n}}) x^{p^m} = (\xi + x)^{p^{m}} - q^{u p^m + v p^{m+n}} x^{p^m}.
\]
If we assign $q = \zeta_{p^{m+n}}$, then we have the element
\[
(\xi + x)^{p^{m}} - \zeta_{p^{m+n}}^{u p^m} x^{p^m} = \prod_{v''=0}^{p^m-1} (\xi + x - \zeta_{p^{m+n}}^{u+p^nv''}x),
\]
which is equal to the right term of (\ref{pmxp}) when $q = \zeta_{p^{m+n}}$. This proves the claim.

Now we want to show (\ref{pxrq}). We first note that
\[
\xi^{{(p^{m+n})}_q} = (p^n)_{q^{p^m}} \xi^{ \{ p^{m+n} \}_{(m), q}} = (p)_{q^{p^m}}(p)_{q^{p^{m+1}}} \cdots (p)_{q^{p^{m+n-1}}} \xi^{ \{ p^{m+n} \}_{(m), q}}.
\]
So, by using Claim \ref{clpin}, we see that:
\begin{equation}
\label{uzpn}
\prod_{u=0}^{p^n-1} (\phi^m(\xi)+ (1 - q^{u p^m + v p^{m+n}}) x^{p^m}) \in (p)_{q^{p^{m+n-1}}} \Z [ q ]_{(p,q-1)}[x]\langle \xi \rangle_{(m), q, x}.
\end{equation}
Notice that, for $l_n,s_n \in \N$ such that $0 \le s_n < p^n$, the element
 \[ (\phi^m(\xi))^{(p^n l_n + s_n)_{q^{p^m}, (1-q^{p^m})x^{p^m}}} \] can be factored as the product
\[
\left( \prod_{v=0}^{l_n-1} \prod_{u=0}^{p^n-1} (\phi^m(\xi)+ (1 - q^{u p^m + v p^{m+n}}) x^{p^m}) \right) \left( \prod_{u''=0}^{s_n-1} (\phi^m(\xi)+ (1 - q^{u'' p^m + l_n p^{m+n}}) x^{p^m}) \right).
\]
So, by applying (\ref{uzpn}) to the left term, we see that
\begin{equation}
\label{pmxpa}
(\phi^m(\xi))^{(p^n l_n + s_n)_{q^{p^m}, (1-q^{p^m})x^{p^m}}} \in (p)_{q^{p^{m+n-1}}}^{l_n} \Z [ q ]_{(p,q-1)}[x]\langle \xi \rangle_{(m), q, x}.
\end{equation}
Since $\Z [ q ]_{(p,q-1)}$ is a UFD, and since monic irreducible polynomials $(p)_{q^{p^{m+n-1}}}$ in $\Z [ q ]_{(p,q-1)}$ are distinct for $n \in \N$, if we set $c$ as the largest integer that satisfies $p^c \le r$, we see that 
\begin{equation}
\label{pmxpb}
(\phi^m(\xi))^{(r)_{q^{p^m}, (1-q^{p^m})x^{p^m}}} \in \left( \prod_{n=1}^{c} (p)_{q^{p^{m+n-1}}}^{l_n} \right) \Z [ q ]_{(p,q-1)}[x]\langle \xi \rangle_{(m), q, x}.
\end{equation}

On the other hand, for each positive integer $n$, if we consider \[ r= p^n l_n + s_n \ (0 \le s_n < p^n) \] as before, then by using (\ref{pnlq}), we see that the number of irreducible factors $(p)_{q^{p^{m+n-1}}}$ in $(r)_{q^{p^m}}!$, which appeared in (\ref{pxrq}), is equal to $l_n$. Together with (\ref{pmxpb}), we can check that (\ref{pxrq}) holds, so the result follows.
\end{proof}




Now, we consider the theory of $m$-$q^{p^m}$-crystalline site. Our setting is based on that of Section 1 and 4 in \cite{GLQ23a}. First, we begin with the following general construction. Let $T$ be any site and let $\wt{T}$ be the corresponding topos. If we denote by $\1 := \{0\}$ the final category, then there exists a unique cocontinuous functor $e_T : T \to \1$. By the  identification of the category of sheaves on $\1$ and the category $\S$ets of sets, we obtain the final morphism of topoi
\[
e_T : \wt{T} \to \S\text{ets}.
\]
We have $e_{T*} (E) = \Gamma(T,E)$ and $H^k(T,E) = R^k e_{T*} (E)$.

Next, we recall the notion of $m$-$q^{p^m}$-crystalline site in \cite{Li23}, in the situation we will use later. Let $(R, I_R) \to (A, I_A)$ be a morphism of $q^{p^m}$-PD pairs with fixed rank one  \'{e}tale coordinates $\ul{x} = (x_1, \dots, x_d)$ in $A$. We also assume that \begin{align*}
I_{A} = \ol{I_R A}^{\rm cl}.
\end{align*}
For a $q^{p^m}$-PD pair $(D, I_D)$, let $J_D := (\phi^m)^{-1}(I_D)$.  

\begin{defi}
\label{mqpmc}
Let $\ol{A} := A/J_A$. We define the  $m$-$q^{p^m}$-\textit{crystalline site} $(\ol{A}/R)_{m\text{-}q^{p^m}\text{-crys}}$ of $\ol{A} / R$ as follows.  
Objects are maps $(R,I_R) \to (D,I_D)$ of $q^{p^m}$-PD pairs together with an $R/J_R$-algebra map $\ol{A} \to D/J_D$.
We shall often denote such an object by $(D,I_D)$ if no confusion arises. A morphism is a map of $q^{p^m}$-PD pairs $(D,I_D) \to (D',I_{D'})$ over $(R,I_R)$ such that the induced morphism
\[
D/J_D \to D'/J_{D'}
\]
 is compatible with the maps $\ol{A} \to D/J_D$, $\ol{A} \to D'/J_{D'}$. 
A map ($D,I_D$) $\to$ ($D',I_{D'}$) in  $(\ol{A}/R)_{m\text{-}q^{p^m}\text{-crys}}$ is a cover if it is a $(p,(p)_{q^{p^m}})$-completely faithfully flat map and satisfies
\begin{align*}
I_{D'} = \ol{I_D D'}^{\rm cl}.
\end{align*}
The structural ring $\cO^{(m)}_{\ol{A}/R}$ of the site $(\ol{A}/R)_{m\text{-}q^{p^m}\text{-crys}}$ is given by $\cO^{(m)}_{\ol{A}/R}((D,I_D)) = D$. By the same argument as in the proof of Corollary 3.12 of \cite{BS22}, we can check that $\cO^{(m)}_{\ol{A}/R}$ is a sheaf.
\end{defi}

We can also consider the slice category $q^{p^m}\text{-CRYS}_{/A}$ over $(A,I_A)$ as follows: objects are maps $(A, I_A) \to (D, I_D)$ of $q^{p^m}$-PD pairs, and the notion of morphism is the obvious one. We endow $q^{p^m}\text{-CRYS}_{/A}$ with the flat topology as in Definition \ref{mqpmc}. Then we have the final morphism of topoi
\[
e_A : \wt{q^{p^m}\text{-CRYS}_{/A}} \to \S\text{ets}.
\]
Since the $q^{p^m}$-PD pair $(A, I_A)$ is the final object of the category $q^{p^m}\text{-CRYS}_{/A}$, we see that $e_{A*} (E) = E_A$ is the realization of the sheaf $E$ on $A$. The structural ring $\cO_{q^{p^m}\text{-CRYS}_{/A}}$ of the site $q^{p^m}\text{-CRYS}_{/A}$ is given by $\cO_{q^{p^m}\text{-CRYS}_{/A}}((D,I_D)) = D$. In particular, we have $\cO_{q^{p^m}\text{-CRYS}_{/A}}((A,I_A)) = A$. As above, $\cO_{q^{p^m}\text{-CRYS}_{/A}}$ is a sheaf. So we can consider $e_A$ as a morphism of ringed topoi
\[
e_A : (\wt{q^{p^m}\text{-CRYS}_{/A}},\cO_{q^{p^m}\text{-CRYS}_{/A}}) \to (\S\text{ets},A).
\]
We have $e_{A*} (E) = E_A$ as above, and $e^*_A (M)$ is the sheaf associated to the presheaf $(D,I_D) \mapsto D \otimes_A M$. We define $L \wh{e}^*_A$ as the derived $(p, (p)_{q^{p^m}})$-completion of  $L e^*_A$.

Let $C$ be a commutative ring. We denote by ${\bf FMod}^{\wedge}_{(p,(p)_{q^{p^m}})}(C)$ the category of derived $(p,(p)_{q^{p^m}})$-complete  $(p,(p)_{q^{p^m}})$-completely flat $C$-modules. Then we have the following lemma:

\begin{lem}
\label{fmod}
\phantom{ }
\begin{enumerate}
\item {\rm (cf. \cite[Lemma 2.7]{Tia23})} Let $M,N \in {\bf FMod}^{\wedge}_{(p,(p)_{q^{p^m}})}(A)$. Then the derived $(p, (p)_{q^{p^m}})$-completion of $M \otimes^{L}_A N$ is $(p,(p)_{q^{p^m}})$-completely flat and discrete, and it coincides with the classical $(p, (p)_{q^{p^m}})$-completion of $M \otimes_A N$.
\item {\rm (cf. \cite[Lemma 3.7.(i)]{MT20})} Let $(A, I_A) \to (D, I_D)$ be a map of $q^{p^m}$-PD pairs and let $M \in {\bf FMod}^{\wedge}_{(p,(p)_{q^{p^m}})}(A)$. Then the derived $(p, (p)_{q^{p^m}})$-completion of $D \otimes^{L}_A M$ is $(p,(p)_{q^{p^m}})$-completely flat and discrete, and it coincides with the classical $(p, (p)_{q^{p^m}})$-completion of $D \otimes_A M$.
\end{enumerate}
\end{lem}
\begin{proof}
If $(A, I_A) \to (D, I_D)$ is a map of $q^{p^m}$-PD pairs, then \[ (A, ((p)_{q^{p^m}})) \to (D, ((p)_{q^{p^m}})) \] is a map of bounded prisms. So 1. follows from Lemma 2.7 of \cite{Tia23} by considering the pair of prismatic type $(A, (p, (p)_{q^{p^m}}))$, and 2. follows from Lemma 3.7.(i) of \cite{MT20} by considering $I = ((p)_{q^{p^m}})$.
\end{proof}

In this article, when we consider the derived $(p, (p)_{q^{p^m}})$-completion of the derived tensor product, we always consider the cases in Lemma \ref{fmod}. So we shall simply denote it by $\wh{\otimes}$ when there is no risk of confusion. 

We consider the notion of complete crystal in the sense of  Definition 3.3 of \cite{Tia23}:

\begin{defi}
\label{cpfcr}
A \textit{complete crystal} on $(\ol{A}/R)_{m\text{-}q^{p^m}\text{-crys}}$ (resp. on $q^{p^m}\text{-CRYS}_{/A}$) is a sheaf of $\cO^{(m)}_{\ol{A}/R}$-modules (resp. $\cO_{q^{p^m}\text{-CRYS}_{/A}}$-modules) $E$ such that for each object $(D,I_D)$ of the site, the evaluation $E_D$ of $E$ at $(D,I_D)$ belongs to ${\bf FMod}^{\wedge}_{(p,(p)_{q^{p^m}})}(D)$, and for any morphism $f: (D,I_D) \to (D', I_{D'})$ in the site, the canonical linearized transition map
\[
c_f(E) : D' \wh{\otimes}_D E_D \to E_{D'}
\]
is an isomorphism. We denote by \[{\bf CR}((\ol{A}/R)_{m\text{-}q^{p^m}\text{-crys}},\cO^{(m)}_{\ol{A}/R}) \text{ (resp. } {\bf CR}(q^{p^m}\text{-CRYS}_{/A},\cO_{q^{p^m}\text{-CRYS}_{/A}})) \] the category of complete crystals.
\end{defi}

\begin{prop}
\label{eas}
{\rm (cf. \cite[Proposition 4.1]{GLQ23a})}
The functors $Re_{A*}$ and $L \wh{e}^*_A$ induce an equivalence between   ${\bf CR}(q^{p^m}\text{-}{\rm CRYS}_{/A},\cO_{q^{p^m}\text{-}{\rm CRYS}_{/A}})$ and ${\bf FMod}^{\wedge}_{(p,(p)_{q^{p^m}})}(A)$.
\end{prop}
\begin{proof}
Let $M \in {\bf FMod}^{\wedge}_{(p,(p)_{q^{p^m}})}(A)$. We consider the presheaf $\wh{e}^*_A(M)$ of $\cO_{q^{p^m}\text{-CRYS}_{/A}}$-modules that sends an object $(D,I_D)$ in $q^{p^m}\text{-CRYS}_{/A}$ to $D \wh{\otimes}_{A} M$. By Lemma \ref{fmod}.2, we have $D \wh{\otimes}_{A} M \in {\bf FMod}^{\wedge}_{(p,(p)_{q^{p^m}})}(D)$.  If $(D,I_D) \to (D', I_{D'})$ is a flat cover in $q^{p^m}\text{-CRYS}_{/A}$, then by Proposition 2.9.(i) of \cite{Tia23}, we have an equalizer diagram of $D$-modules
\[
D \wh{\otimes}_{A} M \to D' \wh{\otimes}_{A} M \rightrightarrows D' \wh{\otimes}_{D} D' \wh{\otimes}_{A} M, 
\]
so $\wh{e}^*_A(M)$ is a sheaf. Again by Lemma \ref{fmod}.2 and the definition of $L \wh{e}^*_A$, we see that $L \wh{e}^*_A(M) = \wh{e}^*_A(M)$. The canonical linearized transition maps associated to $\wh{e}^*_A(M)$ is clearly isomorphisms, so we have $L \wh{e}^*_A(M) \in {\bf CR}(q^{p^m}\text{-CRYS}_{/A},\cO_{q^{p^m}\text{-CRYS}_{/A}})$.

Now, for $E \in {\bf CR}(q^{p^m}\text{-CRYS}_{/A},\cO_{q^{p^m}\text{-CRYS}_{/A}})$, we show that $R^i e_{A*} E = 0$ for $i > 0$. We need to show that $H^i (q^{p^m}\text{-CRYS}_{/A},E) = 0$ for $i >0$, but this can be checked as in Lemma 4.12 of \cite{Tia23}.

We can calculate that $e_{A*}\wh{e}^*_A(M) = A \wh{\otimes}_{A} M = M$ and $(\wh{e}^*_Ae_{A*}(E))_D = D \wh{\otimes}_{A} E_A = E_D$, so $e_{A*}$ and $\wh{e}^*_A$ induce an equivalence between ${\bf CR}(q^{p^m}\text{-CRYS}_{/A},\cO_{q^{p^m}\text{-CRYS}_{/A}})$ and ${\bf FMod}^{\wedge}_{(p,(p)_{q^{p^m}})}(A)$.
\end{proof}

Now, we consider the localization functor
\[
j_A : q^{p^m}\text{-CRYS}_{/A} \to (\ol{A}/R)_{m\text{-}q^{p^m}\text{-crys}}
\]
that sends a map $(A, I_A) \to (D, I_D)$ of $q^{p^m}$-PD pairs to a map $(R,I_R) \to (D,I_D)$ of $q^{p^m}$-PD pairs together with an $R/J_R$-algebra map $\ol{A} \to D/J_D$. This is a morphism of sites which extends to a morphism of topoi
\[
j_A : \wt{q^{p^m}\text{-CRYS}_{/A}} \to \wt{(\ol{A}/R)}_{m\text{-}q^{p^m}\text{-crys}}.
\]

\begin{defi}
If $M \in {\bf FMod}^{\wedge}_{(p,(p)_{q^{p^m}})}(A)$, then the \textit{linearization} of $M$ is \[ L^{(m)}(M) := j_{A*}\wh{e}^*_A(M), \] which is a sheaf on $(\ol{A}/R)_{m\text{-}q^{p^m}\text{-crys}}$.
\end{defi}

To compute the linearization, we need to consider the product in $(\ol{A}/R)_{m\text{-}q^{p^m}\text{-crys}}$. For a $\delta$-pair $(C,J)$, we denote by $\ol{J}^{\rm cl}$ the closure of $J$ in $C$ for the $(p, (p)_{q^{p^m}})$-adic topology. 

\begin{prop}
\label{prod}
If $(B,I_B)$ is an object of $(\ol{A}/R)_{m\text{-}q^{p^m}\text{-crys}}$, then  \[
\left(B [\ul{\xi}]\left[\frac{\phi^{m+1}(\ul{\xi})}{(p)_{q^{p^m}}}\right]^{\delta, \wedge}, \ol{K_{(m)} + I_B B [\ul{\xi}]\left[\frac{\phi^{m+1}(\ul{\xi})}{(p)_{q^{p^m}}}\right]^{\delta, \wedge}}^{\rm cl}\right)
\]
is the coproduct of $(B,I_B)$ and $(A,I_A)$ in $(\ol{A}/R)_{m\text{-}q^{p^m}\text{-crys}}$. Moreover, if there exists a morphism $(A,I_A) \to (B,I_B)$ in $(\ol{A}/R)_{m\text{-}q^{p^m}\text{-crys}}$, then it is isomorphic to 
\[
\left(\wh{B \langle \ul{\xi} \rangle}_{(m), q, \ul{x}}, \ol{\ol{I^{[p^m]}}^{\rm cl}+I_B \wh{B \langle \ul{\xi} \rangle}}^{\rm cl}_{(m), q, \ul{x}}\right).
\]
\end{prop}
In particular when $(B,I_B)=(A,I_A)$, we use the identification of $A [\ul{\xi}]\left[\frac{\phi^{m+1}(\ul{\xi})}{(p)_{q^{p^m}}}\right]^{\delta, \wedge}$ and $\wh{A \langle \ul{\xi} \rangle}_{(m), q, \ul{x}}$ without comment.
\begin{proof}
We denote the structural map $\ol{A} \to B/J_B =: \ol{B}$ by $\ol{f}$. First, we can choose the liftings of the images of the $x_i$'s under the map $A \to \ol{A} \xrightarrow{\ol{f}} \ol{B}$ and construct the map $R[ \ul{X} ] \to B$. For any element $y \in J_B$, we can check that $y^{p^{m+1}} = \phi^{m+1}(y) - pz$
 for an element $z$ in $B$ by Remark 2.13 of \cite{BS22}, and that $\phi^{m+1}(y) \in (p)_{q^{p^m}}B$ by the definitions of the ideal $J_B$ and the $q^{p^m}$-PD pair.
Since $R[ \ul{X} ] \to A$ is $(p, (p)_{q^{p^m}})$-completely \'{e}tale, by the arguments as in the proofs of Proposition 1.5.5 of \cite{Ber74} and Th\'{e}or\`{e}me 18.1.2 of \cite{EGA4}, 
there exists a unique morphism of $R$-algebras $f: A \to B$ making the following diagram commute:
\[
\begin{tikzcd}
A \arrow[drr,dashed,"f"]\arrow[r,twoheadrightarrow] & \ol{A} \arrow[r,"\ol{f}"] & \ol{B}  \\
R[ \ul{X} ]  \arrow[u]\arrow[rr] && B, \arrow[u,twoheadrightarrow]
\end{tikzcd}
\]
so we can consider the $\delta$-ring $B \wh{\otimes}_R A$ and its ideal $J$ as before, 
where $J$ is the kernel of multiplication $B \wh{\otimes}_R A \twoheadrightarrow B$. By Theorem \ref{qpde}, we can construct the $q^{p^m}$-PD envelope $\left(B [\ul{\xi}]\left[\frac{\phi^{m+1}(\ul{\xi})}{(p)_{q^{p^m}}}\right]^{\delta, \wedge}, K_{(m)}\right)$ of $(B \wh{\otimes}_R A, \phi^m(J)(B \wh{\otimes}_R A))$. Since $I_B$ is a $q^{p^m}$-PD ideal of $B$, by Lemma 3.2 and Corollary 4.4 of \cite{GLQ23b}, we see that 
\[ (B', I_{B'}) :=  \left(B [\ul{\xi}]\left[\frac{\phi^{m+1}(\ul{\xi})}{(p)_{q^{p^m}}}\right]^{\delta, \wedge}, \ol{K_{(m)} + I_B B [\ul{\xi}]\left[\frac{\phi^{m+1}(\ul{\xi})}{(p)_{q^{p^m}}}\right]^{\delta, \wedge}}^{\rm cl}\right) \] 
is a $q^{p^m}$-PD pair. As we assume that $I_{A} = \ol{I_R A}^{\rm cl}$, we have the morphisms of $q^{p^m}$-PD pairs  $\theta_{(m)} : (A,I_A) \to (B', I_{B'})$ and $ (B,I_B) \to (B', I_{B'})$. 
Since  we have \linebreak $\phi^m(\xi_i) \in K_{(m)}$ for any $i=1, \dots, d$, we see that $\xi_i \in J_{B'}$. By the construction of $B' = B [\ul{\xi}]\left[\frac{\phi^{m+1}(\ul{\xi})}{(p)_{q^{p^m}}}\right]^{\delta, \wedge}$, we can check that $\ol{B'} := B'/J_{B'}$ is a quotient of $\ol{B}$, and by the construction of $\theta_{(m)}$, we see that $\ol{\theta}_{(m)} : \ol{A} \to \ol{B'}$ is a quotient of $\ol{f}$. Therefore, the maps $\theta_{(m)} : (A,I_A) \to (B', I_{B'})$ and $ (B,I_B) \to (B', I_{B'})$ are actually morphisms in $(\ol{A}/R)_{m\text{-}q^{p^m}\text{-crys}}$.
It remains to show that for any morphisms 
\begin{align*}
&g: (A,I_A) \to (C, I_{C}) \\
&h: (B,I_B) \to (C, I_{C})
\end{align*}
in $(\ol{A}/R)_{m\text{-}q^{p^m}\text{-crys}}$, these morphisms uniquely extend to a morphism \[ (B', I_{B'}) \to (C, I_{C}) \] in $(\ol{A}/R)_{m\text{-}q^{p^m}\text{-crys}}$.

We first have a unique morphism of $\delta$-rings
\begin{align*}
F: B \wh{\otimes}_R A &\to C \\
\sum_{j=1}^{r} b_j \otimes a_j &\mapsto \sum_{j=1}^{r} h(b_j)g(a_j)
\end{align*}
that extends the underlying rings maps of $g,h$. We can also consider the following $\ol{A}$-algebras map induced by $g,h$:
\begin{align*}
\ol{g} : \ol{A} &\to C/J_C =: \ol{C}, \\
\ol{h}: \ol{B} &\to \ol{C}.
\end{align*}
then we have $\ol{h} \circ \ol{f} = \ol{g}$. Now, for any element $\sum_{j=1}^{r} b_j \otimes a_j \in J$, we have \[ \sum_{j=1}^{r} b_j \cdot f(a_j)=0 \] by definition. So we have
\[
\ol{F\left(\sum_{j=1}^{r} b_j \otimes a_j\right)} = \sum_{j=1}^{r} \ol{h}(\ol{b_j})\ol{g}(\ol{a_j}) = \ol{h} \left(\sum_{j=1}^{r} \ol{b_j} \cdot \ol{f}(\ol{a_j})\right) =0,
\]
namely, $F(\sum_{j=1}^{r} b_j \otimes a_j) \in J_C$. Therefore, the morphism of $\delta$-rings $F$ uniquely extends to a morphism of $\delta$-pairs:
\[
F: (B \wh{\otimes}_R A, \phi^m(J)(B \wh{\otimes}_R A)) \to (C, I_C).
\]
By the universal property of $q^{p^m}$-PD envelope, we can get a unique morphism of $q^{p^m}$-PD pairs $(B',K_{(m)}) \to (C,I_C)$  extending $F$, and it actually extends to a unique morphism of $q^{p^m}$-PD pairs $(B',I_{B'}) \to (C,I_C)$. Since $\ol{B'}$ is a quotient of $\ol{B}$, it is easy to check that this is a morphism in $(\ol{A}/R)_{m\text{-}q^{p^m}\text{-crys}}$.



For the second assertion, if we assume that there exists a morphism \[ (A,I_A) \to (B,I_B) \] in $(\ol{A}/R)_{m\text{-}q^{p^m}\text{-crys}}$, then the  $x_i$'s in $B$ are all rank one elements. Then the result follows from Theorem \ref{qpde} and Proposition \ref{rg}.
\end{proof}

\begin{prop}
\label{mqcov}
For any object $(B,I_B)$ in $(\ol{A}/R)_{m\text{-}q^{p^m}\text{-crys}}$, we can find a map $i_1: (B,I_B) \to (B', I_{B'})$ and a map $i_2 :(A,I_A) \to (B', I_{B'})$ in $(\ol{A}/R)_{m\text{-}q^{p^m}\text{-crys}}$, where the underlying rings map of $i_1$ is $(p, (p)_{q^{p^m}})$-completely faithfully flat.
\end{prop}
\begin{proof}
By Proposition 5.2 of \cite{GLQ23a}, we see that $B [\ul{\xi}]\left[\frac{\phi^{m+1}(\ul{\xi})}{(p)_{q^{p^m}}}\right]^{\delta, \wedge}$ is $(p, (p)_{q^{p^m}})$-completely faithfully flat over $B$. So we can simply choose the object in Proposition \ref{prod}, and use the fact that it is the coproduct of $(B,I_B)$ and $(A,I_A)$.
\end{proof}

\begin{rem}
\label{covrm}
For the object in Proposition \ref{prod}, the inclusion \[ K_{(m)} \subset I_B B [\ul{\xi}]\left[\frac{\phi^{m+1}(\ul{\xi})}{(p)_{q^{p^m}}}\right]^{\delta, \wedge} \] does not hold in general. Therefore, Proposition 5.3 of \cite{GLQ23a} would no longer be correct if we used the flat topology in Definition \ref{mqpmc}.
\end{rem}

From now on, if we write $B [\ul{\xi}]\left[\frac{\phi^{m+1}(\ul{\xi})}{(p)_{q^{p^m}}}\right]^{\delta, \wedge} \otimes'_A -$, then the notation $\otimes'$ indicates that we use the level-$m$ $q$-Taylor map $\theta_{(m)}$ (which was introduced before Lemma \ref{qpdf}) for the $A$-structure on the left hand side. 

By Proposition \ref{prod}, we can calculate the functor
\begin{align*}
&j_A^{-1} : \wt{(\ol{A}/R)}_{m\text{-}q^{p^m}\text{-crys}} \to \wt{q^{p^m}\text{-CRYS}_{/A}} \\
&j_{A*} : \wt{q^{p^m}\text{-CRYS}_{/A}} \to \wt{(\ol{A}/R)}_{m\text{-}q^{p^m}\text{-crys}}
\end{align*}
associated to the morphism of topoi $j_A$. By the same argument as in the proof of Proposition 5.25 of \cite{BO78}, we see that for $F \in  \wt{(\ol{A}/R)}_{m\text{-}q^{p^m}\text{-crys}}$, the sheaf $j_A^{-1}(F)$ \nolinebreak sends  an object $(D,I_D)$ of $q^{p^m}\text{-CRYS}_{/A}$ to $F_D$, and for an object $(B,I_B)$ of $(\ol{A}/R)_{m\text{-}q^{p^m}\text{-crys}}$, the sheaf $j_A^{-1}((B,I_B))$ is represented by the object
\[
\theta_{(m)} : (A,I_A) \to \left(B [\ul{\xi}]\left[\frac{\phi^{m+1}(\ul{\xi})}{(p)_{q^{p^m}}}\right]^{\delta, \wedge}, \ol{K_{(m)} + I_B B [\ul{\xi}]\left[\frac{\phi^{m+1}(\ul{\xi})}{(p)_{q^{p^m}}}\right]^{\delta, \wedge}}^{\rm cl}\right)
\]
of $q^{p^m}\text{-CRYS}_{/A}$. Then it follows that for $F' \in  \wt{q^{p^m}\text{-CRYS}_{/A}}$, the sheaf $j_{A*}(F')$ sends an object $(B,I_B)$ of $(\ol{A}/R)_{m\text{-}q^{p^m}\text{-crys}}$ to $F'_{B [\ul{\xi}]\left[\frac{\phi^{m+1}(\ul{\xi})}{(p)_{q^{p^m}}}\right]^{\delta, \wedge}}$. This gives an 
explicit description of $L^{(m)}(M)$:

\begin{lem}
\label{lmmb}
Let $M \in {\bf FMod}^{\wedge}_{(p,(p)_{q^{p^m}})}(A)$.
\begin{enumerate}
\item If $(B,I_B)$ is an object of $(\ol{A}/R)_{m\text{-}q^{p^m}\text{-crys}}$, then
\[
L^{(m)}(M)_{B} = B [\ul{\xi}]\left[\frac{\phi^{m+1}(\ul{\xi})}{(p)_{q^{p^m}}}\right]^{\delta, \wedge} \wh{\otimes}_A' M.
\]
\item If $(B,I_B) \to (C, I_C)$ is a morphism in $(\ol{A}/R)_{m\text{-}q^{p^m}\text{-crys}}$, there exists a canonical isomorphism
\[
C \wh{\otimes}_B L^{(m)}(M)_{B} \xrightarrow{\simeq} L^{(m)}(M)_{C}.
\]
In particular, we have $L^{(m)}(M) \in {\bf CR}((\ol{A}/R)_{m\text{-}q^{p^m}\text{-crys}},\cO^{(m)}_{\ol{A}/R})$.
\end{enumerate}
\end{lem}
\begin{proof}
The first assertion follows from the calculation of $j_{A*}$ above and the definition of $\wh{e}^*_A$ in the proof of Proposition  \ref{eas}, so it follows from Lemma \ref{fmod}.2 that $L^{(m)}(M)_{B} \in {\bf FMod}^{\wedge}_{(p,(p)_{q^{p^m}})}(B)$. Therefore, the second assertion follows from Lemma 16.10(3) of \cite{BS22}.
\end{proof}

The following standard results are useful to calculate the linearizations:

\begin{lem}
\label{elle}
{\rm (cf. \cite[Lemma 4.7]{GLQ23a})} If $E \in {\bf CR}((\ol{A}/R)_{m\text{-}q^{p^m}\text{-crys}},\cO^{(m)}_{\ol{A}/R})$ and $M \in {\bf FMod}^{\wedge}_{(p,(p)_{q^{p^m}})}(A)$, then
\[
E \wh{\otimes}_{\cO^{(m)}_{\ol{A}/R}} L^{(m)}(M) \simeq L^{(m)}(E_A \wh{\otimes}_A M).
\]
\end{lem}
\begin{proof}
We follow the proof of \cite{GLQ23a}. If $(A,I_A) \to (B,I_B)$ is a map of $q^{p^m}$-PD pairs, then
\[
(j_A^{-1} E)_B = E_B = B \wh{\otimes}_A E_A = (\wh{e}^*_A E_A)_B.
\]
So we see that $j_A^{-1} E = \wh{e}^*_A E_A$. By the adjunction map $j_A^{-1} j_{A*} \wh{e}^*_A M \to \wh{e}^*_A M$, we get the map
\begin{align*}
j_A^{-1} (E \wh{\otimes}_{\cO^{(m)}_{\ol{A}/R}} j_{A*} \wh{e}^*_A M) &= j_A^{-1} E \wh{\otimes}_{\cO_{q^{p^m}\text{-CRYS}_{/A}}} j_A^{-1}  j_{A*} \wh{e}^*_A M \\ 
&= \wh{e}^*_A E_A \wh{\otimes}_{\cO_{q^{p^m}\text{-CRYS}_{/A}}} j_A^{-1}  j_{A*} \wh{e}^*_A M \\
&\to \wh{e}^*_A E_A \wh{\otimes}_{\cO_{q^{p^m}\text{-CRYS}_{/A}}} \wh{e}^*_A M \\
&= (\cO_{q^{p^m}\text{-CRYS}_{/A}} \wh{\otimes}_{\ul{A}} \hspace{1pt} \ul{E_A}) \wh{\otimes}_{\cO_{q^{p^m}\text{-CRYS}_{/A}}} (\cO_{q^{p^m}\text{-CRYS}_{/A}} \wh{\otimes}_{\ul{A}} \ul{M})  \\
&= \cO_{q^{p^m}\text{-CRYS}_{/A}} \wh{\otimes}_{\ul{A}}( \ul{E_A} \wh{\otimes}_{\ul{A}} \ul{M}) \\
&= \wh{e}^*_A (E_A \wh{\otimes}_A M). 
\end{align*}
Again by adjunction, we get the natural map
\[
E \wh{\otimes}_{\cO^{(m)}_{\ol{A}/R}} L^{(m)}(M) = E \wh{\otimes}_{\cO^{(m)}_{\ol{A}/R}} j_{A*} \wh{e}^*_A M \to j_{A*}\wh{e}^*_A (E_A \wh{\otimes}_A M) = L^{(m)}(E_A \wh{\otimes}_A M).
\]
It remains to show that this is an isomorphism. It is sufficient to show that for any object $(B,I_B)$  of $(\ol{A}/R)_{m\text{-}q^{p^m}\text{-crys}}$, we have
\[
E_B \wh{\otimes}_B L^{(m)}(M)_B \simeq L^{(m)}(E_A \wh{\otimes}_A M)_B.
\]
We write $B' := B [\ul{\xi}]\left[\frac{\phi^{m+1}(\ul{\xi})}{(p)_{q^{p^m}}}\right]^{\delta, \wedge}$. Since $E \in {\bf CR}((\ol{A}/R)_{m\text{-}q^{p^m}\text{-crys}},\cO^{(m)}_{\ol{A}/R})$, by Lemma \ref{lmmb}.1, we have
\[
E_B \wh{\otimes}_B L^{(m)}(M)_B \simeq E_B \wh{\otimes}_B B' \wh{\otimes}'_A M \simeq E_{B'} \wh{\otimes}'_A M \simeq B' \wh{\otimes}'_A E_A \wh{\otimes}_A M \simeq L^{(m)}(E_A \wh{\otimes}_A M)_B. \qedhere
\]
\end{proof}

\begin{lem}
\label{rija}
{\rm (cf. \cite[Corollary 4.4]{GLQ23a})} We have $R^i j_{A*} E = 0$ for $i > 0$ and $E \in {\bf CR}(q^{p^m}\text{-}{\rm CRYS}_{/A},\cO_{q^{p^m}\text{-}{\rm CRYS}_{/A}})$.
\end{lem}
\begin{proof}
We follow the proof of \cite{GLQ23a}. We write $B' := B [\ul{\xi}]\left[\frac{\phi^{m+1}(\ul{\xi})}{(p)_{q^{p^m}}}\right]^{\delta, \wedge}$. Then the sheaf $R^i j_{A*} E$ is associated to the presheaf $(B,I_B) \mapsto H^i(q^{p^m}\text{-CRYS}_{/B'}, E_{|B'})$. But by considering the case $A=B'$ in  Proposition \ref{eas}, we see  that \[ H^i(q^{p^m}\text{-CRYS}_{/B'}, E_{|B'}) = R^i e_{B'*} E_{|B'} = 0.  \qedhere \]
\end{proof}

\begin{lem}
\label{rglm}
{\rm (cf. \cite[Corollary 4.5]{GLQ23a})}
If $M \in {\bf FMod}^{\wedge}_{(p,(p)_{q^{p^m}})}(A)$, then
\[
R \Gamma ((\ol{A}/R)_{m\text{-}q^{p^m}\text{-crys}}, L^{(m)}(M)) = M.
\]
\end{lem}
\begin{proof}
We follow the proof of \cite{GLQ23a}. As in the previous construction, we have the final morphism of topoi
\[
e_{\ol{A}/R} : \wt{(\ol{A}/R)}_{m\text{-}q^{p^m}\text{-crys}} \to \S\text{ets}
\]
which satisfies $e_{\ol{A}/R} \circ j_{A} = e_{A}$. Then by Lemma \ref{rija} and Proposition \ref{eas}, we can calculate that
\begin{align*}
R \Gamma ((\ol{A}/R)_{m\text{-}q^{p^m}\text{-crys}}, L^{(m)}(M)) &= Re_{\ol{A}/R*}L^{(m)}(M) \\
&= Re_{\ol{A}/R*}j_{A*}\wh{e}^*_AM \\
&=Re_{\ol{A}/R*}Rj_{A*}\wh{e}^*_AM \\
&= R(e_{\ol{A}/R*}j_{A*})\wh{e}^*_AM \\
&= Re_{A*}\wh{e}^*_AM \\
&= M.  \qedhere
\end{align*} 
\end{proof}

There also exists a notion of stratification in our setting that reads as follows:

\begin{defi}
\label{mqst}
A \textit{hyper $m$-$q^{p^m}$-stratification} on $M \in {\bf FMod}^{\wedge}_{(p,(p)_{q^{p^m}})}(A)$ is an $\wh{A \langle \ul{\xi} \rangle}_{(m), q, \ul{x}}$-linear isomorphism
\[
\epsilon_M : \wh{A \langle \ul{\xi} \rangle}_{(m), q, \ul{x}} \wh{\otimes}'_A M \simeq M \wh{\otimes}_A \wh{A \langle \ul{\xi} \rangle}_{(m), q, \ul{x}}
\]
satisfying the cocycle condition
\[
(\epsilon_M \wh{\otimes} {\rm Id}_{\wh{A \langle \ul{\xi} \rangle}_{(m), q, \ul{x}}}) \circ ({\rm Id}_{\wh{A \langle \ul{\xi} \rangle}_{(m), q, \ul{x}}} \wh{\otimes}' \epsilon_M) \circ (\delta^1_1  \wh{\otimes}'{\rm Id}_M) = ({\rm Id}_M \wh{\otimes} \delta^1_1) \circ \epsilon_M,
\]
where
\[
\delta^1_1 : \wh{A \langle \ul{\xi} \rangle}_{(m), q, \ul{x}} \to \wh{A \langle \ul{\xi} \rangle}_{(m), q, \ul{x}} \wh{\otimes}'_A \wh{A \langle \ul{\xi} \rangle}_{(m), q, \ul{x}}
\]
is the comultiplication map that will be defined after Proposition \ref{pmar}.
\end{defi}

\begin{prop}
\label{crstr}
The category ${\bf CR}((\ol{A}/R)_{m\text{-}q^{p^m}\text{-crys}},\cO^{(m)}_{\ol{A}/R})$ is equivalent to the category of derived $(p,(p)_{q^{p^m}})$-complete  $(p,(p)_{q^{p^m}})$-completely flat $A$-modules endowed with a hyper $m$-$q^{p^m}$-stratification.
\end{prop}
\begin{proof}
By using Proposition \ref{prod} and \ref{mqcov}, this can be proved in the same way as in Proposition 4.8 of \cite{Tia23} by replacing bounded prisms with $q^{p^m}$-PD pairs, and putting $\wt{R} = A$ and $\wt{B} = B [\ul{\xi}]\left[\frac{\phi^{m+1}(\ul{\xi})}{(p)_{q^{p^m}}}\right]^{\delta, \wedge}$ in the notation there.
\end{proof}

\section{$q$-analog of the higher de Rham complex}
\label{hple}

In this section, we construct a $q$-analog of the ``de Rham-like'' complex and prove the corresponding Poincar\'{e} lemma, following \cite{Miy15}. 

First, we do some differential calculus related to the $m$-$q^{p^m}$-crystalline theory. Let $(R, I_R) \to (A, I_A)$ be a morphism of $q^{p^m}$-PD pairs with fixed rank one  \'{e}tale coordinates $\ul{x} = (x_1, \dots, x_d)$ in $A$. We also assume that $
I_{A} = \ol{I_R A}^{\rm cl}$. Let $P := A \wh{\otimes}_R A$. Similarly to Section \ref{mqanc},  if we write $- \otimes'_A -$, then the notation $\otimes'$ indicates that we use the right $A$-module structure on the left hand side. For $r \in \N$, let $I(r)$ be the kernel of the surjection 
\[
A^{\wh{\otimes} (r+1)} = \underbrace{A \wh{\otimes}_R \cdots \wh{\otimes}_R A}_{r+1 \text{ times}} \simeq \underbrace{P \wh{\otimes}'_A \cdots  \wh{\otimes}'_A P}_{r \text{ times}} \to A.
\]
Let $P^{(m)}_{A/R,q} (r)$ be the $q^{p^m}$-PD envelope of $(A^{\wh{\otimes} (r+1)}, \phi^m(I(r))A^{\wh{\otimes} (r+1)})$ with the $A$-module structure induced by multiplication to the first factor of $A^{\wh{\otimes} (r+1)}$. By Theorem \ref{qpde} and Proposition \ref{rg}, we have $P^{(m)}_{A/R,q} (1) = \wh{A \langle \ul{\xi} \rangle}_{(m), q, \ul{x}}$. More generally, we have the following result:

\begin{prop}
\label{pmar}
The graded $A$-algebra $P^{(m)}_{A/R,q} (\bullet)$ can be identified with the complete tensor algebra of $\wh{A \langle \ul{\xi} \rangle}_{(m), q, \ul{x}}$ over $A$, namely, we have
\[
P^{(m)}_{A/R,q} (\bullet) = \bigoplus_{r=0}^{\infty} \wh{A \langle \ul{\xi} \rangle}_{(m), q, \ul{x}}^{\wh{\otimes}' r}.
\]
In particular, $P^{(m)}_{A/R,q} (2) = \wh{A \langle \ul{\xi} \rangle}_{(m), q, \ul{x}} \wh{\otimes}'_A \wh{A \langle \ul{\xi} \rangle}_{(m), q, \ul{x}}$.
\end{prop}
\begin{proof}
We can check as in the proof of Corollary 7.4 of \cite{GLQ23b} that $\wh{A \langle \ul{\xi} \rangle}_{(m), q, \ul{x}}^{\wh{\otimes}' r}$ is bounded. Thus, by Lemma 4.10 of \cite{GLQ23b}, it is the $q^{p^m}$-PD envelope of $P^{\wh{\otimes}' r} = A^{\wh{\otimes} (r+1)}$ with respect to the ideal $\phi^m(I(r))A^{\wh{\otimes} (r+1)}$. Thus we have the canonical isomorphism  $P^{(m)}_{A/R,q} (r) \simeq \wh{A \langle \ul{\xi} \rangle}_{(m), q, \ul{x}}^{\wh{\otimes}' r}$ for $r \in \N$. Moreover, by definition, we have the product structure 
\[
P^{(m)}_{A/R,q} (r) \wh{\otimes}'_A P^{(m)}_{A/R,q} (s) \to P^{(m)}_{A/R,q} (r+s)  
\]
induced by the natural morphism 
\[
P^{\wh{\otimes}' r} \wh{\otimes}'_A P^{\wh{\otimes}' s} \to P^{\wh{\otimes}' (r+s)},
\]
which is compatible with the natural morphism 
\[
\wh{A \langle \ul{\xi} \rangle}_{(m), q, \ul{x}}^{\wh{\otimes}' r} \wh{\otimes}'_A \wh{A \langle \ul{\xi} \rangle}_{(m), q, \ul{x}}^{\wh{\otimes}' s} \to \wh{A \langle \ul{\xi} \rangle}_{(m), q, \ul{x}}^{\wh{\otimes}' (r+s)}.
\]
Thus the isomorphisms $P^{(m)}_{A/R,q} (r) \simeq \wh{A \langle \ul{\xi} \rangle}_{(m), q, \ul{x}}^{\wh{\otimes}' r}$ for $r \in \N$ induce the isomorphism
\[
P^{(m)}_{A/R,q} (\bullet) \simeq \bigoplus_{r=0}^{\infty} \wh{A \langle \ul{\xi} \rangle}_{(m), q, \ul{x}}^{\wh{\otimes}' r}
\]
of graded $A$-algebras.
\end{proof}

For $r \in \N$, let $\delta^r_i : P^{(m)}_{A/R,q} (r) \to P^{(m)}_{A/R,q} (r+1) \ (0 \le i \le r+1)$ denote the map of $\delta$-$R$-algebras corresponding to the strictly increasing map of simplices \[ [r] := \{ 0, 1, \dots, r \} \to [r+1] \] that skips $i$, i.e., $\delta^r_i$ is induced by the map $A^{\wh{\otimes} (r+1)} \to A^{\wh{\otimes} (r+2)}$ given by
\[
a_0 \otimes \cdots \otimes a_r \mapsto a_0 \otimes \cdots \otimes a_{i-1} \otimes 1 \otimes a_i \otimes \cdots \otimes a_r.
\]
Then, we can consider the differential morphism $d^r: P^{(m)}_{A/R,q} (r) \to P^{(m)}_{A/R,q} (r+1)$ by 
\begin{equation}
\label{dmqd}
d^r = \sum_{i=0}^{r+1} (-1)^i \delta^r_i,
\end{equation}
and this makes $P^{(m)}_{A/R,q} (\bullet)$ a DGA (differential graded algebra) over $R$.

For $r \in \N$, let $\sigma^r_i : P^{(m)}_{A/R,q} (r) \to P^{(m)}_{A/R,q} (r-1) \ (0 \le i \le r-1)$ denote the map of $\delta$-$R$-algebras corresponding to the surjective nondecreasing map of simplices \[ [r] := \{ 0, 1, \dots, r \} \to [r-1] \] such that $\{ i \} \subset [r-1]$ has two elements, $i$ and $i+1$, in its preimage, i.e., $\sigma^r_i$ is induced by the map $A^{\wh{\otimes} (r+1)} \to A^{\wh{\otimes} (r)}$ given by
\[
a_0 \otimes \cdots \otimes a_r \mapsto a_0 \otimes \cdots \otimes a_{i-1} \otimes a_i a_{i+1} \otimes a_{i+2} \otimes \cdots \otimes a_r.
\]
Then, we can define the sub-DGA $NP^{(m)}_{A/R,q} (\bullet)$ of $P^{(m)}_{A/R,q} (\bullet)$ by
\[
NP^{(m)}_{A/R,q} (r) := \bigcap_{i=0}^{r-1} \Ker (\sigma^r_i).
\]
It is easy to see that $NP^{(m)}_{A/R,q} (1)$ is the $(p, (p)_{q^{p^m}})$-completion of the $A$-module freely generated by $\{ \ul{\xi}^{ \{ \ul{k} \}_{(m), q}} \}_{\ul{k} \in \N^d \backslash \{ 0 \} }$, and the DGA $NP^{(m)}_{A/R,q} (\bullet)$ is isomorphic to the complete tensor algebra of $NP^{(m)}_{A/R,q} (1)$ over $A$, namely, we have
\[
NP^{(m)}_{A/R,q} (\bullet) = \bigoplus_{r=0}^{\infty} NP^{(m)}_{A/R,q} (1)^{\wh{\otimes}' r}.
\]

Next, we make some explicit computations related to the above constructions.

\begin{thm}
\label{mqbt}
For all $\ul{k} \in \N^d$, the morphism $\delta_1^1 : P^{(m)}_{A/R,q} (1) \to P^{(m)}_{A/R,q} (2)$ satisfies
\[
\delta_1^1(\ul{\xi}^{ \{ \ul{k} \}_{(m), q}}) = \sum_{0 \le \ul{k}' \le \ul{k}} \left\langle 
\begin{array}{@{}c@{}}
\ul{k} \\ \ul{k}'
\end{array}
\right\rangle_{(m), q }
\ul{\xi}^{ \{ \ul{k}' \}_{(m), q}} \otimes' \ul{\xi}^{ \{ \ul{k} - \ul{k}'  \}_{(m), q}}.
\]
\end{thm}
\begin{proof}
As in the proof of Proposition \ref{mr}, we may assume that $R=\Q(q)$ and $A = \Q(q)[\ul{x}]$.
First we consider the elements $\xi_i := 1 \otimes x_i - x_i \otimes 1 \in P$ for $i = 1, \dots, d$ (notice that these would have been called $-\tau_i$ in \cite{Miy15}, and therefore the signs in our formulas will not always agree with those given by Miyatani) and the map
\[
\delta_1^1 : P \to P \wh{\otimes}'_A P, \ \ f \otimes g \mapsto f \otimes 1 \otimes g.
\]
By Theorem 3.5 of \cite{LQ18}, we have
\[
\delta_1^1(\ul{\xi}^{ ( \ul{k} )_{q}}) = \sum_{0 \le \ul{k}' \le \ul{k}} \begin{pmatrix} \ul{k} \\ \ul{k}' \end{pmatrix}_{ q }
\ul{\xi}^{ ( \ul{k}' )_{q}} \otimes' \ul{\xi}^{ ( \ul{k} - \ul{k}'  )_{q}}.
\]
By considering the images of these elements in $\wh{A \langle \ul{\xi} \rangle}_{(m), q, \ul{x}}$ via the map \[ \wt{\theta}_{(m)} : P \to A [\ul{\xi}]\left[\frac{\phi^{m+1}(\ul{\xi})}{(p)_{q^{p^m}}}\right]^{\delta, \wedge} \simeq \wh{A \langle \ul{\xi} \rangle}_{(m), q, \ul{x}} \] before Lemma \ref{qpdf}, we get the equation
\begin{align*}
&\delta_1^1 \left( \left( \prod_{i=1}^d \left(\left\lfloor \frac{k_i}{p^m} \right\rfloor\right)_{q^{p^m}}! \right) \ul{\xi}^{ \{ \ul{k} \}_{(m), q}} \right)  \\ 
= &\sum_{0 \le \ul{k}' \le \ul{k}} \begin{pmatrix} \ul{k} \\ \ul{k}' \end{pmatrix}_{ q }
 \left( \prod_{i=1}^d \left(\left\lfloor \frac{k'_i}{p^m} \right\rfloor\right)_{q^{p^m}}! \right) \ul{\xi}^{ \{ \ul{k}' \}_{(m), q}} \otimes' \left( \prod_{i=1}^d \left(\left\lfloor \frac{k_i-k'_i}{p^m} \right\rfloor\right)_{q^{p^m}}! \right) \ul{\xi}^{ \{ \ul{k} - \ul{k}'  \}_{(m), q}},
\end{align*}
so the result follows.
\end{proof}

\begin{prop}
{\rm (cf. \cite[Proposition 2.18]{Miy15})}
\label{mqbc}
For all $\ul{k} \in \N^d$, the morphism $d^1 : P^{(m)}_{A/R,q} (1) \to P^{(m)}_{A/R,q} (2)$ satisfies
\[
d^1(\ul{\xi}^{ \{ \ul{k} \}_{(m), q}}) = - \sum_{0 < \ul{k}' < \ul{k}} \left\langle 
\begin{array}{@{}c@{}}
\ul{k} \\ \ul{k}'
\end{array}
\right\rangle_{(m), q }
\ul{\xi}^{ \{ \ul{k}' \}_{(m), q}} \otimes' \ul{\xi}^{ \{ \ul{k} - \ul{k}'  \}_{(m), q}}.
\]
\end{prop}
\begin{proof}
By definition, the morphism $d^1$ is equal to $\delta^1_0 - \delta^1_1 +\delta^1_2$. It is easy to see that
\[
\delta^1_0 (\ul{\xi}^{ \{ \ul{k} \}_{(m), q}}) = 1 \otimes' \ul{\xi}^{ \{ \ul{k} \}_{(m), q}} \ \ \text{and} \ \ \delta^1_2 (\ul{\xi}^{ \{ \ul{k} \}_{(m), q}}) = \ul{\xi}^{ \{ \ul{k} \}_{(m), q}} \otimes' 1.
\]
Then the result follows from Theorem \ref{mqbt}.
\end{proof}

Now, we calculate the hyper $m$-$q^{p^m}$-stratification on $L^{(m)}(M)_{A}$, which exists by Lemma \ref{lmmb} and Proposition \ref{crstr}. Note that we consider the left $A$-module structure on $L^{(m)}(M)_{A} = \wh{A \langle \ul{\xi} \rangle}_{(m), q, \ul{x}} \wh{\otimes}_A' M$.

\begin{lem}
\label{strind}
{\rm (cf. \cite[Lemma 2.16]{Miy15})}
Let $M \in {\bf FMod}^{\wedge}_{(p,(p)_{q^{p^m}})}(A)$. Then, the hyper $m$-$q^{p^m}$-stratification on $L^{(m)}(M)_{A} = \wh{A \langle \ul{\xi} \rangle}_{(m), q, \ul{x}} \wh{\otimes}_A' M$ is induced by
\begin{align*}
P \wh{\otimes}'_A (P \wh{\otimes}'_A M) &\to (P \wh{\otimes}'_A M) \wh{\otimes}_A P \\
(a \otimes b) \otimes' ((f \otimes g) \otimes' h) &\mapsto ((1 \otimes g) \otimes' h) \otimes (a \otimes bf).
\end{align*}
\end{lem}
\begin{proof}
By the construction of the hyper $m$-$q^{p^m}$-stratification and Lemma \ref{lmmb}, we see that the hyper $m$-$q^{p^m}$-stratification on $L^{(m)}(M)_{A}$ is the composite of the isomorphism
\[
\wh{A \langle \ul{\xi} \rangle}_{(m), q, \ul{x}} \wh{\otimes}'_A (\wh{A \langle \ul{\xi} \rangle}_{(m), q, \ul{x}} \wh{\otimes}'_A M) \to P^{(m)}_{A/R,q} (2) \wh{\otimes}'_A M
\]
which is induced by
$(a \otimes b) \otimes' ((f \otimes g) \otimes' h) \mapsto (a \otimes bf \otimes g) \otimes' h$,
and the inverse of the isomorphism
\[
(\wh{A \langle \ul{\xi} \rangle}_{(m), q, \ul{x}} \wh{\otimes}'_A M) \wh{\otimes}_A \wh{A \langle \ul{\xi} \rangle}_{(m), q, \ul{x}} \to P^{(m)}_{A/R,q} (2) \wh{\otimes}'_A M
\]
which is induced by
$((f \otimes g) \otimes' h)  \otimes (a \otimes b) \mapsto (af \otimes b \otimes g) \otimes' h$.
Then, the assertion follows from the direct computation.
\end{proof}

In order to make explicit calculations later, we will need the following technical map. It is possible to give an explicit formula for this map as in Proposition 5.4 of \cite{GLQ23b}, but we omit it because we will not need it in this article:

\begin{defi}
The $R$-linear \textit{flip map} $\tau : \wh{A \langle \ul{\xi} \rangle}_{(m), q, \ul{x}} \to \wh{A \langle \ul{\xi} \rangle}_{(m), q, \ul{x}}$ is the $R$-algebra automorphism of $\wh{A \langle \ul{\xi} \rangle}_{(m), q, \ul{x}}$ induced by the map 
\begin{align*}
A \wh{\otimes}_R A \to A \wh{\otimes}_R A , \ \ f \otimes g \mapsto g \otimes f.
\end{align*}
\end{defi}

Note that the flip map $\tau$ is the unique $R$-algebra automorphism of $\wh{A \langle \ul{\xi} \rangle}_{(m), q, \ul{x}}$ that satisfies the following conditions:
\begin{align*}
\begin{cases}
f \mapsto \theta_{(m)}(f) \ \ \ \text{for} \ f \in A, \\
\xi \mapsto -\xi.
\end{cases}
\end{align*}
Then, the explicit formula for the hyper $m$-$q^{p^m}$-stratification on $\wh{A \langle \ul{\xi} \rangle}_{(m), q, \ul{x}} = L^{(m)}(A)_{A}$ is given in the following proposition.

\begin{prop}
\label{strao}
{\rm (cf. \cite[Proposition 2.19]{Miy15})}
The hyper $m$-$q^{p^m}$-stratification
\[
\epsilon : \wh{A \langle \ul{\xi} \rangle}_{(m), q, \ul{x}} \wh{\otimes}'_A \wh{A \langle \ul{\xi} \rangle}_{(m), q, \ul{x}} \to \wh{A \langle \ul{\xi} \rangle}_{(m), q, \ul{x}} \wh{\otimes}_A \wh{A \langle \ul{\xi} \rangle}_{(m), q, \ul{x}}
\]
of $\wh{A \langle \ul{\xi} \rangle}_{(m), q, \ul{x}}$ maps $1 \otimes' \ul{\xi}^{ \{ \ul{k} \}_{(m), q}}$ to
\[
\sum_{0 \le \ul{k}' \le \ul{k}} \left\langle 
\begin{array}{@{}c@{}}
\ul{k} \\ \ul{k}'
\end{array}
\right\rangle_{(m), q }
\ul{\xi}^{ \{ \ul{k}' \}_{(m), q}} \otimes \tau(\ul{\xi}^{ \{ \ul{k} - \ul{k}'  \}_{(m), q}}).
\]
\end{prop}
\begin{proof}
By Lemma \ref{strind}, the hyper $m$-$q^{p^m}$-stratification is induced by
\[
\epsilon' : P \wh{\otimes}'_A P \to P \wh{\otimes}_A P, \ \ (a \otimes b) \otimes' (f \otimes g) \mapsto (1 \otimes g) \otimes (a \otimes bf).
\]
For later computations, we use the map
\[
\tau'' : \wh{A \langle \ul{\xi} \rangle}_{(m), q, \ul{x}} \wh{\otimes}_A \wh{A \langle \ul{\xi} \rangle}_{(m), q, \ul{x}} \to \wh{A \langle \ul{\xi} \rangle}_{(m), q, \ul{x}} \wh{\otimes}'_A \wh{A \langle \ul{\xi} \rangle}_{(m), q, \ul{x}}
\]
induced by the map
\[
\tau' : P \wh{\otimes}_A P \to P \wh{\otimes}'_A P, \ \ (a \otimes b) \otimes (f \otimes g) \mapsto (g \otimes af) \otimes' (1 \otimes b).
\]
For elements $y,z \in \wh{A \langle \ul{\xi} \rangle}_{(m), q, \ul{x}}$, it is easy to see that $(\tau'')^{-1}$ is given by $y \otimes' z \mapsto z \otimes \tau(y)$. If we define the map $p_2$ by
\[
p_2 : P \to P \wh{\otimes}'_A P, \ \ (a \otimes b) \mapsto (1 \otimes 1) \otimes' (a \otimes b),
\]
then we can check that $\tau' \circ \epsilon' \circ p_2$ is equal to $\delta^1_1$. So, by Theorem \ref{mqbt}, we see that 
\[
(\tau'' \circ \epsilon)(1 \otimes' \ul{\xi}^{ \{ \ul{k} \}_{(m), q}}) = \sum_{0 \le \ul{k}' \le \ul{k}} \left\langle 
\begin{array}{@{}c@{}}
\ul{k} \\ \ul{k}'
\end{array}
\right\rangle_{(m), q }
\ul{\xi}^{ \{ \ul{k} - \ul{k}'  \}_{(m), q}} \otimes' \ul{\xi}^{ \{ \ul{k}' \}_{(m), q}}.
\]
Applying the map $(\tau'')^{-1}$, we get the result.
\end{proof}

Based on the explicit description of linearizations in Lemma \ref{lmmb}, we can also construct a DGA $LP^{(m)}_{A/R,q} (\bullet) := P^{(m)}_{A/R,q} (\bullet+1)$ over $A$ and its sub-DGA $LNP^{(m)}_{A/R,q} (\bullet)$ by
\[
LNP^{(m)}_{A/R,q} (r) := \bigcap_{i=1}^{r} \Ker (\sigma^{r+1}_i).
\]
Here, the differential morphism $d^r : LP^{(m)}_{A/R,q} (r) \to LP^{(m)}_{A/R,q} (r+1)$ is defined by
\begin{equation}
\label{lpdr}
d^r = \sum_{i=1}^{r+2} (-1)^{i+1} \delta^{r+1}_i.
\end{equation}
Note that in these definitions, we exclude the map $\sigma^{r+1}_0$ and the map $\delta^{r+1}_0$ respectively.

By  Lemma \ref{lmmb} and Proposition \ref{crstr}, each $LP^{(m)}_{A/R,q} (r) = L^{(m)}(P^{(m)}_{A/R,q} (r))_A$ for $r \in \N$ has a hyper $m$-$q^{p^m}$-stratification.

\begin{lem}
\label{lpstr}
{\rm (cf. \cite[Lemma 2.20]{Miy15})}
For each $r \in \N$, the differential morphism $d^r : LP^{(m)}_{A/R,q} (r) \to LP^{(m)}_{A/R,q} (r+1)$ is compatible with the hyper $m$-$q^{p^m}$-stratification on both sides.
\end{lem}
\begin{proof}
By Lemma \ref{strind}, the hyper $m$-$q^{p^m}$-stratification on $LP^{(m)}_{A/R,q} (r)$ is induced by
\begin{align*}
P \wh{\otimes}'_A A^{\wh{\otimes} (r+2)} &\to A^{\wh{\otimes} (r+2)} \wh{\otimes}_A P \\
(a \otimes b) \otimes' (f \otimes g \otimes h_1 \otimes \cdots \otimes h_r)
 &\mapsto (1 \otimes g \otimes h_1 \otimes \cdots \otimes h_r) \otimes (a \otimes bf).
\end{align*}
We can check that the hyper $m$-$q^{p^m}$-stratification is compatible with the morphism $\delta^{r+1}_i$ for $i = 1, \dots , r+2$. Therefore, the assertion follows from the definition (\ref{lpdr}) of $d^r$.
\end{proof}

Now, we introduce the higher $q$-de Rham complex. We denote by $\textbf{1}_i$ the element $(0, \dots, 0, 1, 0, \dots, 0)$ in $\N^d$, where 1 sits in the $i$-th entry. Let $\breve{K}^{(m)}_{A/R,q} (\bullet)$ be the DG-ideal of $NP^{(m)}_{A/R,q} (\bullet)$ generated by the $(p,(p)_{q^{p^m}})$-completion of the free $A$-module with basis $\ul{\xi}^{ \{ \ul{k} \}_{(m), q}}$, where $\ul{k} \in \N^d \backslash \{ 0, p^m \textbf{1}_1, \dots, p^m \textbf{1}_d \} $. We can check that the ideal $\wh{A \langle \ul{\xi} \rangle}_{(m), q, \ul{x}} \wh{\otimes}'_A \breve{K}^{(m)}_{A/R,q} (\bullet) := \bigoplus_{r=0}^{\infty} \wh{A \langle \ul{\xi} \rangle}_{(m), q, \ul{x}} \wh{\otimes}'_A \breve{K}^{(m)}_{A/R,q} (r)$ of $LNP^{(m)}_{A/R,q} (\bullet)$ is also a DG-ideal.

\begin{defi}
\label{mqdr}
We define the \textit{higher $q$-de Rham complex} $\breve{\Omega}^{(m)}_{A/R,q}  (\bullet)$ as the quotient of $NP^{(m)}_{A/R,q} (\bullet)$ by the DG-ideal $\breve{K}^{(m)}_{A/R,q} (\bullet)$, and define the \textit{linearized higher $q$-de Rham complex} $L \breve{\Omega}^{(m)}_{A/R,q}  (\bullet)$ as the quotient of $LNP^{(m)}_{A/R,q} (\bullet)$ by the DG-ideal $\wh{A \langle \ul{\xi} \rangle}_{(m), q, \ul{x}} \wh{\otimes}'_A \breve{K}^{(m)}_{A/R,q} (\bullet)$.
\end{defi}

Recall that, for $0 \le k < 2p^m$, we have $\xi_i^{ \{ k \}_{(m),q}} = \xi_i^{(k)_q}$ via the map (\ref{filma}). So we shall often denote the images of $\xi_i^{ \{ p^m \}_{(m),q}}$ via the natural surjection \[ NP^{(m)}_{A/R,q} (1) \to \breve{\Omega}^{(m)}_{A/R,q}  (1) \] by $\ol{\xi}_i^{(p^m)_q}$ for $i = 1, \dots, d$ if no confusion arises.

\begin{prop}
\label{drmo}
{\rm (cf. \cite[Proposition 3.2]{Miy15})}
\begin{enumerate}
\item $\breve{\Omega}^{(m)}_{A/R,q}  (1)$ is a free $A$-module of rank $d$ with basis $\{ \ol{\xi}_i^{(p^m)_q} \}_{i=1, \dots, d}$.
\item $\breve{\Omega}^{(m)}_{A/R,q}  (\bullet)$ is isomorphic to the exterior algebra of $\breve{\Omega}^{(m)}_{A/R,q}  (1)$ as a graded $A$-module.
\item The differential map $d^r$ sends the section $\ul{\xi}^{ \{ \ul{k} \}_{(m), q}} \otimes' \ol{\xi}_{i_1}^{(p^m)_q} \wedge' \cdots \wedge' \ol{\xi}_{i_r}^{(p^m)_q}$ of $L \breve{\Omega}^{(m)}_{A/R,q}  (r)$, where $\ul{k} = (k_1, \dots, k_d) \in \N^d$, to the section
\[
\sum_{\substack{i=1, \dots, d \\ k_i \ge p^m}}  \left\langle 
\begin{array}{@{}c@{}}
k_i \\ p^m
\end{array}
\right\rangle_{(m), q }
\ul{\xi}^{ \{ \ul{k} - p^m {\bf 1}_i \}_{(m), q}} \otimes' \ol{\xi}_i^{(p^m)_q} \wedge' \ol{\xi}_{i_1}^{(p^m)_q} \wedge' \cdots \wedge' \ol{\xi}_{i_r}^{(p^m)_q}.
\]
\end{enumerate}
\end{prop}
\begin{proof}
We follow the proof of \cite{Miy15}. By Definition \ref{mqdr} and Proposition \ref{mqbc}, as an $A$-module, $\breve{\Omega}^{(m)}_{A/R,q}  (r)$ is generated by the elements 
\[
\ol{\xi}_{i_1}^{(p^m)_q} \otimes' \ol{\xi}_{i_2}^{(p^m)_q} \otimes' \cdots \otimes' \ol{\xi}_{i_r}^{(p^m)_q},
\]
subject to the relations generated by 
\[
0 =\ol{d^1(\ul{\xi}^{ \{ \ul{k} \}_{(m), q}})} = - \sum_{\substack{0 < \ul{k}' < \ul{k} \\ \ul{k}' = p^m \textbf{1}_i \\ \ul{k} - \ul{k}' = p^m \textbf{1}_j}} \left\langle 
\begin{array}{@{}c@{}}
\ul{k} \\ \ul{k}'
\end{array}
\right\rangle_{(m), q }
\ol{\xi}_{i}^{(p^m)_q} \otimes' \ol{\xi}_{j}^{(p^m)_q}
\]
for $\ul{k} \in \N^d \backslash \{ 0, p^m \textbf{1}_1, \dots, p^m \textbf{1}_d \}$.
The above relations are trivial except for the cases $\ul{k} = 2p^m \textbf{1}_i$ and the cases $\ul{k} = p^m \textbf{1}_i + p^m \textbf{1}_j$ with $i \ne j$.

For the cases $\ul{k} = 2p^m \textbf{1}_i$, the above relations give
\[
\ol{\xi}_{i}^{(p^m)_q} \otimes' \ol{\xi}_{i}^{(p^m)_q} = 0
\]
because the coefficient $\left\langle 
\begin{array}{@{}c@{}}
2p^m \\ p^m
\end{array}
\right\rangle_{(m), q } $ is invertible in $\Z [ q ]_{(p,q-1)}$ by Lemma \ref{ipmq} below. 

For the cases $\ul{k} = p^m \textbf{1}_i + p^m \textbf{1}_j$ with $i \ne j$, the above relations give
\[
\ol{\xi}_{i}^{(p^m)_q} \otimes' \ol{\xi}_{j}^{(p^m)_q} + \ol{\xi}_{j}^{(p^m)_q} \otimes' \ol{\xi}_{i}^{(p^m)_q} = 0,
\]
so the first and the second assertions follow.  Now, the third assertion follows from (\ref{lpdr}) and the formula of Theorem \ref{mqbt}. 
\end{proof}

\begin{lem}
\label{ipmq}
If $k \ge p^m$, then the element $\left\langle 
\begin{array}{@{}c@{}}
k \\ p^m
\end{array}
\right\rangle_{(m), q } $ in $\Z [ q ]_{(p,q-1)}$ is congruent to $1$ modulo $(p, q-1)$.
\end{lem}
\begin{proof}
The element $\left\langle 
\begin{array}{@{}c@{}}
k \\ p^m
\end{array}
\right\rangle_{(m), q } $ is congruent to $\left\langle 
\begin{array}{@{}c@{}}
k \\ p^m
\end{array}
\right\rangle_{(m)} $ modulo $q-1$. By Lemma 3.3 of \cite{Miy15}, the element $\left\langle 
\begin{array}{@{}c@{}}
k \\ p^m
\end{array}
\right\rangle_{(m)} $ is congruent to 1 modulo $p$, so we get the result.
\end{proof}

We prove now the formal Poincar\'{e} lemma for the linearized higher $q$-de Rham complex $L \breve{\Omega}^{(m)}_{A/R,q}  (\bullet)$.

\begin{lem}
\label{fplq}
{\rm (cf. \cite[Lemma 3.4]{Miy15})}
The linearized higher $q$-de Rham complex $L \breve{\Omega}^{(m)}_{A/R,q}  (\bullet)$ is a resolution of the direct sum of $p^{md}$ copies of $A$, i.e., the  $A$-linear map
\begin{align*}
\iota': \bigoplus_{\ul{k} \in \cB^{(m)}_d} A e_{\ul{k}} \to L \breve{\Omega}^{(m)}_{A/R,q}  (\bullet) , \ e_{\ul{k}} \mapsto \ul{\xi}^{ \{ \ul{k}\}_{(m), q}}
\end{align*}
is a quasi-isomorphism, where $\cB^{(m)}_d$ denotes the set of elements $\ul{k} = (k_1, \dots, k_d) \in \N^d$ such that $0 \le k_i < p^m$ for $i = 1, \dots, d$.
\end{lem}
\begin{proof}
We follow the proof of \cite{Miy15}. We first consider the case $d = 1$. Then we need to show that the sequence
\[
0 \to \bigoplus_{k_1 = 0}^{p^m-1} A e_{k_1} \to \wh{A \langle \xi_1 \rangle}_{(m), q, x_1} \xrightarrow{d^0} \wh{A \langle \xi_1 \rangle}_{(m), q, x_1} \ol{\xi}_{1}^{(p^m)_q} \to 0
\]
is exact. We can check that 
\[
d^0(\xi_1^{ \{ k \}_{(m), q}}) = 
\begin{cases}
0 \ \ \ \  \text{if} \ k < p^m, \\
\\
\left\langle 
\begin{array}{@{}c@{}}
k \\ p^m
\end{array}
\right\rangle_{(m), q } \xi_1^{ \{ k-p^m \}_{(m), q}} \otimes' \ol{\xi}_{1}^{(p^m)_q}
\ \ \ \  \text{if} \ k \ge p^m.
\end{cases}
\] 
Since $\left\langle 
\begin{array}{@{}c@{}}
k \\ p^m
\end{array}
\right\rangle_{(m), q }$ is invertible in $A$ by Lemma \ref{ipmq}, the desired exactness follows.

Now we consider a general $d$. In this case, the morphism $\iota'$ is the complete tensor product of that for the $(d-1)$-dimensional case and that for the $1$-dimensional case. Since each term of these complexes belongs to ${\bf FMod}^{\wedge}_{(p,(p)_{q^{p^m}})}(A)$, the desired exactness follows by induction on $d$.
\end{proof}

For arguments later, we need to slightly modify the map $\iota'$ in Lemma \ref{fplq}:

\begin{prop}
\label{tfpl}
{\rm (cf. \cite[Proposition 3.5]{Miy15})}
Let
\[
(- \ul{x})^{{(\ul{k})}_q} := \prod_{i=1}^d \prod_{j=0}^{k_i-1} (-q^j x_i).
\]
Using the isomorphism
\[
\beta : \bigoplus_{\ul{k} \in \cB^{(m)}_d} A e_{\ul{k}} \to \bigoplus_{\ul{k} \in \cB^{(m)}_d} A e_{\ul{k}}, \ e_{\ul{k}} \mapsto \sum_{0 \le \ul{k}' \le \ul{k}} \begin{pmatrix} \ul{k} \\ \ul{k}' \end{pmatrix}_{ q }(- \ul{x})^{{(\ul{k}')}_q} \cdot e_{\ul{k} - \ul{k}'},
\]
we define the morphism
\[
\iota := \iota' \circ \beta^{-1} : \bigoplus_{\ul{k} \in \cB^{(m)}_d} A e_{\ul{k}} \to L \breve{\Omega}^{(m)}_{A/R,q} (\bullet).
\]
Then, this is a quasi-isomorphism.
\end{prop}
\begin{proof}
This is a direct consequence of Lemma \ref{fplq} since $\beta$ is an isomorphism.
\end{proof}

Next, we consider the corresponding results for the $m$-$q^{p^m}$-crystalline site. Let
\[
F_A := \bigoplus_{\ul{k} \in \cB^{(m)}_d} A e_{\ul{k}}.
\]
We endow it with the hyper $m$-$q^{p^m}$-stratification
\[
\epsilon_F : \wh{A \langle \ul{\xi} \rangle}_{(m), q, \ul{x}} \wh{\otimes}'_A F_A \simeq F_A \wh{\otimes}_A \wh{A \langle \ul{\xi} \rangle}_{(m), q, \ul{x}}, \ y \otimes' a e_{\ul{k}} \mapsto e_{\ul{k}} \otimes \theta_{(m)}(a) y .
\]
By Proposition \ref{crstr}, the $A$-module $F_A$ endowed with its hyper $m$-$q^{p^m}$-stratification $\epsilon_F$ uniquely corresponds to $F \in {\bf CR}((\ol{A}/R)_{m\text{-}q^{p^m}\text{-crys}},\cO^{(m)}_{\ol{A}/R})$. The complete crystal $F$ is isomorphic to the direct sum of $p^{md}$ copies of the structure sheaf $\cO^{(m)}_{\ol{A}/R}$.

\begin{lem}
\label{plst}
{\rm (cf. \cite[Lemma 3.6]{Miy15})}
The $A$-linear morphism \[ \iota: F_A \to L \breve{\Omega}^{(m)}_{A/R,q} (0) = \wh{A \langle \ul{\xi} \rangle}_{(m), q, \ul{x}} \] defined in Proposition \ref{tfpl} is compatible with the hyper $m$-$q^{p^m}$-stratification on both sides.
\end{lem}
\begin{proof}
We follow the proof of \cite{Miy15} although the computation here is more complicated.
We need to show that the diagram
\[
\begin{tikzcd}
\wh{A \langle \ul{\xi} \rangle}_{(m), q, \ul{x}} \wh{\otimes}'_A F_A \arrow[r,"1 \otimes' \iota"]\arrow[d,"\epsilon_F"] & \wh{A \langle \ul{\xi} \rangle}_{(m), q, \ul{x}} \wh{\otimes}'_A \wh{A \langle \ul{\xi} \rangle}_{(m), q, \ul{x}} \arrow[d,"\epsilon"] \\
F_A \wh{\otimes}_A \wh{A \langle \ul{\xi} \rangle}_{(m), q, \ul{x}} \arrow[r,"\iota \otimes 1"] & \wh{A \langle \ul{\xi} \rangle}_{(m), q, \ul{x}} \wh{\otimes}_A \wh{A \langle \ul{\xi} \rangle}_{(m), q, \ul{x}}
\end{tikzcd}
\]
is commutative, where the explicit formula for $\epsilon$ is given in Proposition \ref{strao}.
We consider the basis of $\wh{A \langle \ul{\xi} \rangle}_{(m), q, \ul{x}} \wh{\otimes}'_A F_A$ as an $\wh{A \langle \ul{\xi} \rangle}_{(m), q, \ul{x}}$-module consisting of elements
\begin{equation}
\label{sstr}
1 \otimes' \left( \sum_{0 \le \ul{k}' \le \ul{k}} \begin{pmatrix} \ul{k} \\ \ul{k}' \end{pmatrix}_{ q } (- \ul{x})^{{(\ul{k}')}_q} \cdot e_{\ul{k} - \ul{k}'} \right)
\end{equation}
indexed by $\ul{k} \in \cB^{(m)}_d$.
First note that, for $\ul{k},\ul{k}' \in \cB^{(m)}_d$ such that $\ul{k}' \le \ul{k}$, since $k'_i \le k_i < p^m$ for $i = 1, \dots, d$, we have $\left\langle
\begin{array}{@{}c@{}}
\ul{k} \\ \ul{k}'
\end{array}
\right\rangle_{(m), q } = \begin{pmatrix} \ul{k} \\ \ul{k}' \end{pmatrix}_{ q }$. Now, for $\ul{k} \in \cB^{(m)}_d$, the element (\ref{sstr}) is sent by $1 \otimes' \iota$ to $1 \otimes' \ul{\xi}^{ \{ \ul{k}\}_{(m), q}}$, and its image under $\epsilon$ is
\[
\sum_{0 \le \ul{k}' \le \ul{k}} \begin{pmatrix} \ul{k} \\ \ul{k}' \end{pmatrix}_{ q }
\ul{\xi}^{ \{ \ul{k}' \}_{(m), q}} \otimes \tau(\ul{\xi}^{ \{ \ul{k} - \ul{k}' \}_{(m), q}}),
\]
which is the image of
\begin{equation}
\label{ssttrr}
\sum_{0 \le \ul{k}' \le \ul{k}} \begin{pmatrix} \ul{k} \\ \ul{k}' \end{pmatrix}_{ q }
\left( \sum_{0 \le \ul{k}'' \le \ul{k}'} \begin{pmatrix} \ul{k}' \\ \ul{k}'' \end{pmatrix}_{ q } (- \ul{x})^{{(\ul{k}'')}_q} \cdot e_{\ul{k}' - \ul{k}''} \right) \otimes \tau(\ul{\xi}^{ \{ \ul{k} - \ul{k}' \}_{(m), q}})
\end{equation}
under $\iota \otimes 1$.
On the other hand, the image of (\ref{sstr}) under $\epsilon_F$ is the element
\begin{equation}
\label{bstr}
\sum_{0 \le \ul{k}' \le \ul{k}} \begin{pmatrix} \ul{k} \\ \ul{k}' \end{pmatrix}_{ q } e_{\ul{k} - \ul{k}'} \otimes \theta_{(m)}((- \ul{x})^{{(\ul{k}')}_q}).
\end{equation}
Therefore, we are reduced to proving that the two elements (\ref{ssttrr}) and (\ref{bstr}) are identical.

The element (\ref{ssttrr}) is equal to
\begin{equation}
\label{cstr}
\sum_{0 \le \ul{l} \le \ul{k}} \sum_{\ul{k} - \ul{l} \le \ul{k}' \le \ul{k}} \begin{pmatrix} \ul{k} \\ \ul{k}' \end{pmatrix}_{ q } \begin{pmatrix} \ul{k}' \\ \ul{l}-\ul{k}+\ul{k}' \end{pmatrix}_{ q } (- \ul{x})^{{(\ul{l}-\ul{k}+\ul{k}')}_q} \cdot e_{\ul{k} - \ul{l}} \otimes \tau(\ul{\xi}^{ \{ \ul{k} - \ul{k}' \}_{(m), q}})
\end{equation}
by the change of variables $\ul{l} = \ul{k} - \ul{k}' + \ul{k}''$. We can check that
\[
\begin{pmatrix} \ul{k} \\ \ul{k}' \end{pmatrix}_{ q } \begin{pmatrix} \ul{k}' \\ \ul{l}-\ul{k}+\ul{k}' \end{pmatrix}_{ q } = \begin{pmatrix} \ul{k} \\ \ul{l} \end{pmatrix}_{ q } \begin{pmatrix} \ul{l} \\ \ul{k}-\ul{k}' \end{pmatrix}_{ q }.
\]
So, the element (\ref{cstr}) is equal to
\begin{equation}
\label{dstr}
\sum_{0 \le \ul{l} \le \ul{k}} \begin{pmatrix} \ul{k} \\ \ul{l} \end{pmatrix}_{ q } e_{\ul{k} - \ul{l}} \otimes \left( \sum_{\ul{k} - \ul{l} \le \ul{k}' \le \ul{k}} \begin{pmatrix} \ul{l} \\ \ul{k}-\ul{k}' \end{pmatrix}_{ q } (- \ul{x})^{{(\ul{l}-\ul{k}+\ul{k}')}_q} \cdot \tau(\ul{\xi}^{ \{ \ul{k} - \ul{k}' \}_{(m), q}}) \right).
\end{equation}
By the change of variables $\ul{l}' = \ul{k} - \ul{k'}$, the element of $\wh{A \langle \ul{\xi} \rangle}_{(m), q, \ul{x}}$ in the brackets of (\ref{dstr}) is equal to
\begin{equation}
\label{estr}
\sum_{0 \le \ul{l}' \le \ul{l}} \begin{pmatrix} \ul{l} \\ \ul{l}' \end{pmatrix}_{ q } (- \ul{x})^{{(\ul{l}-\ul{l}')}_q} \cdot \tau(\ul{\xi}^{ \{ \ul{l}' \}_{(m), q}}).
\end{equation}
Comparing with (\ref{bstr}), we are reduced to proving that the element (\ref{estr}) and the element
\begin{equation}
\label{fstr}
\theta_{(m)}((- \ul{x})^{{(\ul{l})}_q})
\end{equation}
are identical. First, as elements of $P$, if we set $a= -1 \otimes x_i, \ b= x_i \otimes 1$ for $i = 1, \dots, d$ in Proposition \ref{qbfm} below, we see that
\begin{equation}
\label{gstr}
\prod_{i=1}^d \prod_{j=0}^{l'_i-1}(x_i \otimes 1 - 1 \otimes q^j x_i) = \sum_{0 \le \ul{l}'' \le \ul{l}'} \begin{pmatrix} \ul{l}' \\ \ul{l}'' \end{pmatrix}_{ q } (1 \otimes (- \ul{x})^{{(\ul{l}'')}_q})(\ul{x} \otimes 1)^{\ul{l}'-\ul{l}''}.
\end{equation}
By considering the images of these elements in $\wh{A \langle \ul{\xi} \rangle}_{(m), q, \ul{x}}$ via the map \[ \wt{\theta}_{(m)} : P \to \wh{A \langle \ul{\xi} \rangle}_{(m), q, \ul{x}} \] before Lemma \ref{qpdf}, we get the equation
\begin{equation}
\label{hstr}
\tau(\ul{\xi}^{ \{ \ul{l}' \}_{(m), q}}) = \tau \left( \left( \prod_{i=1}^d \left(\left\lfloor \frac{l'_i}{p^m} \right\rfloor\right)_{q^{p^m}}! \right) \ul{\xi}^{ \{ \ul{l}' \}_{(m), q}} \right) = \sum_{0 \le \ul{l}'' \le \ul{l}'} \begin{pmatrix} \ul{l}' \\ \ul{l}'' \end{pmatrix}_{ q } \theta_{(m)}((- \ul{x})^{{(\ul{l}'')}_q})\ul{x} ^{\ul{l}'-\ul{l}''},
\end{equation}
where the first equality holds because $l'_i \le l_i \le k_i < p^m$.
By applying it to (\ref{estr}), we get the element
\begin{align}
\label{istr}
&\sum_{0 \le \ul{l}' \le \ul{l}} \sum_{0 \le \ul{l}'' \le \ul{l}'} \begin{pmatrix} \ul{l} \\ \ul{l}' \end{pmatrix}_{ q } \begin{pmatrix} \ul{l}' \\ \ul{l}'' \end{pmatrix}_{ q } (- \ul{x})^{{(\ul{l}-\ul{l}')}_q} \cdot \theta_{(m)}((- \ul{x})^{{(\ul{l}'')}_q})\ul{x}^{\ul{l}'-\ul{l}''} \notag \\
=&\sum_{0 \le \ul{l}'' \le \ul{l}} \theta_{(m)}((- \ul{x})^{{(\ul{l}'')}_q}) \left( \sum_{\ul{l}'' \le \ul{l}' \le \ul{l}} \begin{pmatrix} \ul{l} \\ \ul{l}' \end{pmatrix}_{ q } \begin{pmatrix} \ul{l}' \\ \ul{l}'' \end{pmatrix}_{ q } (- \ul{x})^{{(\ul{l}-\ul{l}')}_q} \cdot \ul{x}^{\ul{l}'-\ul{l}''} \right).
\end{align}
So, it remains to show that, for $\ul{l}'' < \ul{l}$ (that is, if there exists some $i$ such that $l''_i < l_i$),
\begin{equation}
\label{istra}
\sum_{\ul{l}'' \le \ul{l}' \le \ul{l}} \begin{pmatrix} \ul{l} \\ \ul{l}' \end{pmatrix}_{ q } \begin{pmatrix} \ul{l}' \\ \ul{l}'' \end{pmatrix}_{ q } (- \ul{x})^{{(\ul{l}-\ul{l}')}_q} \cdot \ul{x}^{\ul{l}'-\ul{l}''} = 0.
\end{equation}
As usual, we may assume that $A = \Z [ q ]_{(p,q-1)}[\ul{x}]$.
We can check that
\[
\begin{pmatrix} \ul{l} \\ \ul{l}' \end{pmatrix}_{ q } \begin{pmatrix} \ul{l}' \\ \ul{l}'' \end{pmatrix}_{ q } = \begin{pmatrix} \ul{l}-\ul{l}'' \\ \ul{l}- \ul{l}' \end{pmatrix}_{ q } \begin{pmatrix} \ul{l} \\ \ul{l}'' \end{pmatrix}_{ q },
\]
so it remains to show that, for $\ul{l}'' < \ul{l}$,
\begin{equation}
\label{istrab}
\sum_{\ul{l}'' \le \ul{l}' \le \ul{l}} \begin{pmatrix} \ul{l}-\ul{l}'' \\ \ul{l}- \ul{l}' \end{pmatrix}_{ q } (- \ul{x})^{{(\ul{l}-\ul{l}')}_q} \cdot \ul{x}^{\ul{l}'-\ul{l}''} = 0.
\end{equation}
If we set $a= -x_i \otimes 1, \ b= x_i \otimes 1$ in Proposition \ref{qbfm} below, we can check that
\begin{equation}
\label{kstr}
\prod_{i=1}^d \prod_{j=0}^{l_i-l''_i-1}(-q^j x_i \otimes 1 + x_i \otimes 1 ) = \sum_{\ul{l}'' \le \ul{l}' \le \ul{l}} \begin{pmatrix} \ul{l}-\ul{l}'' \\ \ul{l}- \ul{l}' \end{pmatrix}_{ q } (- \ul{x})^{{(\ul{l}-\ul{l}')}_q} \cdot \ul{x}^{\ul{l}'-\ul{l}''}.
\end{equation}
Notice that, for $i$ such that $l''_i < l_i$, there exists a term $- x_i \otimes 1 + x_i \otimes 1 $ on the left hand side, so (\ref{kstr}) is equal to 0 and the result follows.
\end{proof}

The following quantum binomial formula was used above.

\begin{prop}
\label{qbfm}
{\rm  (\cite[Proposition 2.14]{LQ15})}
Assume that $C$ is a commutative ring with a fixed $q \in C$. Then, we have for all $k \in \N$ and $a, b \in C$,
\[
\prod_{j=0}^{k-1}(q^j a + b) = \sum_{k'=0}^k q^{\frac{k'(k'-1)}{2}} \begin{pmatrix} k \\ k' \end{pmatrix}_{ q } a^{k'} b^{k-k'}.
\]
\end{prop}

We are now ready to prove the $q$-analog of the higher Poincar\'{e} lemma in \cite{Miy15}.

\begin{prop}
\label{mqpl}
{\rm (cf. \cite[Theorem 3.7]{Miy15})}
If $E \in {\bf CR}((\ol{A}/R)_{m\text{-}q^{p^m}\text{-crys}},\cO^{(m)}_{\ol{A}/R})$, then 
\[
E \wh{\otimes}_{\cO^{(m)}_{\ol{A}/R}} L^{(m)}(\breve{\Omega}^{(m)}_{A/R,q}  (\bullet))
\]
 forms a complex of $\cO^{(m)}_{\ol{A}/R}$-modules that resolves the direct sum of $p^{md}$ copies of $E$.
\end{prop}
\begin{proof}
We follow the proof of \cite{Miy15}. 
By Lemma \ref{lpstr}, each differential of  $L \breve{\Omega}^{(m)}_{A/R,q}  (\bullet)$ is compatible with the hyper $m$-$q^{p^m}$-stratification on each term. So, by Proposition \ref{crstr}, we have a complex $L^{(m)}(\breve{\Omega}^{(m)}_{A/R,q}  (\bullet))$ of $\cO^{(m)}_{\ol{A}/R}$-modules that gives $L \breve{\Omega}^{(m)}_{A/R,q}  (\bullet)$ when we evaluate at $A$. Also, by Lemma \ref{plst}, we have an $\cO^{(m)}_{\ol{A}/R}$-linear map \[ F \to L^{(m)}(\breve{\Omega}^{(m)}_{A/R,q}  (0)). \] By taking the complete tensor product with $E$ over $\cO^{(m)}_{\ol{A}/R}$, we get an $\cO^{(m)}_{\ol{A}/R}$-linear map
\[
E \wh{\otimes}_{\cO^{(m)}_{\ol{A}/R}} F \to E \wh{\otimes}_{\cO^{(m)}_{\ol{A}/R}} L^{(m)}(\breve{\Omega}^{(m)}_{A/R,q}  (\bullet)).
\]
We need to show that, for any object $(B,I_B)$  of $(\ol{A}/R)_{m\text{-}q^{p^m}\text{-crys}}$, the map
\[
\bigoplus_{\ul{k} \in \cB^{(m)}_d} E_B e_{\ul{k}}  \to E_B \wh{\otimes}_{B} L^{(m)}(\breve{\Omega}^{(m)}_{A/R,q}  (\bullet))_B
\]
is a quasi-isomorphism. By Proposition \ref{mqcov}, we can find a map \[ i_1: (B,I_B) \to (B', I_{B'}) \]  and a map \[ i_2 :(A,I_A) \to (B', I_{B'}) \] in $(\ol{A}/R)_{m\text{-}q^{p^m}\text{-crys}}$, where the underlying rings map of $i_1$ is $(p, (p)_{q^{p^m}})$-completely faithfully flat. Therefore, we are reduced to proving that the map
\[
\bigoplus_{\ul{k} \in \cB^{(m)}_d} E_A e_{\ul{k}}  \to E_A \wh{\otimes}_{A} L \breve{\Omega}^{(m)}_{A/R,q}  (\bullet)
\]
is a quasi-isomorphism, which follows because it is the complete tensor product of the quasi-isomorphism $\iota$ in Proposition \ref{tfpl} with $E_A \in {\bf FMod}^{\wedge}_{(p,(p)_{q^{p^m}})}(A)$.
\end{proof}

\begin{cor}
\label{gspl}
If $E \in {\bf CR}((\ol{A}/R)_{m\text{-}q^{p^m}\text{-crys}},\cO^{(m)}_{\ol{A}/R})$, then we have a quasi-isomorphism
\[
\left( R \Gamma ((\ol{A}/R)_{m\text{-}q^{p^m}\text{-crys}}, E) \right)^{\oplus p^{md}} \xrightarrow{\simeq} E_A \wh{\otimes}_A \breve{\Omega}^{(m)}_{A/R,q}  (\bullet).
\]
\end{cor}
\begin{proof}
We follow the notations in Section \ref{mqanc}. Then we have
\begin{align*}
\left( R \Gamma ((\ol{A}/R)_{m\text{-}q^{p^m}\text{-crys}}, E) \right)^{\oplus p^{md}} &\simeq Re_{\ol{A}/R*} (E^{\oplus p^{md}}) \\
&\simeq Re_{\ol{A}/R*} ( E \wh{\otimes}_{\cO^{(m)}_{\ol{A}/R}} L^{(m)} ( \breve{\Omega}^{(m)}_{A/R,q}  (\bullet)) ) .
\end{align*}
By Lemma \ref{elle}, this is quasi-isomorphic to \[ Re_{\ol{A}/R*} (L^{(m)}(E_A \wh{\otimes}_A \breve{\Omega}^{(m)}_{A/R,q}  (\bullet))). \] Then,  by  the  calculations as in the proof of Lemma \ref{rglm}, this is quasi-isomorphic to \[ E_A \wh{\otimes}_A \breve{\Omega}^{(m)}_{A/R,q}  (\bullet). \qedhere \] 
\end{proof}

\section{$q$-analog of the jet complex}
\label{qanj}

In this section, we construct the $q$-analog of the  ``jet complex of order $p^m$'' and the corresponding Poincar\'{e} lemma, following \cite{LQ01}. 

As in Section \ref{hple}, let $(R, I_R) \to (A, I_A)$ be a morphism of $q^{p^m}$-PD pairs with fixed rank one  \'{e}tale coordinates $\ul{x} = (x_1, \dots, x_d)$ in $A$. We also assume that $
I_{A} = \ol{I_R A}^{\rm cl}$. We consider the DGA $P^{(m)}_{A/R,q} (\bullet)$ and its sub-DGA $NP^{(m)}_{A/R,q} (\bullet)$. Let $K^{(m)}_{A/R,q} (\bullet)$ be the DG-ideal of $NP^{(m)}_{A/R,q} (\bullet)$ generated by the $(p,(p)_{q^{p^m}})$-completion of the free $A$-module with basis $\ul{\xi}^{ \{ \ul{k} \}_{(m), q}}$, where $| \ul{k} | \ge p^m + 1$.

\begin{defi}
\label{mqjc}
We define the \textit{$q$-jet complex} (of order $p^m$) $\Omega^{(m)}_{A/R,q}  (\bullet)$ as the quotient of $NP^{(m)}_{A/R,q} (\bullet)$ by the DG-ideal $K^{(m)}_{A/R,q} (\bullet)$.
\end{defi}

In this section, we sometimes omit to write the index $r$ of the differential $d^r$ of $P^{(m)}_{A/R,q} (\bullet), NP^{(m)}_{A/R,q} (\bullet), \Omega^{(m)}_{A/R,q}  (\bullet)$ and simply denote the differential by $d^{\bullet}$. For elements $y,z$ of these DGA, possibly of different degrees, Leibniz rule holds:
\begin{align}
\label{leiba}
d^{\bullet}(y \otimes' z) = d^{\bullet}y \otimes' z + (-1)^{\deg(y)} y \otimes' d^{\bullet}z.
\end{align}
Also note that if $y,z \in P^{(m)}_{A/R,q} (1)$ and $f \in A$, then we have the formula for moving coefficients inside products:
\begin{align}
\label{movc}
y \otimes' fz = \theta_{(m)}(f) y \otimes' z = fy \otimes' z + d^{0}(f)  y \otimes' z.
\end{align}
As before, if $0 \le k_i < 2p^m$ for $i = 1, \dots, d$, then we have $\ul{\xi}^{ \{ \ul{k} \}_{(m),q} } = \ul{\xi}^{ (\ul{k})_{q} }$ via the map (\ref{filma}). When $0 < | \ul{k} | \le p^m$, we shall often denote the image of $\ul{\xi}^{ \{ \ul{k} \}_{(m),q} }$ via the natural surjection $NP^{(m)}_{A/R,q} (1) \to \Omega^{(m)}_{A/R,q}  (1)$ by $(\ul{dx})^{ (\ul{k})_{q} }$, and for convenience, when $| \ul{k} | > p^m$, we set $(\ul{dx})^{ (\ul{k})_{q} }= 0$.
Then  $\Omega^{(m)}_{A/R,q}  (1)$ is a free $A$-module with basis $(\ul{dx})^{ (\ul{k})_{q} }$, where $0 < | \ul{k} | \le p^m$. 

The description of $\Omega^{(m)}_{A/R,q}  (r)$ for $r > 1$ is more complicated. 
By Proposition \ref{mqbc} and Definition \ref{mqjc}, we can check that $\Omega^{(m)}_{A/R,q} (\bullet)$ is the quotient of the tensor algebra of $\Omega^{(m)}_{A/R,q} (1)$ by the ideal generated by the images of the $d^{1}(\ul{\xi}^{ \{ \ul{k} \}_{(m),q} })$ with $p^m < | \ul{k} | \le 2p^m$. So for the degree 2 part, $\Omega^{(m)}_{A/R,q} (2)$ is generated by $(\ul{dx})^{ (\ul{k}_{(1)})_{q} } \otimes' (\ul{dx})^{ (\ul{k}_{(2)})_{q} }$ with $\ul{k}_{(1)},\ul{k}_{(2)} \in \N^d$ and $0 < | \ul{k}_{(1)} |, | \ul{k}_{(2)} | \le p^m$ subject to the relations
\begin{equation}
\label{dxdx}
\sum_{0 < \ul{k}' < \ul{k}} \left\langle
\begin{array}{@{}c@{}}
\ul{k} \\ \ul{k}'
\end{array}
\right\rangle_{(m), q }
(\ul{dx})^{ ( \ul{k} - \ul{k}' )_{q}} \otimes' (\ul{dx})^{ ( \ul{k}' )_{q}} = 0 \ \ \ \ \text{with} \ p^m < | \ul{k} | \le 2p^m.
\end{equation}
For the degree $r > 2$ parts, $\Omega^{(m)}_{A/R,q} (r)$ is generated by the $(\ul{dx})^{ (\ul{k}_{(1)})_{q} } \otimes' (\ul{dx})^{ (\ul{k}_{(2)})_{q} } \otimes' \cdots \otimes' (\ul{dx})^{ (\ul{k}_{(r)})_{q} } $ with $0 < | \ul{k}_{(1)} |, \dots, | \ul{k}_{(r)} | \le p^m$. The relations are given by tensoring on the right and on the left of (\ref{dxdx}) by the elements of the form $(\ul{dx})^{ (\ul{k}_{(1)})_{q} } \otimes' (\ul{dx})^{ (\ul{k}_{(2)})_{q} } \otimes' \cdots \otimes' (\ul{dx})^{ (\ul{k}_{(j)})_{q} } $.
Note that in general, the $q$-jet complex may not be bounded above.

As in Section \ref{hple}, we can also consider the DGA $LP^{(m)}_{A/R,q} (\bullet)$ and its sub-DGA $LNP^{(m)}_{A/R,q} (\bullet)$. We can check that the ideal \[ \wh{A \langle \ul{\xi} \rangle}_{(m), q, \ul{x}} \wh{\otimes}'_A K^{(m)}_{A/R,q} (\bullet) := \bigoplus_{r=0}^{\infty} \wh{A \langle \ul{\xi} \rangle}_{(m), q, \ul{x}} \wh{\otimes}'_A K^{(m)}_{A/R,q} (r) \] of $LNP^{(m)}_{A/R,q} (\bullet)$ is a DG-ideal. Then we can get the linearization of the $q$-jet complex:

\begin{defi}
\label{lmqjc}
We define the \textit{linearized $q$-jet complex} (of order $p^m$) $L\Omega^{(m)}_{A/R,q} (\bullet)$ as the quotient of $LNP^{(m)}_{A/R,q} (\bullet)$ by the DG-ideal $\wh{A \langle \ul{\xi} \rangle}_{(m), q, \ul{x}} \wh{\otimes}'_A K^{(m)}_{A/R,q} (\bullet)$.
\end{defi}

As before, we sometimes simply write $d^{\bullet}$ for the differential of $LP^{(m)}_{A/R,q} (\bullet),$ \linebreak $LNP^{(m)}_{A/R,q} (\bullet),$ $L\Omega^{(m)}_{A/R,q} (\bullet)$. Let $a \in \wh{A \langle \ul{\xi} \rangle}_{(m), q, \ul{x}}$ and  $z \in \Omega^{(m)}_{A/R,q} (r)$. For the differential $d^{\bullet}$ of $L\Omega^{(m)}_{A/R,q} (\bullet)$, Leibniz rule holds:
\begin{align}
\label{leibb}
d^{\bullet}(a \otimes' z) = d^{0}a \otimes' z + a \otimes' d^{r}z,
\end{align}
where
\begin{align}
\label{lindif}
d^0 : \wh{A \langle \ul{\xi} \rangle}_{(m), q, \ul{x}} \to \wh{A \langle \ul{\xi} \rangle}_{(m), q, \ul{x}} \wh{\otimes}'_A \Omega^{(m)}_{A/R,q} (1)
\end{align}
is the differential map of the linearized complex as in (\ref{lpdr}), and
\begin{align}
\label{usdif}
d^r : \Omega^{(m)}_{A/R,q} (r) \to \Omega^{(m)}_{A/R,q} (r+1)
\end{align}
is the differential map of the usual complex as in (\ref{dmqd}).

There is an identification 
\begin{align*}
L\Omega^{(m)}_{A/R,q} (1) \wh{\otimes}'_{\wh{A \langle \ul{\xi} \rangle}_{(m), q, \ul{x}}} L\Omega^{(m)}_{A/R,q} (1)  &\simeq L\Omega^{(m)}_{A/R,q} (2) \\
(a \otimes' y) \otimes' (b \otimes' z) &\mapsto (\delta_1^1(b)(a \otimes' y)) \otimes' z,
\end{align*}
where this time the notation $\wh{\otimes}'$ indicates that, for the $\wh{A \langle \ul{\xi} \rangle}_{(m), q, \ul{x}}$-module structure on the left hand side, we use the map $\delta_1^1 : P^{(m)}_{A/R,q} (1) \to P^{(m)}_{A/R,q} (2)$ in Theorem \ref{mqbt} and the multiplication rule induced by the ring structure on $P^{(m)}_{A/R,q} (2)$.
Then we have the following formula for moving coefficients inside products: if $\omega,\eta \in L\Omega^{(m)}_{A/R,q} (1)$ and $y \in \wh{A \langle \ul{\xi} \rangle}_{(m), q, \ul{x}}$, then
\begin{equation}
\label{mclo}
\omega \otimes' y \eta = y \omega \otimes' \eta + ((d^0 y) \omega) \otimes' \eta,
\end{equation}
where $d^0$ is the differential map in (\ref{lindif}), and the multiplication $(d^0 y) \omega$ is the one induced by the ring structure on $P^{(m)}_{A/R,q} (2)$.

We shall write $(\ul{d \xi})^{ ( \ul{k} )_{q}}:= 1 \otimes' (\ul{dx})^{ (\ul{k})_{q} }$ as an element of $L\Omega^{(m)}_{A/R,q} (1)$. Note that when $| \ul{k} | > p^m$, we have $(\ul{d \xi})^{ ( \ul{k} )_{q}}= 0$. For the differential $d^{\bullet}$ of $L\Omega^{(m)}_{A/R,q} (\bullet)$, we have the following formulas:
\begin{align}
d^{0}(\ul{\xi}^{ \{ \ul{k} \}_{(m),q} }) =
&\sum_{0 < \ul{k}' \le \ul{k}} \left\langle
\begin{array}{@{}c@{}}
\ul{k} \\ \ul{k}'
\end{array}
\right\rangle_{(m), q }
\ul{\xi}^{ \{ \ul{k} - \ul{k}' \}_{(m),q}} (\ul{d \xi})^{ ( \ul{k}' )_{q}}, \label{dxik} \\
d^{1}((\ul{d \xi})^{ ( \ul{k} )_{q}}) =
- &\sum_{0 < \ul{k}' < \ul{k}} \left\langle
\begin{array}{@{}c@{}}
\ul{k} \\ \ul{k}'
\end{array}
\right\rangle_{(m), q }
(\ul{d \xi})^{ ( \ul{k} - \ul{k}' )_{q}} \otimes' (\ul{d \xi})^{ ( \ul{k}' )_{q}}. \label{ddxi}
\end{align}
As before, we can consider \[ (\ul{d \xi})^{ (\ul{k}_{(1)})_{q} } \otimes' (\ul{d \xi})^{ (\ul{k}_{(2)})_{q} } \otimes' \cdots \otimes' (\ul{d \xi})^{ (\ul{k}_{(r)})_{q} } \] with $0 < | \ul{k}_{(1)} |, \dots, | \ul{k}_{(r)} | \le p^m$ as a set of generators of $L\Omega^{(m)}_{A/R,q} (r)$ over $\wh{A \langle \ul{\xi} \rangle}_{(m), q, \ul{x}}$. The relations are given by tensoring on the right and on the left of
\begin{equation}
\label{dxid}
\sum_{0 < \ul{k}' < \ul{k}} \left\langle
\begin{array}{@{}c@{}}
\ul{k} \\ \ul{k}'
\end{array}
\right\rangle_{(m), q }
(\ul{d \xi})^{ ( \ul{k} - \ul{k}' )_{q}} \otimes' (\ul{d \xi})^{ ( \ul{k}' )_{q}} = 0 \ \ \ \ \text{with} \ p^m < | \ul{k} | \le 2p^m
\end{equation}
by the elements of the form $(\ul{d \xi})^{ (\ul{k}_{(1)})_{q} } \otimes' (\ul{d \xi})^{ (\ul{k}_{(2)})_{q} } \otimes' \cdots \otimes' (\ul{d \xi})^{ (\ul{k}_{(j)})_{q} } $.

To prove the formal Poincar\'{e} lemma for the $q$-analog of the jet complex, we need to consider a technical homotopy map called the integration of differential forms. We write $\ul{x} = (x_1, \dots , x_d)$ as before and we set $\wh{\ul{x}} = (x_1, \dots , x_{d-1})$. We also write $\ul{\xi} = (\xi_1, \dots , \xi_d)$ and $\wh{\ul{\xi}} = (\xi_1, \dots , \xi_{d-1})$. When we consider the multi-index $\ul{k}$, it will belong to  $\N^d \ \text{or} \ \N^{d-1}$, and when it is a $d$-tuple $\ul{k} = (k_1, \dots, k_d)$, we will write $\wh{\ul{k}} = (k_1, \dots, k_{d-1})$. Moreover, given a $(d-1)$-tuple $\wh{\ul{k}}$ and $n \in \N$, we will denote by $(\wh{\ul{k}}, n)$ the $d$-tuple obtained by concatenating $\wh{\ul{k}}$ and $n$.

We define the following homotopy map $h$. We will show in Proposition \ref{wede} that $h$ is well-defined, and in Proposition \ref{hcond} that $h$ is unique.

\begin{prop}
\label{holo}
{\rm (cf. \cite[Proposition and Definition 2.1]{LQ01})}
There exists a unique $A$-linear map $h: L\Omega^{(m)}_{A/R,q} (r) \to L\Omega^{(m)}_{A/R,q} (r-1)$ such that
{\fontsize{8pt}{10pt}\selectfont
\begin{align*}
&h \left[ \wh{\ul{\xi}}^{ \{ \ul{k} \}_{(m),q}} \xi_d^{ \{ k \}_{(m),q}} (\ul{d \wh{\xi}})^{ (\ul{k}_{(1)})_{q} } \otimes' \cdots \otimes' (\ul{d \wh{\xi}})^{ (\ul{k}_{(s-1)})_{q} } \otimes' (\ul{d \wh{\xi}})^{ (\ul{k}_{(s)})_{q} }(d \xi_d)^{(l)_q} \otimes' (\ul{d \xi})^{ (\ul{k}_{(s+1)})_{q} } \otimes' \cdots \otimes' (\ul{d \xi})^{ (\ul{k}_{(r)})_{q} } \right] \\
=&\begin{cases}
0 \ \ \ \ \text{if} \ (1) \ s=r \ \text{and} \ l=0 \ \text{or} \ (2) \ l \ne 0 \ \text{and} \ \ul{k}_{(s)} \ne 0 \ \text{or} \ (3) \ p^m \nmid k \\
(-1)^{s-1} \left\langle
\begin{array}{@{}c@{}}
k+l \\ k
\end{array}
\right\rangle_{(m), q }^{-1} \wh{\ul{\xi}}^{ \{ \ul{k} \}_{(m),q}} (\ul{d \wh{\xi}})^{ (\ul{k}_{(1)})_{q} } \otimes' \cdots \otimes' (\ul{d \wh{\xi}})^{ (\ul{k}_{(s-1)})_{q} } \otimes' \xi_d^{ \{ k+l \}_{(m),q}} (\ul{d \xi})^{ (\ul{k}_{(s+1)})_{q} } \otimes' \cdots \otimes' (\ul{d \xi})^{ (\ul{k}_{(r)})_{q} } \\ \ \ \ \ \text{if} \ \ul{k}_{(s)} = 0, \ p^m \mid k \ \text{and either} \ (1) \ 0<l<p^m \ \text{or} \ (2) \ l = p^m \ \text{and} \ s = r.
\end{cases}
\end{align*}
}

\noindent
In the above equation and in the sequel of the article, if $\omega \in L\Omega^{(m)}_{A/R,q} (1)$ and $y \in \wh{A \langle \ul{\xi} \rangle}_{(m), q, \ul{x}}$, then, following {\rm (\ref{mclo})}, we set 
\begin{equation}
\label{omy}
\omega \otimes' y := y \omega  + (d^0 y) \omega.
\end{equation}
So by {\rm (\ref{dxik})}, in the case $l = p^m \ \text{and} \ s = r$, the above equation means the following:
\begin{align*}
&h \left[ \wh{\ul{\xi}}^{ \{ \ul{k} \}_{(m),q}} \xi_d^{ \{ k \}_{(m),q}} (\ul{d \wh{\xi}})^{ (\ul{k}_{(1)})_{q} } \otimes' \cdots \otimes' (\ul{d \wh{\xi}})^{ (\ul{k}_{(r-1)})_{q} } \otimes' (d \xi_d)^{(p^m)_q} \right] \\
=&
(-1)^{r-1}
\sum_{0 \le k' < p^m}
 \left\langle
\begin{array}{@{}c@{}}
k+p^m \\ k
\end{array}
\right\rangle_{(m), q }^{-1}
 \left\langle
\begin{array}{@{}c@{}}
k+p^m \\ k'
\end{array}
\right\rangle_{(m), q } \\
&\times
 \wh{\ul{\xi}}^{ \{ \ul{k} \}_{(m),q}} (\ul{d \wh{\xi}})^{ (\ul{k}_{(1)})_{q} } \otimes' \cdots \otimes' (\ul{d \wh{\xi}})^{ (\ul{k}_{(r-2)})_{q} } \otimes' \xi_d^{ \{ k+p^m - k' \}_{(m),q}}  (\ul{d \wh{\xi}})^{ (\ul{k}_{(r-1)})_{q} } (d \xi_d)^{ ( k' )_{q}}.
\end{align*}
\end{prop}

Note that when $p^m \mid k$ and $0 < l \le p^m$, we can show that
$\left\langle 
\begin{array}{@{}c@{}}
k+l \\ k
\end{array}
\right\rangle_{(m), q }$
in the formula above is invertible in $\Z [ q ]_{(p,q-1)}$ by the same argument as in the proof of Lemma \ref{ipmq}.


We first check that, if $h$ is well-defined, then we have the following formula for moving coefficients inside products:

\begin{lem}
\label{xiin}
{\rm (cf. \cite[Lemma 2.3]{LQ01})}
We always have
{\fontsize{8pt}{10pt}\selectfont
\begin{align*}
&h \left[ \wh{\ul{\xi}}^{ \{ \ul{k}  \}_{(m),q}} \xi_d^{ \{ k  \}_{(m),q}} (\ul{d \wh{\xi}})^{ (\ul{k}_{(1)})_{q} } \otimes' \cdots \otimes' (\ul{d \wh{\xi}})^{ (\ul{k}_{(s-1)})_{q} } \otimes' (\ul{d \wh{\xi}})^{ (\ul{k}_{(s)})_{q} }(d \xi_d)^{(l)_q} \otimes' (\ul{d \xi})^{ (\ul{k}_{(s+1)})_{q} } \otimes' \cdots \otimes' (\ul{d \xi})^{ (\ul{k}_{(r)})_{q} } \right] \\
= &h \left[ \wh{\ul{\xi}}^{ \{ \ul{k}  \}_{(m),q}}  (\ul{d \wh{\xi}})^{ (\ul{k}_{(1)})_{q} } \otimes' \cdots \otimes' (\ul{d \wh{\xi}})^{ (\ul{k}_{(s-1)})_{q} } \otimes' \xi_d^{ \{ k  \}_{(m),q}} (\ul{d \wh{\xi}})^{ (\ul{k}_{(s)})_{q} }(d \xi_d)^{(l)_q} \otimes' (\ul{d \xi})^{ (\ul{k}_{(s+1)})_{q} } \otimes' \cdots \otimes' (\ul{d \xi})^{ (\ul{k}_{(r)})_{q} } \right].
\end{align*}
}
\end{lem}
\begin{proof}
We follow the proof of \cite{LQ01}.
We proceed by induction on $s$. The case $s=1$ is trivial. Using the formula (\ref{mclo}) for moving coefficients across the tensor products, we have
{\fontsize{9pt}{10pt}\selectfont
\begin{align*}
&h \left[ \wh{\ul{\xi}}^{ \{ \ul{k}  \}_{(m),q}}  (\ul{d \wh{\xi}})^{ (\ul{k}_{(1)})_{q} } \otimes' \cdots \otimes' (\ul{d \wh{\xi}})^{ (\ul{k}_{(s-2)})_{q} } \otimes' (\ul{d \wh{\xi}})^{ (\ul{k}_{(s-1)})_{q} } \otimes' \xi_d^{ \{ k  \}_{(m),q}} (\ul{d \wh{\xi}})^{ (\ul{k}_{(s)})_{q} }(d \xi_d)^{(l)_q} \otimes' \cdots \right] \\
=&h \left[ \wh{\ul{\xi}}^{ \{ \ul{k}  \}_{(m),q}}  (\ul{d \wh{\xi}})^{ (\ul{k}_{(1)})_{q} } \otimes' \cdots \otimes' (\ul{d \wh{\xi}})^{ (\ul{k}_{(s-2)})_{q} } \otimes' \xi_d^{ \{ k  \}_{(m),q}} (\ul{d \wh{\xi}})^{ (\ul{k}_{(s-1)})_{q} }\otimes' (\ul{d \wh{\xi}})^{ (\ul{k}_{(s)})_{q} }(d \xi_d)^{(l)_q} \otimes' \cdots \right] \\
+ &h \left[ \wh{\ul{\xi}}^{ \{ \ul{k}  \}_{(m),q}}  (\ul{d \wh{\xi}})^{ (\ul{k}_{(1)})_{q} } \otimes' \cdots \otimes' (\ul{d \wh{\xi}})^{ (\ul{k}_{(s-2)})_{q} } \otimes' d^{0}(\xi_d^{ \{ k  \}_{(m),q}}) (\ul{d \wh{\xi}})^{ (\ul{k}_{(s-1)})_{q} }\otimes' (\ul{d \wh{\xi}})^{ (\ul{k}_{(s)})_{q} }(d \xi_d)^{(l)_q} \otimes' \cdots \right] .
\end{align*}
}

\noindent
For the first term, we can apply the induction hypothesis with the case $l=0$. So it suffices to show for the second term that, if $\ul{k}_{(s-1)} \ne 0$, then
\[
h \left[ \wh{\ul{\xi}}^{ \{ \ul{k}  \}_{(m),q}}  (\ul{d \wh{\xi}})^{ (\ul{k}_{(1)})_{q} } \otimes' \cdots \otimes' (\ul{d \wh{\xi}})^{ (\ul{k}_{(s-2)})_{q} } \otimes' d^{0}(\xi_d^{ \{ k  \}_{(m),q}}) (\ul{d \wh{\xi}})^{ (\ul{k}_{(s-1)})_{q} }\otimes'  \cdots \right] = 0.
\]
By (\ref{dxik}), we have
\[
d^{0}(\xi_d^{ \{ k  \}_{(m),q}}) =
\sum_{0 < k' \le k} \left\langle 
\begin{array}{@{}c@{}}
k \\ k'
\end{array}
\right\rangle_{(m), q }
\xi_d^{ \{ k - k'  \}_{(m),q}}  (d \xi_d)^{ ( k' )_{q}},
\]
so it is enough to show that for $k' \ne 0$, we have
{\fontsize{11pt}{10pt}\selectfont
\[
h \left[ \wh{\ul{\xi}}^{ \{ \ul{k}  \}_{(m),q}}  (\ul{d \wh{\xi}})^{ (\ul{k}_{(1)})_{q} } \otimes' \cdots \otimes' (\ul{d \wh{\xi}})^{ (\ul{k}_{(s-2)})_{q} } \otimes' \xi_d^{ \{ k - k'  \}_{(m),q}} (\ul{d \wh{\xi}})^{ (\ul{k}_{(s-1)})_{q} } (d \xi_d)^{ ( k' )_{q}} \otimes'  \cdots \right] = 0.
\]
}

\noindent
By induction hypothesis, this is
\[
h \left[ \wh{\ul{\xi}}^{ \{ \ul{k}  \}_{(m),q}} \xi_d^{ \{ k - k'  \}_{(m),q}}  (\ul{d \wh{\xi}})^{ (\ul{k}_{(1)})_{q} } \otimes' \cdots \otimes' (\ul{d \wh{\xi}})^{ (\ul{k}_{(s-2)})_{q} } \otimes'  (\ul{d \wh{\xi}})^{ (\ul{k}_{(s-1)})_{q} } (d \xi_d)^{ ( k' )_{q}} \otimes'  \cdots \right].
\]
Since $k' \ne 0$ and $\ul{k}_{(s-1)} \ne 0$, it is 0 by the definition of $h$.
\end{proof}

\begin{prop}
\label{hcond}
{\rm (cf. \cite[Proposition 2.4]{LQ01})}
If the map $h$ is well-defined, then $h$ is unique. More precisely, for $\ul{l} \in \N^d$ such that $p^m < | \ul{l} | \le 2p^m$ and $l_d \ge p^m$, we write $l = l_d, l'' = l_d - p^m$. Then, for the condition
{\fontsize{9pt}{10pt}\selectfont
\begin{equation}
\label{cda}
\sum_{0 < \ul{l}' < \ul{l}} \left\langle 
\begin{array}{@{}c@{}}
\ul{l} \\ \ul{l}'
\end{array}
\right\rangle_{(m), q }
h \left[ \wh{\ul{\xi}}^{ \{ \ul{k}  \}_{(m),q}} \xi_d^{ \{ k  \}_{(m),q}} (\ul{d \wh{\xi}})^{ (\ul{k}_{(1)})_{q} } \otimes' \cdots \otimes' (\ul{d \wh{\xi}})^{ (\ul{k}_{(s-1)})_{q} } \otimes' (\ul{d \xi})^{ ( \ul{l} - \ul{l}'  )_{q}} \otimes' (\ul{d \xi})^{ ( \ul{l}' )_{q}} \otimes' \cdots \right] = 0
\end{equation}}

\noindent
to be satisfied, it is necessary and sufficient that both the condition

{\fontsize{9pt}{10pt}\selectfont
\begin{align}
\label{cdb}
&h \left[ \wh{\ul{\xi}}^{ \{ \ul{k}  \}_{(m),q}} \xi_d^{ \{ k  \}_{(m),q}} (\ul{d \wh{\xi}})^{ (\ul{k}_{(1)})_{q} } \otimes' \cdots \otimes' (\ul{d \wh{\xi}})^{ (\ul{k}_{(s-1)})_{q} } \otimes' (d \xi_d)^{ (p^m  )_{q}} \otimes' (\ul{d \wh{\xi}})^{ ( \wh{\ul{l}} )_{q}} (d \xi_d)^{ (l''  )_{q}} \otimes' \cdots \right] \notag \\
= &- \sum_{l'' < l' \le p^m} \left\langle 
\begin{array}{@{}c@{}}
p^m+l'' \\ l''
\end{array}
\right\rangle_{(m), q }^{-1}
\left\langle 
\begin{array}{@{}c@{}}
p^m+l'' \\ l'
\end{array}
\right\rangle_{(m), q } \notag \\
&\times h \left[ \wh{\ul{\xi}}^{ \{ \ul{k}  \}_{(m),q}} \xi_d^{ \{ k  \}_{(m),q}} (\ul{d \wh{\xi}})^{ (\ul{k}_{(1)})_{q} } \otimes' \cdots \otimes' (\ul{d \wh{\xi}})^{ (\ul{k}_{(s-1)})_{q} } \otimes' (d \xi_d)^{ (p^m+l''-l'  )_{q}} \otimes' (\ul{d \wh{\xi}})^{ ( \wh{\ul{l}} )_{q}} (d \xi_d)^{ (l'  )_{q}} \otimes' \cdots \right],
\end{align}}

\noindent
when $l'' > 0$, and the condition

{\fontsize{9pt}{10pt}\selectfont
\begin{align}
\label{cdc}
&h \left[ \wh{\ul{\xi}}^{ \{ \ul{k}  \}_{(m),q}} \xi_d^{ \{ k  \}_{(m),q}} (\ul{d \wh{\xi}})^{ (\ul{k}_{(1)})_{q} } \otimes' \cdots \otimes' (\ul{d \wh{\xi}})^{ (\ul{k}_{(s-1)})_{q} } \otimes' (d \xi_d)^{ (p^m  )_{q}} \otimes' (\ul{d \wh{\xi}})^{ ( \wh{\ul{l}} )_{q}}  \otimes' \cdots \right] \notag \\
= &-h \left[ \wh{\ul{\xi}}^{ \{ \ul{k}  \}_{(m),q}} \xi_d^{ \{ k  \}_{(m),q}} (\ul{d \wh{\xi}})^{ (\ul{k}_{(1)})_{q} } \otimes' \cdots \otimes' (\ul{d \wh{\xi}})^{ (\ul{k}_{(s-1)})_{q} } \otimes' (\ul{d \wh{\xi}})^{ ( \wh{\ul{l}} )_{q}} \otimes' (d \xi_d)^{ (p^m  )_{q}}   \otimes' \cdots \right] \notag \\
&- \sum_{0 < l' < p^m} \left\langle 
\begin{array}{@{}c@{}}
p^m \\ l'
\end{array}
\right\rangle_{(m), q } \notag \\
&\times h \left[ \wh{\ul{\xi}}^{ \{ \ul{k}  \}_{(m),q}} \xi_d^{ \{ k  \}_{(m),q}} (\ul{d \wh{\xi}})^{ (\ul{k}_{(1)})_{q} } \otimes' \cdots \otimes' (\ul{d \wh{\xi}})^{ (\ul{k}_{(s-1)})_{q} } \otimes' (d \xi_d)^{ (p^m-l'  )_{q}} \otimes' (\ul{d \wh{\xi}})^{ ( \wh{\ul{l}} )_{q}} (d \xi_d)^{ (l'  )_{q}} \otimes' \cdots \right],
\end{align}}

\noindent
when $l'' = 0$, are satisfied.
\end{prop}

\begin{proof}
We follow the proof of \cite{LQ01}. First note that the statement gives a definition  of $h$ for all generators: for $l'' > 0$, we can get the definition for the generators of the type
\[
\wh{\ul{\xi}}^{ \{ \ul{k}  \}_{(m),q}} \xi_d^{ \{ k  \}_{(m),q}} (\ul{d \wh{\xi}})^{ (\ul{k}_{(1)})_{q} } \otimes' \cdots \otimes' (\ul{d \wh{\xi}})^{ (\ul{k}_{(s-1)})_{q} } \otimes' (d \xi_d)^{ (p^m  )_{q}} \otimes' (\ul{d \wh{\xi}})^{ ( \wh{\ul{l}} )_{q}} (d \xi_d)^{ (l''  )_{q}} \otimes' \cdots
\]
by using (\ref{cdb}), and for $l'' = 0$, we can get the definition for the generators of the type
\[
\wh{\ul{\xi}}^{ \{ \ul{k}  \}_{(m),q}} \xi_d^{ \{ k  \}_{(m),q}} (\ul{d \wh{\xi}})^{ (\ul{k}_{(1)})_{q} } \otimes' \cdots \otimes' (\ul{d \wh{\xi}})^{ (\ul{k}_{(s-1)})_{q} } \otimes' (d \xi_d)^{ (p^m  )_{q}} \otimes' (\ul{d \wh{\xi}})^{ ( \wh{\ul{l}} )_{q}}  \otimes' \cdots
\]
by using (\ref{cdc}) and descending induction on $s$.

Now, for $p^m < | \ul{l} | \le 2p^m$ and $l \ge p^m$, we consider the condition (\ref{cda}). We write $l' = l'_d$. By checking the term $(\ul{d \xi})^{ ( \ul{l} - \ul{l}'  )_{q}}$, we see that all the terms in the sum are 0 except in two mutually exclusive cases, $l' = l$ or $\wh{\ul{l'}}=\wh{\ul{l}}$. Thus, we are only left with these two sums. When $l > p^m$, there is only one sum related to the case $\wh{\ul{l'}}=\wh{\ul{l}}$:
{\fontsize{9pt}{10pt}\selectfont
\begin{align*}
&\sum_{0 \le l' < l} \left\langle 
\begin{array}{@{}c@{}}
l \\ l'
\end{array}
\right\rangle_{(m), q } \\
&\times h \left[ \wh{\ul{\xi}}^{ \{ \ul{k}  \}_{(m),q}} \xi_d^{ \{ k  \}_{(m),q}} (\ul{d \wh{\xi}})^{ (\ul{k}_{(1)})_{q} } \otimes' \cdots \otimes' (\ul{d \wh{\xi}})^{ (\ul{k}_{(s-1)})_{q} } \otimes' (d \xi_d)^{ ( l - l'  )_{q}} \otimes' (\ul{d \wh{\xi}})^{ ( \wh{\ul{l}} )_{q}} (d \xi_d)^{ (l'  )_{q}} \otimes' \cdots \right]
 = 0.
\end{align*} }

\noindent
Since all terms are $0$ unless $l'' \le l' \le p^m$, we get the formula (\ref{cdb}).

When $l = p^m$, note that since $| \ul{l} | > p^m$, we have $\wh{\ul{l}} \ne 0$. By the sum
{\fontsize{9pt}{10pt}\selectfont
\begin{align*}
&h \left[ \wh{\ul{\xi}}^{ \{ \ul{k}  \}_{(m),q}} \xi_d^{ \{ k  \}_{(m),q}} (\ul{d \wh{\xi}})^{ (\ul{k}_{(1)})_{q} } \otimes' \cdots \otimes' (\ul{d \wh{\xi}})^{ (\ul{k}_{(s-1)})_{q} } \otimes' (\ul{d \wh{\xi}})^{ ( \wh{\ul{l}} )_{q}} \otimes' (d \xi_d)^{ (p^m  )_{q}}   \otimes' \cdots \right] \\
+ &\sum_{0 \le l' < p^m} \left\langle 
\begin{array}{@{}c@{}}
p^m \\ l'
\end{array}
\right\rangle_{(m), q } \\ 
&\times h \left[ \wh{\ul{\xi}}^{ \{ \ul{k}  \}_{(m),q}} \xi_d^{ \{ k  \}_{(m),q}} (\ul{d \wh{\xi}})^{ (\ul{k}_{(1)})_{q} } \otimes' \cdots \otimes' (\ul{d \wh{\xi}})^{ (\ul{k}_{(s-1)})_{q} } \otimes' (d \xi_d)^{ (p^m-l'  )_{q}} \otimes' (\ul{d \wh{\xi}})^{ ( \wh{\ul{l}} )_{q}} (d \xi_d)^{ (l'  )_{q}} \otimes' \cdots \right] = 0,
\end{align*}}

\noindent
we get the formula (\ref{cdc}).
\end{proof}

\begin{rem}
\label{kspm}
{\rm (cf. \cite[Remark 2.5]{LQ01})}
We follow the remark and make a small correction of \cite{LQ01}.
We can check by the definition of $h$ in Proposition \ref{holo} that, unless $l=0$ or
\[
\ul{k}_{(s)} = 0, p^m \mid k, l = p^m \ \text{and} \ s<r,
\]
we have
{\fontsize{8pt}{10pt}\selectfont
\begin{align}
\label{rema}
&h \left[ \wh{\ul{\xi}}^{ \{ \ul{k}  \}_{(m),q}} \xi_d^{ \{ k  \}_{(m),q}} (\ul{d \wh{\xi}})^{ (\ul{k}_{(1)})_{q} } \otimes' \cdots \otimes' (\ul{d \wh{\xi}})^{ (\ul{k}_{(s-1)})_{q} } \otimes' (\ul{d \wh{\xi}})^{ (\ul{k}_{(s)})_{q} }(d \xi_d)^{(l)_q} \otimes' (\ul{d \xi})^{ (\ul{k}_{(s+1)})_{q} } \otimes' \cdots \otimes' (\ul{d \xi})^{ (\ul{k}_{(r)})_{q} } \right] \notag \\
=&
(-1)^{s-1} \notag \\
\times  &\wh{\ul{\xi}}^{ \{ \ul{k}  \}_{(m),q}}  (\ul{d \wh{\xi}})^{ (\ul{k}_{(1)})_{q} } \otimes' \cdots \otimes' (\ul{d \wh{\xi}})^{ (\ul{k}_{(s-1)})_{q} } \otimes' h \left[ \xi_d^{ \{ k  \}_{(m),q}}(\ul{d \wh{\xi}})^{ (\ul{k}_{(s)})_{q} }(d \xi_d)^{(l)_q} \right] (\ul{d \xi})^{ (\ul{k}_{(s+1)})_{q} } \otimes' \cdots \otimes' (\ul{d \xi})^{ (\ul{k}_{(r)})_{q} }.
\end{align}
}

\noindent
(The condition ``unless $l=0$'' was missing in \cite{LQ01}.)

We can also check that we always have
{\fontsize{8pt}{10pt}\selectfont
\begin{align}
\label{remb}
&h \left[ \wh{\ul{\xi}}^{ \{ \ul{k}  \}_{(m),q}} \xi_d^{ \{ k  \}_{(m),q}} (\ul{d \wh{\xi}})^{ (\ul{k}_{(1)})_{q} } \otimes' \cdots \otimes' (\ul{d \wh{\xi}})^{ (\ul{k}_{(s-1)})_{q} } \otimes' (\ul{d \wh{\xi}})^{ (\ul{k}_{(s)})_{q} }(d \xi_d)^{(l)_q} \otimes' (\ul{d \xi})^{ (\ul{k}_{(s+1)})_{q} } \otimes' \cdots \otimes' (\ul{d \xi})^{ (\ul{k}_{(r)})_{q} } \right] \notag \\
=&
(-1)^{s-1} \notag \\
\times  &\wh{\ul{\xi}}^{ \{ \ul{k}  \}_{(m),q}}  (\ul{d \wh{\xi}})^{ (\ul{k}_{(1)})_{q} } \otimes' \cdots \otimes' (\ul{d \wh{\xi}})^{ (\ul{k}_{(s-1)})_{q} } \otimes' h \left[ \xi_d^{ \{ k  \}_{(m),q}}(\ul{d \wh{\xi}})^{ (\ul{k}_{(s)})_{q} }(d \xi_d)^{(l)_q} \otimes' (\ul{d \xi})^{ (\ul{k}_{(s+1)})_{q} } \otimes' \cdots \otimes' (\ul{d \xi})^{ (\ul{k}_{(r)})_{q} } \right] .
\end{align}
}

\noindent
since when $l=p^m$ and $s<r$, we can apply the formula (\ref{cdb}), (\ref{cdc}).

Moreover, if $l \ne 0$, then by (\ref{cdb}), we have:
{\fontsize{7.8pt}{10pt}\selectfont
\begin{align}
\label{remc}
&h \left[ \wh{\ul{\xi}}^{ \{ \ul{k}  \}_{(m),q}} \xi_d^{ \{ k  \}_{(m),q}} (\ul{d \wh{\xi}})^{ (\ul{k}_{(1)})_{q} } \otimes' \cdots \otimes' (\ul{d \wh{\xi}})^{ (\ul{k}_{(s-1)})_{q} } \otimes' (d \xi_d)^{ (p^m  )_{q}} \otimes' (\ul{d \wh{\xi}})^{ ( \ul{k}_{(s+1)})_{q}} (d \xi_d)^{ (l  )_{q}} \otimes' \cdots \otimes' (\ul{d \xi})^{ (\ul{k}_{(r)})_{q} } \right] \notag \\
=&
(-1)^{s-1} \notag \\
\times &\wh{\ul{\xi}}^{ \{ \ul{k}  \}_{(m),q}}  (\ul{d \wh{\xi}})^{ (\ul{k}_{(1)})_{q} } \otimes' \cdots \otimes' (\ul{d \wh{\xi}})^{ (\ul{k}_{(s-1)})_{q} } \otimes' h \left[ \xi_d^{ \{ k  \}_{(m),q}} (d \xi_d)^{ (p^m  )_{q}} \otimes' (\ul{d \wh{\xi}})^{ ( \ul{k}_{(s+1)})_{q}} (d \xi_d)^{ (l  )_{q}} \right] \otimes' \cdots \otimes' (\ul{d \xi})^{ (\ul{k}_{(r)})_{q} }.
\end{align}}
\end{rem}

\begin{prop}
\label{wede}
{\rm (cf. \cite[Proposition 2.6]{LQ01})}
The map $h$ is well-defined.
\end{prop}
\begin{proof}
We follow the proof of \cite{LQ01}.
We need to show that the definition of $h$ is stable under the relations induced by
\begin{align}
\label{rela}
\sum_{0 < \ul{l}' < \ul{l}} \left\langle 
\begin{array}{@{}c@{}}
\ul{l} \\ \ul{l}'
\end{array}
\right\rangle_{(m), q }
(\ul{d \xi})^{ ( \ul{l} - \ul{l}'  )_{q}} \otimes' (\ul{d \xi})^{ ( \ul{l}' )_{q}} = 0
\end{align}
for $p^m < | \ul{l} | \le 2p^m$. More precisely, 
after excluding the trivial case, we need to consider the following two cases whenever $p^m \mid k$ and $p^m < | \ul{l} | \le 2p^m$: 
{\fontsize{9pt}{10pt}\selectfont
\begin{equation}
\label{cdf}
\sum_{0 < \ul{l}' < \ul{l}} \left\langle 
\begin{array}{@{}c@{}}
\ul{l} \\ \ul{l}'
\end{array}
\right\rangle_{(m), q }
h \left[ \wh{\ul{\xi}}^{ \{ \ul{k}  \}_{(m),q}} \xi_d^{ \{ k  \}_{(m),q}} (\ul{d \wh{\xi}})^{ (\ul{k}_{(1)})_{q} } \otimes' \cdots \otimes' (\ul{d \wh{\xi}})^{ (\ul{k}_{(s-1)})_{q} } \otimes' (\ul{d \xi})^{ ( \ul{l} - \ul{l}'  )_{q}} \otimes' (\ul{d \xi})^{ ( \ul{l}' )_{q}} \otimes' \cdots \right] = 0,
\end{equation}}

\noindent
and
{\fontsize{9pt}{10pt}\selectfont
\begin{align}
\label{cdg}
&\sum_{0 < \ul{l}' < \ul{l}} \left\langle 
\begin{array}{@{}c@{}}
\ul{l} \\ \ul{l}'
\end{array}
\right\rangle_{(m), q } \notag \\
&\times h \left[ \wh{\ul{\xi}}^{ \{ \ul{k}  \}_{(m),q}} \xi_d^{ \{ k  \}_{(m),q}} (\ul{d \wh{\xi}})^{ (\ul{k}_{(1)})_{q} } \otimes' \cdots \otimes' (\ul{d \wh{\xi}})^{ (\ul{k}_{(s-1)})_{q} } \otimes' (d \xi_d)^{ (l''  )_{q}} \otimes' \cdots \otimes' (\ul{d \xi})^{ ( \ul{l} - \ul{l}'  )_{q}} \otimes' (\ul{d \xi})^{ ( \ul{l}' )_{q}} \otimes' \cdots \right] = 0,
\end{align}}

\noindent
where $(d \xi_d)^{ (l''  )_{q}}$ is the $s$-th term and $(\ul{d \xi})^{ ( \ul{l} - \ul{l}'  )_{q}}$ is the $(s+r)$-th term for a positive integer $r$.

We first show the case (\ref{cdf}). As in the proof of Proposition \ref{hcond}, we write $l = l_d$ and $l' = l'_d$. The case $l = 0$ is trivial by using (\ref{remb}). The case $l \ge p^m$ was done in Proposition \ref{hcond}. So it suffices to consider the case $0 < l < p^m$. As in the proof of Proposition \ref{hcond}, we are left with two sums:
{\fontsize{8pt}{10pt}\selectfont
\begin{align*}
&\sum_{0 \le \wh{\ul{l'}} < \wh{\ul{l}}} \left\langle 
\begin{array}{@{}c@{}}
\wh{\ul{l}} \\ \wh{\ul{l'}}
\end{array}
\right\rangle_{(m), q }
h \left[ \wh{\ul{\xi}}^{ \{ \ul{k}  \}_{(m),q}} \xi_d^{ \{ k  \}_{(m),q}} (\ul{d \wh{\xi}})^{ (\ul{k}_{(1)})_{q} } \otimes' \cdots \otimes' (\ul{d \wh{\xi}})^{ (\ul{k}_{(s-1)})_{q} } \otimes' (\ul{d \wh{\xi}})^{ (\wh{\ul{l}} - \wh{\ul{l'}})_{q} } \otimes' (\ul{d \wh{\xi}})^{ ( \wh{\ul{l'}} )_{q}} (d \xi_d)^{ (l  )_{q}} \otimes' \cdots \right] \\
+ &\sum_{0 \le l' < l} \left\langle 
\begin{array}{@{}c@{}}
l \\ l'
\end{array}
\right\rangle_{(m), q }
h \left[ \wh{\ul{\xi}}^{ \{ \ul{k}  \}_{(m),q}} \xi_d^{ \{ k  \}_{(m),q}} (\ul{d \wh{\xi}})^{ (\ul{k}_{(1)})_{q} } \otimes' \cdots \otimes' (\ul{d \wh{\xi}})^{ (\ul{k}_{(s-1)})_{q} } \otimes' (d \xi_d)^{ ( l - l'  )_{q}} \otimes' (\ul{d \wh{\xi}})^{ ( \wh{\ul{l}} )_{q}} (d \xi_d)^{ (l'  )_{q}} \otimes' \cdots \right].
\end{align*} }

\noindent
For the first sum, it is non-trivial only when $\wh{\ul{l'}} = 0$, and we have the following:
{\fontsize{9pt}{10pt}\selectfont
\begin{align}
\label{cdi}
&h \left[ \wh{\ul{\xi}}^{ \{ \ul{k}  \}_{(m),q}} \xi_d^{ \{ k  \}_{(m),q}} (\ul{d \wh{\xi}})^{ (\ul{k}_{(1)})_{q} } \otimes' \cdots \otimes' (\ul{d \wh{\xi}})^{ (\ul{k}_{(s-1)})_{q} } \otimes' (\ul{d \wh{\xi}})^{ (\wh{\ul{l}})_{q} } \otimes' (d \xi_d)^{ (l  )_{q}} \otimes' \cdots \right] \notag \\
= &(-1)^s \left\langle 
\begin{array}{@{}c@{}}
k+l \\ k
\end{array}
\right\rangle_{(m), q }^{-1}
\wh{\ul{\xi}}^{ \{ \ul{k}  \}_{(m),q}} (\ul{d \wh{\xi}})^{ (\ul{k}_{(1)})_{q} } \otimes' \cdots \otimes' (\ul{d \wh{\xi}})^{ (\ul{k}_{(s-1)})_{q} } \otimes' (\ul{d \wh{\xi}})^{ (\wh{\ul{l}})_{q} } \otimes' \xi_d^{ \{ k+l  \}_{(m),q}} \cdots \notag \\
= &(-1)^s \left\langle 
\begin{array}{@{}c@{}}
k+l \\ k
\end{array}
\right\rangle_{(m), q }^{-1}
\wh{\ul{\xi}}^{ \{ \ul{k}  \}_{(m),q}} (\ul{d \wh{\xi}})^{ (\ul{k}_{(1)})_{q} } \otimes' \cdots \otimes' (\ul{d \wh{\xi}})^{ (\ul{k}_{(s-1)})_{q} } \otimes' \xi_d^{ \{ k+l  \}_{(m),q}} (\ul{d \wh{\xi}})^{ (\wh{\ul{l}})_{q} } \otimes'  \cdots  \notag \\
+ &(-1)^s \left\langle 
\begin{array}{@{}c@{}}
k+l \\ k
\end{array}
\right\rangle_{(m), q }^{-1}
\wh{\ul{\xi}}^{ \{ \ul{k}  \}_{(m),q}} (\ul{d \wh{\xi}})^{ (\ul{k}_{(1)})_{q} } \otimes' \cdots \otimes' (\ul{d \wh{\xi}})^{ (\ul{k}_{(s-1)})_{q} } \otimes' d^{\bullet}(\xi_d^{ \{ k+l  \}_{(m),q}}) (\ul{d \wh{\xi}})^{ (\wh{\ul{l}})_{q} } \otimes'  \cdots.
\end{align} }

\noindent
Now we consider the second sum. We get
{\fontsize{9pt}{10pt}\selectfont
\begin{align}
\label{cdh}
(-1)^{s-1} &\sum_{0 \le l' < l} \left\langle 
\begin{array}{@{}c@{}}
l \\ l'
\end{array}
\right\rangle_{(m), q }
\left\langle 
\begin{array}{@{}c@{}}
k+l-l' \\ k
\end{array}
\right\rangle_{(m), q }^{-1} \notag \\
&\times \wh{\ul{\xi}}^{ \{ \ul{k}  \}_{(m),q}} (\ul{d \wh{\xi}})^{ (\ul{k}_{(1)})_{q} } \otimes' \cdots \otimes' (\ul{d \wh{\xi}})^{ (\ul{k}_{(s-1)})_{q} } \otimes' \xi_d^{ \{ k+l-l' \}_{(m),q}}  (\ul{d \wh{\xi}})^{ (\wh{\ul{l}})_{q} } (d \xi_d)^{ (l'  )_{q}} \otimes'  \cdots.
\end{align} }

\noindent
We can replace the sum over $0 \le l' < l$ by that over $0 \le l' \le k+l$, because if $l \le l' \le k+l$, we have $| (\wh{\ul{l}}, l') | \ge | \ul{l} | > p^m$ and so
\[
(\ul{d \wh{\xi}})^{ (\wh{\ul{l}})_{q} } (d \xi_d)^{ (l'  )_{q}} = 0.
\]
We can check that
\[
\left\langle 
\begin{array}{@{}c@{}}
l \\ l'
\end{array}
\right\rangle_{(m), q }
\left\langle 
\begin{array}{@{}c@{}}
k+l-l' \\ k
\end{array}
\right\rangle_{(m), q }^{-1} = 
\left\langle 
\begin{array}{@{}c@{}}
k+l \\ k
\end{array}
\right\rangle_{(m), q }^{-1}
\left\langle 
\begin{array}{@{}c@{}}
k+l \\ l'
\end{array}
\right\rangle_{(m), q },
\]
and (\ref{cdh}) can be split into the $l'=0$ part and the sum for $0 < l' \le k+l$:
{\fontsize{9pt}{10pt}\selectfont
\begin{align*}
&(-1)^{s-1} \left\langle 
\begin{array}{@{}c@{}}
k+l \\ k
\end{array}
\right\rangle_{(m), q }^{-1} 
\wh{\ul{\xi}}^{ \{ \ul{k}  \}_{(m),q}} (\ul{d \wh{\xi}})^{ (\ul{k}_{(1)})_{q} } \otimes' \cdots \otimes' (\ul{d \wh{\xi}})^{ (\ul{k}_{(s-1)})_{q} } \otimes' \xi_d^{ \{ k+l  \}_{(m),q}} (\ul{d \wh{\xi}})^{ (\wh{\ul{l}})_{q} } \otimes'  \cdots \\
+ &(-1)^{s-1} \left\langle 
\begin{array}{@{}c@{}}
k+l \\ k
\end{array}
\right\rangle_{(m), q }^{-1} \cdot
\sum_{0 < l' \le k+l} \left\langle 
\begin{array}{@{}c@{}}
k+l \\ l'
\end{array}
\right\rangle_{(m), q } \cdots \otimes' \xi_d^{ \{ k+l-l' \}_{(m),q}}  (\ul{d \wh{\xi}})^{ (\wh{\ul{l}})_{q} } (d \xi_d)^{ (l'  )_{q}} \otimes'  \cdots.
\end{align*} }

\noindent
By (\ref{dxik}), this is equal to
{\fontsize{9pt}{10pt}\selectfont
\begin{align}
\label{cdj}
&(-1)^{s-1} \left\langle 
\begin{array}{@{}c@{}}
k+l \\ k
\end{array}
\right\rangle_{(m), q }^{-1} 
\wh{\ul{\xi}}^{ \{ \ul{k}  \}_{(m),q}} (\ul{d \wh{\xi}})^{ (\ul{k}_{(1)})_{q} } \otimes' \cdots \otimes' (\ul{d \wh{\xi}})^{ (\ul{k}_{(s-1)})_{q} } \otimes' \xi_d^{ \{ k+l  \}_{(m),q}} (\ul{d \wh{\xi}})^{ (\wh{\ul{l}})_{q} } \otimes'  \cdots  \notag \\
+ &(-1)^{s-1} \left\langle 
\begin{array}{@{}c@{}}
k+l \\ k
\end{array}
\right\rangle_{(m), q }^{-1}
\wh{\ul{\xi}}^{ \{ \ul{k}  \}_{(m),q}} (\ul{d \wh{\xi}})^{ (\ul{k}_{(1)})_{q} } \otimes' \cdots \otimes' (\ul{d \wh{\xi}})^{ (\ul{k}_{(s-1)})_{q} } \otimes' d^{\bullet}(\xi_d^{ \{ k+l  \}_{(m),q}}) (\ul{d \wh{\xi}})^{ (\wh{\ul{l}})_{q} } \otimes'  \cdots.
\end{align} }

\noindent
Adding (\ref{cdi}) and (\ref{cdj}), we get 0 as expected.

Now we consider the case (\ref{cdg}). If $0 < l'' < p^m$, then we can use $(\ref{rema})$ and get 0 by the relation $(\ref{rela})$. So we assume that $l'' = p^m$. For the case $r>1$, we focus on the $(s+1)$-st term $(\ul{d \wh{\xi}})^{ (\ul{k}_{(s+1)})_{q} }(d \xi_d)^{ (l'''  )_{q}}$. If $l''' > 0$, then we can use  $(\ref{remc})$ and get 0 by the relation $(\ref{rela})$. If $l''' = 0$, then we can apply the formula (\ref{cdc}) and, again by using $(\ref{rema})$, we are able to apply induction on $r$. So it suffices to consider the case $l'' = p^m, r=1$  whenever $p^m \mid k$ and $p^m < | \ul{l} | \le 2p^m$:
{\fontsize{9pt}{10pt}\selectfont
\begin{align}
\label{cdk}
&\sum_{0 < \ul{l}' < \ul{l}} \left\langle 
\begin{array}{@{}c@{}}
\ul{l} \\ \ul{l}'
\end{array}
\right\rangle_{(m), q } \notag \\
&\times h \left[ \wh{\ul{\xi}}^{ \{ \ul{k}  \}_{(m),q}} \xi_d^{ \{ k  \}_{(m),q}} (\ul{d \wh{\xi}})^{ (\ul{k}_{(1)})_{q} } \otimes' \cdots \otimes' (\ul{d \wh{\xi}})^{ (\ul{k}_{(s-1)})_{q} } \otimes' (d \xi_d)^{ (p^m)_{q}} \otimes' (\ul{d \xi})^{ ( \ul{l} - \ul{l}'  )_{q}} \otimes' (\ul{d \xi})^{ ( \ul{l}' )_{q}} \otimes' \cdots \right] = 0.
\end{align}}

\noindent
We consider the $d$-tuple $\ul{u} = (\wh{\ul{l}}, l_d + p^m)$ and the elements
\begin{align*}
1 \otimes' \ul{\xi}^{ \{ \ul{u}  \}_{(m),q}} &\in LP^{(m)}_{A/R,q} (1), \\ d^{\bullet}(d^{\bullet}(1 \otimes' \ul{\xi}^{ \{ \ul{u}  \}_{(m),q}})) &\in LP^{(m)}_{A/R,q} (3).
\end{align*}
Since the differential satisfies $d^{\bullet} \circ d^{\bullet}=0$, by the formulas similar to (\ref{ddxi}) and (\ref{leiba}), we  have
\begin{align*}
&\sum_{0 < \ul{w} < \ul{u}} 
\left\langle 
\begin{array}{@{}c@{}}
\ul{u} \\ \ul{w}
\end{array}
\right\rangle_{(m), q }
d^{\bullet}(1 \otimes' \ul{\xi}^{ \{ \ul{u} - \ul{w}  \}_{(m),q}}) \otimes' (1 \otimes' \ul{\xi}^{ \{ \ul{w}  \}_{(m),q}})  \\
=
&\sum_{0 < \ul{w} < \ul{u}} 
\left\langle 
\begin{array}{@{}c@{}}
\ul{u} \\ \ul{w}
\end{array}
\right\rangle_{(m), q }
(1 \otimes' \ul{\xi}^{ \{ \ul{u} - \ul{w}  \}_{(m),q}}) \otimes' d^{\bullet}(1 \otimes' \ul{\xi}^{ \{ \ul{w}  \}_{(m),q}}).
\end{align*}
By considering the images of these elements in $L\Omega^{(m)}_{A/R,q}  (\bullet)$, we see that
{\fontsize{9pt}{10pt}\selectfont
\begin{align}
\label{cdl}
&\sum_{0 < \ul{w} < \ul{u}} 
\left\langle 
\begin{array}{@{}c@{}}
\ul{u} \\ \ul{w}
\end{array}
\right\rangle_{(m), q }
h \left[ \wh{\ul{\xi}}^{ \{ \ul{k}  \}_{(m),q}} \xi_d^{ \{ k  \}_{(m),q}} (\ul{d \wh{\xi}})^{ (\ul{k}_{(1)})_{q} } \otimes' \cdots \otimes' (\ul{d \wh{\xi}})^{ (\ul{k}_{(s-1)})_{q} } \otimes' d^{\bullet}((\ul{d \xi})^{ ( \ul{u} - \ul{w}  )_{q}}) \otimes'  (\ul{d \xi})^{ ( \ul{w}  )_{q}} \otimes' \cdots \right] \notag \\
=&\sum_{0 < \ul{w} < \ul{u}} 
\left\langle 
\begin{array}{@{}c@{}}
\ul{u} \\ \ul{w}
\end{array}
\right\rangle_{(m), q }
h \left[ \wh{\ul{\xi}}^{ \{ \ul{k}  \}_{(m),q}} \xi_d^{ \{ k  \}_{(m),q}} (\ul{d \wh{\xi}})^{ (\ul{k}_{(1)})_{q} } \otimes' \cdots \otimes' (\ul{d \wh{\xi}})^{ (\ul{k}_{(s-1)})_{q} } \otimes' (\ul{d \xi})^{ ( \ul{u} - \ul{w}  )_{q}}  \otimes' d^{\bullet}((\ul{d \xi})^{ ( \ul{w}  )_{q}})  \otimes' \cdots \right].
\end{align}}

\noindent
We now consider the terms in the first sum. If $| \ul{w} | > p^m$, then by $(\ul{d \xi})^{ ( \ul{w}  )_{q}}$ we get 0. If not, then $| \ul{u} - \ul{w} | > p^m$, so by using (\ref{ddxi}) for $d^{\bullet}((\ul{d \xi})^{ ( \ul{u} - \ul{w}  )_{q}})$ and applying (\ref{cdf}), we see that the corresponding term is 0. Thus the first sum is equal to 0 and we have
{\fontsize{9pt}{10pt}\selectfont
\begin{align*}
\sum_{0 < \ul{w} < \ul{u}} 
\left\langle 
\begin{array}{@{}c@{}}
\ul{u} \\ \ul{w}
\end{array}
\right\rangle_{(m), q }
h \left[ \wh{\ul{\xi}}^{ \{ \ul{k}  \}_{(m),q}} \xi_d^{ \{ k  \}_{(m),q}} (\ul{d \wh{\xi}})^{ (\ul{k}_{(1)})_{q} } \otimes' \cdots \otimes' (\ul{d \wh{\xi}})^{ (\ul{k}_{(s-1)})_{q} } \otimes' (\ul{d \xi})^{ ( \ul{u} - \ul{w}  )_{q}}  \otimes' d^{\bullet}((\ul{d \xi})^{ ( \ul{w}  )_{q}})  \otimes' \cdots \right]=0.
\end{align*}}

\noindent
We write $u = u_d$ and $w = w_d$. As before, all the terms are 0 unless $w = u$ or $\wh{\ul{w}} = \wh{\ul{u}}$, so we have
{\fontsize{7pt}{10pt}\selectfont
\begin{align*}
&\sum_{0 \le \wh{\ul{w}} < \wh{\ul{u}}} \left\langle 
\begin{array}{@{}c@{}}
\wh{\ul{u}} \\ \wh{\ul{w}}
\end{array}
\right\rangle_{(m), q }
h \left[ \wh{\ul{\xi}}^{ \{ \ul{k}  \}_{(m),q}} \xi_d^{ \{ k  \}_{(m),q}} (\ul{d \wh{\xi}})^{ (\ul{k}_{(1)})_{q} } \otimes' \cdots \otimes' (\ul{d \wh{\xi}})^{ (\ul{k}_{(s-1)})_{q} } \otimes' (\ul{d \wh{\xi}})^{ (\wh{\ul{u}} - \wh{\ul{w}})_{q} } \otimes' d^{\bullet}((\ul{d \wh{\xi}})^{ ( \wh{\ul{w}} )_{q}} (d \xi_d)^{ (u  )_{q}}) \otimes' \cdots \right] \\
+ &\sum_{0 \le w < u} \left\langle 
\begin{array}{@{}c@{}}
u \\ w
\end{array}
\right\rangle_{(m), q } \\
&\times 
h \left[ \wh{\ul{\xi}}^{ \{ \ul{k}  \}_{(m),q}} \xi_d^{ \{ k  \}_{(m),q}} (\ul{d \wh{\xi}})^{ (\ul{k}_{(1)})_{q} } \otimes' \cdots \otimes' (\ul{d \wh{\xi}})^{ (\ul{k}_{(s-1)})_{q} } \otimes' (d \xi_d)^{ ( u - w  )_{q}} \otimes' d^{\bullet}((\ul{d \wh{\xi}})^{ ( \wh{\ul{u}} )_{q}} (d \xi_d)^{ (w  )_{q}}) \otimes' \cdots \right] = 0.
\end{align*} }

\noindent
For the first sum, if $| (\wh{\ul{w}},u) | > p^m$, then we can use (\ref{ddxi}) for $d^{\bullet}((\ul{d \wh{\xi}})^{ ( \wh{\ul{w}} )_{q}} (d \xi_d)^{ (u  )_{q}})$  and apply (\ref{cdf}). So all the terms in the first sum are 0 unless $u = p^m$  and $\wh{\ul{w}} = 0$. But in this case, we have $l_d = 0$ and so $| \wh{\ul{u}} - \wh{\ul{w}} | = |\wh{\ul{l}}| > p^m$, so by the factor $(\ul{d \wh{\xi}})^{ (\wh{\ul{u}} - \wh{\ul{w}})_{q} }$ we also get 0. For the second sum, if $ u - w  > p^m$, then by the factor $(d \xi_d)^{ ( u - w  )_{q}}$ we get 0. If $0 <  u - w  < p^m$, then we have $| (\wh{\ul{u}}, w) | > |\ul{l}| > p^m$. So after using $(\ref{rema})$, we can use (\ref{ddxi}) for $d^{\bullet}((\ul{d \wh{\xi}})^{ ( \wh{\ul{u}} )_{q}} (d \xi_d)^{ (w  )_{q}})$ and get 0 by the relation $(\ref{rela})$. So we are left with the term
$u-w = p^m$, namely, $w = l$. Since $\left\langle 
\begin{array}{@{}c@{}}
p^m+l \\ l
\end{array}
\right\rangle_{(m), q }$ is invertible in $\Z [ q ]_{(p,q-1)}$ by Lemma \ref{ipmq},  we get the formula (\ref{cdk}) and the result follows.
\end{proof}

We also need the following lemma.

\begin{lem}
{\rm (cf. \cite[Proposition 3.1]{LQ01})}
\label{exla}
Let $P$ be an element of the form  $(\ul{d \xi})^{ (\ul{k}_{(1)})_{q} } \otimes' \cdots \otimes' (\ul{d \xi})^{ (\ul{k}_{(j)})_{q} }$. Then for $0 < l < p^m,$ we have
\[
h \left[ d^{\bullet}( \xi_d^{ \{ k  \}_{(m),q}} (\ul{d \wh{\xi}})^{ (\ul{l} )_{q}} (d \xi_d)^{ (l  )_{q}}) \otimes' P \right] = h  d^{\bullet}( \xi_d^{ \{ k  \}_{(m),q}} (\ul{d \wh{\xi}})^{ (\ul{l} )_{q}} (d \xi_d)^{ (l  )_{q}}) \otimes' P.
\]
\end{lem}
\begin{proof}
We follow the proof of \cite{LQ01}.
Since $h$ is $A$-linear and
\[
d^{\bullet}( \xi_d^{ \{ k  \}_{(m),q}} (\ul{d \wh{\xi}})^{ (\ul{l} )_{q}} (d \xi_d)^{ (l  )_{q}}) = d^{\bullet}( \xi_d^{ \{ k  \}_{(m),q}}) \otimes' (\ul{d \wh{\xi}})^{ (\ul{l} )_{q}} (d \xi_d)^{ (l  )_{q}} + \xi_d^{ \{ k  \}_{(m),q}} d^{\bullet}((\ul{d \wh{\xi}})^{ (\ul{l} )_{q}} (d \xi_d)^{ (l  )_{q}}),
\]
it suffices to show that we can move the $P$ outside in \[h \left[  d^{\bullet}( \xi_d^{ \{ k  \}_{(m),q}}) \otimes' (\ul{d \wh{\xi}})^{ (\ul{l} )_{q}} (d \xi_d)^{ (l  )_{q}} \otimes' P \right]\] and \[h \left[ \xi_d^{ \{ k  \}_{(m),q}} d^{\bullet}((\ul{d \wh{\xi}})^{ (\ul{l} )_{q}} (d \xi_d)^{ (l  )_{q}}) \otimes' P \right].\]

First, we compute
\[
d^{\bullet}( \xi_d^{ \{ k  \}_{(m),q}}) \otimes' (\ul{d \wh{\xi}})^{ (\ul{l} )_{q}} (d \xi_d)^{ (l  )_{q}} = 
\sum_{0 < k' \le k} \left\langle 
\begin{array}{@{}c@{}}
k \\ k'
\end{array}
\right\rangle_{(m), q }
\xi_d^{ \{ k - k'  \}_{(m),q}}  (d \xi_d)^{ ( k' )_{q}} \otimes' (\ul{d \wh{\xi}})^{ (\ul{l} )_{q}} (d \xi_d)^{ (l  )_{q}}.
\]
For the terms $0 < k' < p^m$, we can use (\ref{rema}), and for the term $k' = p^m$, since $l \ne 0$, we can use (\ref{remc}) to move the $P$ outside in $h \left[  d^{\bullet}( \xi_d^{ \{ k  \}_{(m),q}}) \otimes' (\ul{d \wh{\xi}})^{ (\ul{l} )_{q}} (d \xi_d)^{ (l  )_{q}} \otimes' P \right]$. 

Next, we compute
{\fontsize{9pt}{10pt}\selectfont
\[
d^{\bullet}((\ul{d \wh{\xi}})^{ (\ul{l} )_{q}} (d \xi_d)^{ (l  )_{q}}) = - \sum_{0 < (\ul{l}' , l') < (\ul{l} , l)} 
\left\langle 
\begin{array}{@{}c@{}}
\ul{l} \\ \ul{l}'
\end{array}
\right\rangle_{(m), q }
\left\langle 
\begin{array}{@{}c@{}}
l \\ l'
\end{array}
\right\rangle_{(m), q }
(\ul{d \wh{\xi}})^{ ( \ul{l} - \ul{l}' )_{q}} (d \xi_d)^{ (l - l'  )_{q}} \otimes' (\ul{d \wh{\xi}})^{ (\ul{l}' )_{q}} (d \xi_d)^{ (l'  )_{q}}.
\]
}

\noindent
As in the proof of Proposition \ref{hcond}, we are left with two sums for the element \linebreak $h \left[ \xi_d^{ \{ k  \}_{(m),q}} d^{\bullet}((\ul{d \wh{\xi}})^{ (\ul{l} )_{q}} (d \xi_d)^{ (l  )_{q}}) \otimes' P \right]$:
\begin{align*}
&- \sum_{0 \le \ul{l}'  < \ul{l} } 
\left\langle 
\begin{array}{@{}c@{}}
\ul{l} \\ \ul{l}'
\end{array}
\right\rangle_{(m), q }
h \left[ \xi_d^{ \{ k  \}_{(m),q}}
(\ul{d \wh{\xi}})^{ ( \ul{l} - \ul{l}' )_{q}}  \otimes' (\ul{d \wh{\xi}})^{ (\ul{l}' )_{q}}  (d \xi_d)^{ (l  )_{q}} \otimes' P \right] \\
&- \sum_{0 \le l'  < l }  \left\langle 
\begin{array}{@{}c@{}}
l \\ l'
\end{array}
\right\rangle_{(m), q }
h \left[ \xi_d^{ \{ k  \}_{(m),q}}
(d \xi_d)^{ (l - l'  )_{q}} \otimes' (\ul{d \wh{\xi}})^{ (\ul{l} )_{q}}  (d \xi_d)^{ (l'  )_{q}} \otimes' P \right].
\end{align*}
Since in the first sum, we have $0 < l < p^m$, and in the second sum, we have $0 < l-l' < p^m$, we can use (\ref{rema}) to move the $P$ outside in $h \left[ \xi_d^{ \{ k  \}_{(m),q}} d^{\bullet}((\ul{d \wh{\xi}})^{ (\ul{l} )_{q}} (d \xi_d)^{ (l  )_{q}}) \otimes' P \right]$.
\end{proof}

\begin{prop}
\label{proj}
{\rm (cf. \cite[Proposition 3.2]{LQ01})}
We have $h \circ d^{\bullet}+  d^{\bullet} \circ h = \mathrm{Id} - \pi$,  where $\pi$ is the projector that sends $\xi_i^{ \{ k_i  \}_{(m),q}}$ and $(d \xi_i)^{ (l_i  )_{q}}$ to itself if $i < d$ and to $0$ if $i=d$.
\end{prop}
\begin{proof}
We follow the proof and make some corrections of \cite{LQ01}. It is enough to check the identity on a set of generators of $L\Omega^{(m)}_{A/R,q}  (r)$. By using (\ref{mclo}) and (\ref{dxik}) repeatedly, we can choose the generators of the form
\[
Q \otimes' \xi_d^{ \{ k  \}_{(m),q}} (\ul{d \wh{\xi}})^{ (\ul{l} )_{q}} (d \xi_d)^{ (l  )_{q}} \otimes' P
\]
with $Q := \wh{\ul{\xi}}^{ \{ \ul{k}  \}_{(m),q}}  (\ul{d \wh{\xi}})^{ (\ul{k}_{(1)})_{q} } \otimes' \cdots \otimes' (\ul{d \wh{\xi}})^{ (\ul{k}_{(s-1)})_{q} }, P:= (\ul{d \xi})^{ (\ul{k}_{(s+1)})_{q} } \otimes' \cdots \otimes' (\ul{d \xi})^{ (\ul{k}_{(r)})_{q} }$ and, either $s=r$ or $0 < l < p^m$. By using Lemma \ref{xiin},  (\ref{rema}) and (\ref{remb}), we have
{ \allowdisplaybreaks
{\fontsize{9pt}{10pt}\selectfont
\begin{align*}
&(h d^{\bullet}+  d^{\bullet}  h) \left[ Q \otimes' \xi_d^{ \{ k  \}_{(m),q}} (\ul{d \wh{\xi}})^{ (\ul{l} )_{q}} (d \xi_d)^{ (l  )_{q}} \otimes' P \right] \\
= &h  d^{\bullet} \left[ Q \otimes' \xi_d^{ \{ k  \}_{(m),q}} (\ul{d \wh{\xi}})^{ (\ul{l} )_{q}} (d \xi_d)^{ (l  )_{q}} \otimes' P \right] + d^{\bullet}  h \left[ Q \otimes' \xi_d^{ \{ k  \}_{(m),q}} (\ul{d \wh{\xi}})^{ (\ul{l} )_{q}} (d \xi_d)^{ (l  )_{q}} \otimes' P \right] \\
= &h   \left[ d^{\bullet} Q \otimes' \xi_d^{ \{ k  \}_{(m),q}} (\ul{d \wh{\xi}})^{ (\ul{l} )_{q}} (d \xi_d)^{ (l  )_{q}} \otimes' P \right] + (-1)^{s-1} h   \left[  Q \otimes' d^{\bullet}(\xi_d^{ \{ k  \}_{(m),q}} (\ul{d \wh{\xi}})^{ (\ul{l} )_{q}} (d \xi_d)^{ (l  )_{q}}) \otimes' P \right] \\
+ &(-1)^s h   \left[  Q \otimes' \xi_d^{ \{ k  \}_{(m),q}} (\ul{d \wh{\xi}})^{ (\ul{l} )_{q}} (d \xi_d)^{ (l  )_{q}} \otimes' d^{\bullet}P \right] + (-1)^{s-1} d^{\bullet} \left[ Q \otimes' h(\xi_d^{ \{ k  \}_{(m),q}} (\ul{d \wh{\xi}})^{ (\ul{l} )_{q}} (d \xi_d)^{ (l  )_{q}}) P \right] \\
= &(-1)^s  d^{\bullet} Q \otimes' h   \left[ \xi_d^{ \{ k  \}_{(m),q}} (\ul{d \wh{\xi}})^{ (\ul{l} )_{q}} (d \xi_d)^{ (l  )_{q}} \right] P +  Q \otimes' h   \left[   d^{\bullet}(\xi_d^{ \{ k  \}_{(m),q}} (\ul{d \wh{\xi}})^{ (\ul{l} )_{q}} (d \xi_d)^{ (l  )_{q}}) \otimes' P \right] \\
- &Q \otimes' h \left[ \xi_d^{ \{ k  \}_{(m),q}} (\ul{d \wh{\xi}})^{ (\ul{l} )_{q}} (d \xi_d)^{ (l  )_{q}} \right] d^{\bullet}P  + (-1)^{s-1} d^{\bullet}  Q \otimes' h \left[ \xi_d^{ \{ k  \}_{(m),q}} (\ul{d \wh{\xi}})^{ (\ul{l} )_{q}} (d \xi_d)^{ (l  )_{q}} \right] P \\
+ &Q \otimes' d^{\bullet}h \left[ \xi_d^{ \{ k  \}_{(m),q}} (\ul{d \wh{\xi}})^{ (\ul{l} )_{q}} (d \xi_d)^{ (l  )_{q}} \right] \otimes' P + Q \otimes' h \left[ \xi_d^{ \{ k  \}_{(m),q}} (\ul{d \wh{\xi}})^{ (\ul{l} )_{q}} (d \xi_d)^{ (l  )_{q}} \right] d^{\bullet} P \\
= &Q \otimes' h   \left[   d^{\bullet}(\xi_d^{ \{ k  \}_{(m),q}} (\ul{d \wh{\xi}})^{ (\ul{l} )_{q}} (d \xi_d)^{ (l  )_{q}}) \otimes' P \right] + Q \otimes' d^{\bullet}h \left[ \xi_d^{ \{ k  \}_{(m),q}} (\ul{d \wh{\xi}})^{ (\ul{l} )_{q}} (d \xi_d)^{ (l  )_{q}} \right] \otimes' P.
\end{align*}
}}

\noindent
Then, by Lemma \ref{exla}, we see that

{\fontsize{9pt}{10pt}\selectfont
\begin{align*}
(h  d^{\bullet}+  d^{\bullet}  h) \left[ Q \otimes' \xi_d^{ \{ k  \}_{(m),q}} (\ul{d \wh{\xi}})^{ (\ul{l} )_{q}} (d \xi_d)^{ (l  )_{q}} \otimes' P \right] = Q \otimes' (h  d^{\bullet}+  d^{\bullet}  h) \left[ \xi_d^{ \{ k  \}_{(m),q}} (\ul{d \wh{\xi}})^{ (\ul{l} )_{q}} (d \xi_d)^{ (l  )_{q}} \right] \otimes' P,
\end{align*}}

\noindent
so we may assume that $r \le 1$. We need to show that
\[
(h  d^{\bullet}+  d^{\bullet}  h) \left[ \xi_d^{ \{ k  \}_{(m),q}} (\ul{d \wh{\xi}})^{ (\ul{l} )_{q}} (d \xi_d)^{ (l  )_{q}} \right] = \xi_d^{ \{ k  \}_{(m),q}} (\ul{d \wh{\xi}})^{ (\ul{l} )_{q}} (d \xi_d)^{ (l  )_{q}}
\]
unless $k=l=0$ in which case we get 0. We write $\omega := \xi_d^{ \{ k  \}_{(m),q}} (d \xi_d)^{ (l  )_{q}}$ and compute $(h  d^{\bullet}+  d^{\bullet}  h) \left[ \omega (\ul{d \wh{\xi}})^{ (\ul{l} )_{q}} \right]$. First, we assume that $l \ne 0$ and $\ul{l} \ne 0$. By using (\ref{leibb}) and making the same computation as in the proof of Proposition \ref{hcond}, we have
{\fontsize{9pt}{10pt}\selectfont
\begin{align*}
h  d^{\bullet} \left[ \omega (\ul{d \wh{\xi}})^{ (\ul{l} )_{q}} \right] &= \sum_{0 < k'  \le k }  \left\langle 
\begin{array}{@{}c@{}}
k \\ k'
\end{array}
\right\rangle_{(m), q }
h \left[ \xi_d^{ \{ k-k'  \}_{(m),q}}
(d \xi_d)^{ (k'  )_{q}} \otimes' (\ul{d \wh{\xi}})^{ (\ul{l} )_{q}}  (d \xi_d)^{ (l  )_{q}}  \right] \\
&- \sum_{0 \le \ul{l}'  < \ul{l} } 
\left\langle 
\begin{array}{@{}c@{}}
\ul{l} \\ \ul{l}'
\end{array}
\right\rangle_{(m), q }
h \left[ \xi_d^{ \{ k  \}_{(m),q}}
(\ul{d \wh{\xi}})^{ ( \ul{l} - \ul{l}' )_{q}}  \otimes' (\ul{d \wh{\xi}})^{ (\ul{l}' )_{q}}  (d \xi_d)^{ (l  )_{q}} \right] \\
&- \sum_{0 \le l'  < l }  \left\langle 
\begin{array}{@{}c@{}}
l \\ l'
\end{array}
\right\rangle_{(m), q }
h \left[ \xi_d^{ \{ k  \}_{(m),q}}
(d \xi_d)^{ (l - l'  )_{q}} \otimes' (\ul{d \wh{\xi}})^{ (\ul{l} )_{q}}  (d \xi_d)^{ (l'  )_{q}}  \right].
\end{align*}
}

\noindent
For the term $k' = p^m$ in the first sum, we can use (\ref{cdb}), (\ref{rema}), and for the rest, we can use (\ref{rema}) to get
{\fontsize{9pt}{10pt}\selectfont
\begin{align*}
&h  d^{\bullet} \left[ \omega (\ul{d \wh{\xi}})^{ (\ul{l} )_{q}} \right]  \\ 
&= 
\sum_{0 < k'  \le k }  \left\langle 
\begin{array}{@{}c@{}}
k \\ k'
\end{array}
\right\rangle_{(m), q }
h \left[ \xi_d^{ \{ k-k'  \}_{(m),q}}
(d \xi_d)^{ (k'  )_{q}} \otimes' (d \xi_d)^{ (l  )_{q}} \right] (\ul{d \wh{\xi}})^{ (\ul{l} )_{q}}  + (\ul{d \wh{\xi}})^{ (\ul{l} )_{q}} \otimes' h \left[ \xi_d^{ \{ k  \}_{(m),q}} (d \xi_d)^{ (l  )_{q}} \right] \\
&- \sum_{0 < l'  < l }  \left\langle 
\begin{array}{@{}c@{}}
l \\ l'
\end{array}
\right\rangle_{(m), q }
h \left[ \xi_d^{ \{ k  \}_{(m),q}}
(d \xi_d)^{ (l - l'  )_{q}} \otimes' (d \xi_d)^{ (l'  )_{q}}  \right] (\ul{d \wh{\xi}})^{ (\ul{l} )_{q}} - h \left[ \xi_d^{ \{ k  \}_{(m),q}} (d \xi_d)^{ (l  )_{q}} \right] (\ul{d \wh{\xi}})^{ (\ul{l} )_{q}}.
\end{align*}
}

\noindent
The following calculations are not correct in \cite{LQ01} (but the end result is), so we make corrections here. By using (\ref{leibb}), we have
\[
h  d^{\bullet}(\omega)(\ul{d \wh{\xi}})^{ (\ul{l} )_{q}}+(\ul{d \wh{\xi}})^{ (\ul{l} )_{q}} \otimes' h (\omega) - h (\omega) (\ul{d \wh{\xi}})^{ (\ul{l} )_{q}},
\]
and by the formula (\ref{omy}), this is equal to
\[
h  d^{\bullet}(\omega)(\ul{d \wh{\xi}})^{ (\ul{l} )_{q}}+ d^{\bullet}h(\omega) (\ul{d \wh{\xi}})^{ (\ul{l} )_{q}}.
\]
On the other hand, by the definition of $h$, we can check that $d^{\bullet} h   \left[ \omega (\ul{d \wh{\xi}})^{ (\ul{l} )_{q}} \right] = 0$.  So we see that 
\[
(h  d^{\bullet}+  d^{\bullet}  h) \left[ \omega (\ul{d \wh{\xi}})^{ (\ul{l} )_{q}} \right] = (h  d^{\bullet}+  d^{\bullet}  h) (\omega) (\ul{d \wh{\xi}})^{ (\ul{l} )_{q}}.
\]
Obviously, this formula still holds when $\ul{l} = 0$.

Now we consider the case $l = 0$ and $\ul{l} \ne 0$ in order to move $(\ul{d \wh{\xi}})^{ (\ul{l} )_{q}}$ outside. By using (\ref{leibb}), we have
{\fontsize{9pt}{10pt}\selectfont
\begin{align}
\label{llnz}
&h  d^{\bullet}\left[ \xi_d^{ \{ k  \}_{(m),q}}  (\ul{d \wh{\xi}})^{ (\ul{l} )_{q}} \right] \notag \\
=& \sum_{0 < k'  \le k }  \left\langle 
\begin{array}{@{}c@{}}
k \\ k'
\end{array}
\right\rangle_{(m), q }
h \left[ \xi_d^{ \{ k-k'  \}_{(m),q}}
(d \xi_d)^{ (k'  )_{q}} \otimes' (\ul{d \wh{\xi}})^{ (\ul{l} )_{q}}   \right] 
- \sum_{0 < \ul{l}'  < \ul{l} } 
\left\langle 
\begin{array}{@{}c@{}}
\ul{l} \\ \ul{l}'
\end{array}
\right\rangle_{(m), q }
h \left[ \xi_d^{ \{ k  \}_{(m),q}}
(\ul{d \wh{\xi}})^{ ( \ul{l} - \ul{l}' )_{q}}  \otimes' (\ul{d \wh{\xi}})^{ (\ul{l}' )_{q}}   \right].
\end{align}
}

\noindent
The second sum is equal to 0 by the definition of $h$. The calculations for the first sum is not obvious as in the proof of \cite{LQ01}, so we give them explicitly. If $k' \ne p^m$, then we can move $(\ul{d \wh{\xi}})^{ (\ul{l} )_{q}}$ outside by using (\ref{rema}). For the term $k' = p^m$, we may calculate it by using (\ref{cdc}):
\begin{align*}
&h \left[ \xi_d^{ \{ k-p^m  \}_{(m),q}}
(d \xi_d)^{ (p^m  )_{q}} \otimes' (\ul{d \wh{\xi}})^{ (\ul{l} )_{q}}   \right]  \\
= &-h \left[ \xi_d^{ \{ k-p^m  \}_{(m),q}}
(\ul{d \wh{\xi}})^{ (\ul{l} )_{q}}  \otimes' (d \xi_d)^{  (p^m  )_{q}}   \right] \\
&- \sum_{0 < l'  < p^m }
\left\langle 
\begin{array}{@{}c@{}}
p^m \\ l' 
\end{array}
\right\rangle_{(m), q }
h \left[ \xi_d^{ \{ k-p^m  \}_{(m),q}}
(d \xi_d)^{ (p^m - l'  )_{q}} \otimes' (\ul{d \wh{\xi}})^{ (\ul{l} )_{q}} (d \xi_d)^{ (l'  )_{q}}  \right].
\end{align*}
By excluding the trivial case, we may assume that $k>0$ and $p^m \mid k$. Then we can use the definition of $h$ and the formula (\ref{omy}) to get
{\fontsize{9pt}{10pt}\selectfont
\begin{align}
\label{ssts}
= 
&\left\langle 
\begin{array}{@{}c@{}}
k \\ p^m 
\end{array}
\right\rangle_{(m), q }^{-1}
(\ul{d \wh{\xi}})^{ (\ul{l} )_{q}} \otimes' \xi_d^{ \{ k  \}_{(m),q}} 
- \sum_{0 < l'  < p^m }
\left\langle 
\begin{array}{@{}c@{}}
p^m \\ l' 
\end{array}
\right\rangle_{(m), q }
\left\langle 
\begin{array}{@{}c@{}}
k-l' \\ k-p^m
\end{array}
\right\rangle_{(m), q }^{-1}
\xi_d^{ \{ k-l'  \}_{(m),q}}
(\ul{d \wh{\xi}})^{ (\ul{l} )_{q}} (d \xi_d)^{ (l'  )_{q}} \notag \\
=&\left\langle 
\begin{array}{@{}c@{}}
k \\ p^m 
\end{array}
\right\rangle_{(m), q }^{-1}
\xi_d^{ \{ k  \}_{(m),q}}  (\ul{d \wh{\xi}})^{ (\ul{l} )_{q}} +
\left\langle 
\begin{array}{@{}c@{}}
k \\ p^m 
\end{array}
\right\rangle_{(m), q }^{-1}
\cdot
\sum_{0 < l''  \le k }
\left\langle 
\begin{array}{@{}c@{}}
k \\ l''
\end{array}
\right\rangle_{(m), q }
\xi_d^{ \{ k-l''  \}_{(m),q}}
(\ul{d \wh{\xi}})^{ (\ul{l} )_{q}} (d \xi_d)^{ (l''  )_{q}} \notag \\
- &\sum_{0 < l'  < p^m }
\left\langle 
\begin{array}{@{}c@{}}
p^m \\ l' 
\end{array}
\right\rangle_{(m), q }
\left\langle 
\begin{array}{@{}c@{}}
k-l' \\ k-p^m
\end{array}
\right\rangle_{(m), q }^{-1}
\xi_d^{ \{ k-l'  \}_{(m),q}}
(\ul{d \wh{\xi}})^{ (\ul{l} )_{q}} (d \xi_d)^{ (l'  )_{q}}.
\end{align}
}

\noindent
We can check that
\[
\left\langle 
\begin{array}{@{}c@{}}
k \\ p^m 
\end{array}
\right\rangle_{(m), q }^{-1}
\left\langle 
\begin{array}{@{}c@{}}
k \\ l'
\end{array}
\right\rangle_{(m), q }=
\left\langle 
\begin{array}{@{}c@{}}
p^m \\ l' 
\end{array}
\right\rangle_{(m), q }
\left\langle 
\begin{array}{@{}c@{}}
k-l' \\ k-p^m
\end{array}
\right\rangle_{(m), q }^{-1},
\]
and since $(\ul{d \wh{\xi}})^{ (\ul{l} )_{q}} (d \xi_d)^{ (l''  )_{q}}$ is 0 when $p^m \le l'' \le k$, we can cancel the sums in (\ref{ssts}), so we get
{\fontsize{9pt}{10pt}\selectfont
\begin{align*}
h \left[ \xi_d^{ \{ k-p^m  \}_{(m),q}}
(d \xi_d)^{ (p^m  )_{q}} \otimes' (\ul{d \wh{\xi}})^{ (\ul{l} )_{q}}   \right] = 
\left\langle 
\begin{array}{@{}c@{}}
k \\ p^m 
\end{array}
\right\rangle_{(m), q }^{-1}
\xi_d^{ \{ k  \}_{(m),q}}  (\ul{d \wh{\xi}})^{ (\ul{l} )_{q}}
= h \left[ \xi_d^{ \{ k-p^m  \}_{(m),q}}
(d \xi_d)^{ (p^m  )_{q}}  \right] (\ul{d \wh{\xi}})^{ (\ul{l} )_{q}}.
\end{align*}
}

\noindent
For (\ref{llnz}), after moving $(\ul{d \wh{\xi}})^{ (\ul{l} )_{q}}$ outside in the first sum, we have
\[
h  d^{\bullet}\left[ \xi_d^{ \{ k  \}_{(m),q}}  (\ul{d \wh{\xi}})^{ (\ul{l} )_{q}} \right] = h  d^{\bullet} ( \xi_d^{ \{ k  \}_{(m),q}} )  (\ul{d \wh{\xi}})^{ (\ul{l} )_{q}}.
\]
On the other hand, by the definition of $h$, we have \[ d^{\bullet} h  \left[ \xi_d^{ \{ k  \}_{(m),q}}  (\ul{d \wh{\xi}})^{ (\ul{l} )_{q}} \right] = 0 = d^{\bullet} h  ( \xi_d^{ \{ k  \}_{(m),q}} ) (\ul{d \wh{\xi}})^{ (\ul{l} )_{q}}. \] So we get
\[
(h  d^{\bullet} + d^{\bullet} h )\left[ \xi_d^{ \{ k  \}_{(m),q}}  (\ul{d \wh{\xi}})^{ (\ul{l} )_{q}} \right] = (h  d^{\bullet} + d^{\bullet} h ) ( \xi_d^{ \{ k  \}_{(m),q}} )  (\ul{d \wh{\xi}})^{ (\ul{l} )_{q}}.
\]

By the arguments above, we are reduced to showing the equation when  $d=1$:
\[
(h  d^{\bullet} + d^{\bullet} h )\left[ \xi^{ \{ k  \}_{(m),q}}  (d \xi)^{ (l )_{q}} \right] = \xi^{ \{ k  \}_{(m),q}}  (d \xi)^{ (l )_{q}}.
\]
The case $k = l = 0$ is trivial. For the case $k \ne 0$ and $l=0$, by (\ref{dxik}), we have
\[
h  d^{\bullet} \left[ \xi^{ \{ k  \}_{(m),q}} \right] = \sum_{0 < k' \le k} \left\langle 
\begin{array}{@{}c@{}}
k \\ k'
\end{array}
\right\rangle_{(m), q }
h \left[ \xi^{ \{ k - k'  \}_{(m),q}}  (d \xi)^{ ( k' )_{q}} \right].
\]
We write $k = p^mr + s$ with $0 < s \le p^m$. Then, by the definition of $h$, we see that $h \left[ \xi^{ \{ k - k'  \}_{(m),q}}  (d \xi)^{ ( k' )_{q}} \right] = 0$ unless $p^m \mid (k - k')$ and $0 < k' \le p^m$, namely, $k' = s$. So we have
{\fontsize{9pt}{10pt}\selectfont
\[
(h  d^{\bullet} + d^{\bullet} h ) \left[ \xi^{ \{ k  \}_{(m),q}} \right] =
\left\langle 
\begin{array}{@{}c@{}}
k \\ s
\end{array}
\right\rangle_{(m), q }
h \left[ \xi^{ \{ k - s  \}_{(m),q}}  (d \xi)^{ ( s )_{q}} \right]
=
\left\langle 
\begin{array}{@{}c@{}}
k \\ s
\end{array}
\right\rangle_{(m), q }
\left\langle 
\begin{array}{@{}c@{}}
k \\ s
\end{array}
\right\rangle_{(m), q }^{-1}
\xi^{ \{ k - s + s  \}_{(m),q}} = \xi^{ \{ k   \}_{(m),q}}.
\]}

\noindent
Now we assume that $l \ne 0$. If $k = 0$, then we can check that
\begin{align*}
(h  d^{\bullet} + d^{\bullet} h )\left[  (d \xi)^{ (l )_{q}} \right] &= h  d^{\bullet}\left[  (d \xi)^{ (l )_{q}} \right] + d^{\bullet} h \left[  (d \xi)^{ (l )_{q}} \right] \\
&= -\sum_{0< l' < l}
\left\langle 
\begin{array}{@{}c@{}}
l \\ l' 
\end{array}
\right\rangle_{(m), q }
h \left[  (d \xi)^{ (l- l' )_{q}} \otimes' (d \xi)^{ (l' )_{q}} \right] + d^{\bullet}(\xi^{ \{ l  \}_{(m),q}} ) \\
&= -\sum_{0< l' < l}
\left\langle 
\begin{array}{@{}c@{}}
l \\ l' 
\end{array}
\right\rangle_{(m), q }
\xi^{ \{ l - l' \}_{(m),q}}(d \xi)^{ (l' )_{q}}  \\
&\phantom{=} \ + 
\sum_{0< l'' \le l}
\left\langle 
\begin{array}{@{}c@{}}
l \\ l'' 
\end{array}
\right\rangle_{(m), q }
\xi^{ \{ l - l'' \}_{(m),q}}(d \xi)^{ (l'' )_{q}} = (d \xi)^{ (l )_{q}}.
\end{align*}
Therefore, we are left with the case $k \ne 0$ and $l \ne 0$. We have
\begin{align*}
h  d^{\bullet}\left[  \xi^{ \{ k  \}_{(m),q}}  (d \xi)^{ (l )_{q}} \right] &= h  \left[  d^{\bullet} \xi^{ \{ k  \}_{(m),q}} \otimes' (d \xi)^{ (l )_{q}} \right] + h  \left[   \xi^{ \{ k  \}_{(m),q}}  d^{\bullet}((d \xi)^{ (l )_{q}}) \right] \\
&= \sum_{0< k' \le k}
\left\langle 
\begin{array}{@{}c@{}}
k \\ k' 
\end{array}
\right\rangle_{(m), q }
h  \left[  \xi^{ \{ k - k'  \}_{(m),q}} (d \xi)^{ (k' )_{q}} \otimes'  (d \xi)^{ (l )_{q}} \right] \\
&\phantom{=} \ -\sum_{0< l' < l}
\left\langle 
\begin{array}{@{}c@{}}
l \\ l' 
\end{array}
\right\rangle_{(m), q }
h \left[ \xi^{ \{ k  \}_{(m),q}} (d \xi)^{ (l- l' )_{q}} \otimes' (d \xi)^{ (l' )_{q}} \right].
\end{align*}
First, we assume that $p^m \nmid k$, and we write $k = p^mr + s$ with $0 < s < p^m$. Then the second sum is 0, and for the first sum, all terms are 0 unless $k' = s$. Thus, we see that
\[
h  d^{\bullet}\left[  \xi^{ \{ k  \}_{(m),q}}  (d \xi)^{ (l )_{q}} \right] = 
\left\langle 
\begin{array}{@{}c@{}}
k \\ s
\end{array}
\right\rangle_{(m), q }
\left\langle 
\begin{array}{@{}c@{}}
k \\ s
\end{array}
\right\rangle_{(m), q }^{-1}
\xi^{ \{ k  \}_{(m),q}}  (d \xi)^{ (l )_{q}}=\xi^{ \{ k  \}_{(m),q}}  (d \xi)^{ (l )_{q}}.
\]
Since in this case $d^{\bullet} h  \left[  \xi^{ \{ k  \}_{(m),q}}  (d \xi)^{ (l )_{q}} \right] = 0$, we get the equation as desired.

Now, we are left with the case $p^m \mid k$ and $k \ne 0$. We have 
\begin{align}
\label{dbhx}
d^{\bullet} h  \left[  \xi^{ \{ k  \}_{(m),q}}  (d \xi)^{ (l )_{q}} \right] &=  
d^{\bullet}\left[
\left\langle 
\begin{array}{@{}c@{}}
k+l \\ k
\end{array}
\right\rangle_{(m), q }^{-1} 
\xi^{ \{ k+l  \}_{(m),q}} \right] \notag \\
&= 
\left\langle 
\begin{array}{@{}c@{}}
k+l \\ k
\end{array}
\right\rangle_{(m), q }^{-1}  \cdot
\sum_{0< k' \le p^m}
\left\langle 
\begin{array}{@{}c@{}}
k+l \\ k' 
\end{array}
\right\rangle_{(m), q }
 \xi^{ \{ k+l - k'  \}_{(m),q}} (d \xi)^{ (k' )_{q}}
\end{align}
and
\begin{align*}
h d^{\bullet}   \left[  \xi^{ \{ k  \}_{(m),q}}  (d \xi)^{ (l )_{q}} \right] &=  
\sum_{0< k' \le p^m}
\left\langle 
\begin{array}{@{}c@{}}
k \\ k' 
\end{array}
\right\rangle_{(m), q }
h\left[  \xi^{ \{ k- k'  \}_{(m),q}} (d \xi)^{ (k' )_{q}} \otimes'  (d \xi)^{ (l )_{q}} \right] \\
&\phantom{=} \ -\sum_{0< l' < l}
\left\langle 
\begin{array}{@{}c@{}}
l \\ l' 
\end{array}
\right\rangle_{(m), q }
h \left[ \xi^{ \{ k  \}_{(m),q}} (d \xi)^{ (l- l' )_{q}} \otimes' (d \xi)^{ (l' )_{q}} \right] \\
&= 
\left\langle 
\begin{array}{@{}c@{}}
k \\ p^m
\end{array}
\right\rangle_{(m), q }
h\left[  \xi^{ \{ k- p^m  \}_{(m),q}} (d \xi)^{ (p^m )_{q}} \otimes'  (d \xi)^{ (l )_{q}} \right] \\
&\phantom{=} \ -\sum_{0< l' < l}
\left\langle 
\begin{array}{@{}c@{}}
l \\ l' 
\end{array}
\right\rangle_{(m), q }
\left\langle 
\begin{array}{@{}c@{}}
k+l-l' \\ k
\end{array}
\right\rangle_{(m), q }^{-1}
\xi^{ \{ k+l-l'  \}_{(m),q}} (d \xi)^{ (l' )_{q}}.
\end{align*}
By the formula (\ref{cdb}), we see that:
\begin{align*}
&h\left[  \xi^{ \{ k- p^m  \}_{(m),q}} (d \xi)^{ (p^m )_{q}} \otimes'  (d \xi)^{ (l )_{q}} \right] \\
= &- \sum_{l< l'' \le p^m }
\left\langle 
\begin{array}{@{}c@{}}
p^m + l \\ l
\end{array}
\right\rangle_{(m), q }^{-1}
\left\langle 
\begin{array}{@{}c@{}}
p^m + l \\ l''
\end{array}
\right\rangle_{(m), q }
h\left[  \xi^{ \{ k- p^m  \}_{(m),q}} (d \xi)^{ (p^m + l - l'' )_{q}} \otimes'  (d \xi)^{ (l'' )_{q}} \right] \\
= &- \sum_{l< l'' \le p^m }
\left\langle 
\begin{array}{@{}c@{}}
p^m + l \\ l
\end{array}
\right\rangle_{(m), q }^{-1}
\left\langle 
\begin{array}{@{}c@{}}
p^m + l \\ l''
\end{array}
\right\rangle_{(m), q }
\left\langle 
\begin{array}{@{}c@{}}
k + l - l'' \\ k-p^m
\end{array}
\right\rangle_{(m), q }^{-1}
 \xi^{ \{ k + l - l''  \}_{(m),q}}  (d \xi)^{ (l'' )_{q}}.
\end{align*}
Now, we can check that
\[
\left\langle 
\begin{array}{@{}c@{}}
k \\ p^m
\end{array}
\right\rangle_{(m), q }
\left\langle 
\begin{array}{@{}c@{}}
p^m + l \\ l
\end{array}
\right\rangle_{(m), q }^{-1}
\left\langle 
\begin{array}{@{}c@{}}
p^m + l \\ l''
\end{array}
\right\rangle_{(m), q }
\left\langle 
\begin{array}{@{}c@{}}
k + l - l'' \\ k-p^m
\end{array}
\right\rangle_{(m), q }^{-1}
=
\left\langle 
\begin{array}{@{}c@{}}
k+l \\ k
\end{array}
\right\rangle_{(m), q }^{-1}
\left\langle 
\begin{array}{@{}c@{}}
k + l \\ l''
\end{array}
\right\rangle_{(m), q },
\]
so it follows that
\begin{align*}
&\left\langle 
\begin{array}{@{}c@{}}
k \\ p^m
\end{array}
\right\rangle_{(m), q }
h\left[  \xi^{ \{ k- p^m  \}_{(m),q}} (d \xi)^{ (p^m )_{q}} \otimes'  (d \xi)^{ (l )_{q}} \right] \\
=
&-
\left\langle 
\begin{array}{@{}c@{}}
k+l \\ k
\end{array}
\right\rangle_{(m), q }^{-1}
\cdot
\sum_{l< l'' \le p^m }
\left\langle 
\begin{array}{@{}c@{}}
k + l \\ l''
\end{array}
\right\rangle_{(m), q }
\xi^{ \{ k + l - l''  \}_{(m),q}}  (d \xi)^{ (l'' )_{q}}.
\end{align*}
Since we can also check that
\[
\left\langle 
\begin{array}{@{}c@{}}
l \\ l' 
\end{array}
\right\rangle_{(m), q }
\left\langle 
\begin{array}{@{}c@{}}
k+l-l' \\ k
\end{array}
\right\rangle_{(m), q }^{-1} = \left\langle 
\begin{array}{@{}c@{}}
k+l \\ k
\end{array}
\right\rangle_{(m), q }^{-1}
\left\langle 
\begin{array}{@{}c@{}}
k + l \\ l'
\end{array}
\right\rangle_{(m), q },
\]
we have
\begin{align}
\label{hdbx}
h d^{\bullet}   \left[  \xi^{ \{ k  \}_{(m),q}}  (d \xi)^{ (l )_{q}} \right] = 
-
\left\langle 
\begin{array}{@{}c@{}}
k+l \\ k
\end{array}
\right\rangle_{(m), q }^{-1}
\cdot
\sum_{\substack{0< l' \le p^m \\ l' \ne l}}
\left\langle 
\begin{array}{@{}c@{}}
k + l \\ l'
\end{array}
\right\rangle_{(m), q }
\xi^{ \{ k + l - l'  \}_{(m),q}}  (d \xi)^{ (l' )_{q}}.
\end{align}
Adding (\ref{dbhx}) and (\ref{hdbx}), we get
\[
(h d^{\bullet} +  d^{\bullet} h) \left[  \xi^{ \{ k  \}_{(m),q}}  (d \xi)^{ (l )_{q}} \right] = \xi^{ \{ k  \}_{(m),q}}  (d \xi)^{ (l )_{q}},
\]
so the result follows.
\end{proof}

Now, we prove the formal Poincar\'{e} Lemma for the linearized $q$-jet complex \linebreak $L\Omega^{(m)}_{A/R,q}  (\bullet)$.

\begin{thm}
\label{qjfp}
{\rm (cf. \cite[Theorem 3.3]{LQ01})}
The linearized $q$-jet complex $L\Omega^{(m)}_{A/R,q}  (\bullet)$ is a resolution of $A$.
\end{thm}
\begin{proof}
We follow the proof of \cite{LQ01}. 
Since the canonical map \[ A \wh{\otimes}_{R[\ul{X}]^{\wedge}} L\Omega^{(m)}_{R[\ul{X}]^{\wedge}/R,q}(\bullet) \xrightarrow{\simeq} L\Omega^{(m)}_{A/R,q}(\bullet) \] is an isomorphism, it is enough to show the assertion for $A = R[X_1, \dots, X_{d}]^{\wedge}$. We proceed by induction on $d$. We need to show that for $A' = R[X_1, \dots, X_{d-1}]^{\wedge}$,
 the map of complexes $A \wh{\otimes}_{A'} L\Omega^{(m)}_{A'/R,q}(\bullet) \to L\Omega^{(m)}_{A/R,q}(\bullet)$ is a quasi-isomorphism. We check that this map is actually a homotopy equivalence.
Since this map, which is just the inclusion $\xi_i^{ \{ k_i  \}_{(m),q}} \mapsto \xi_i^{ \{ k_i  \}_{(m),q}}$ and $(d \xi_i)^{ (l_i  )_{q}} \mapsto (d \xi_i)^{ (l_i  )_{q}}$ for $i = 1, \dots, d-1$, has an obvious retraction map given by
 $\xi_d^{ \{ k_d  \}_{(m),q}} \mapsto 0$ and $(d \xi_d)^{ (l_d  )_{q}}  \mapsto 0$, it suffices to show that the projector $\pi$ on  $L\Omega^{(m)}_{A/R,q}(\bullet)$ in Proposition $\ref{proj}$ is homotopic to the identity. So we need to construct a homotopy $h$ on $L\Omega^{(m)}_{A/R,q}(\bullet)$ such that $h \circ d^{\bullet}+  d^{\bullet} \circ h = \mathrm{Id} - \pi$,  which was done in Proposition $\ref{proj}$.
\end{proof}

We are now ready to prove the Poincar\'{e} Lemma for the $q$-jet complex.

\begin{prop}
\label{jcpl}
{\rm (cf. \cite[Theorem 4.7]{LQ01})}
If $E \in {\bf CR}((\ol{A}/R)_{m\text{-}q^{p^m}\text{-crys}},\cO^{(m)}_{\ol{A}/R})$,  then 
\[
E \wh{\otimes}_{\cO^{(m)}_{\ol{A}/R}} L^{(m)}(\Omega^{(m)}_{A/R,q} (\bullet))
\]
 forms a complex of $\cO^{(m)}_{\ol{A}/R}$-modules that resolves $E$.
\end{prop}
\begin{proof}
First, as in the proof of Proposition \ref{mqpl}, we have a complex $L^{(m)}(\Omega^{(m)}_{A/R,q} (\bullet))$ of $\cO^{(m)}_{\ol{A}/R}$-modules that gives $L\Omega^{(m)}_{A/R,q}  (\bullet)$ when we evaluate at $A$ by Lemma \ref{lpstr}. 
We need to show that, for any object $(B,I_B)$  of $(\ol{A}/R)_{m\text{-}q^{p^m}\text{-crys}}$, the map
\[
E_B \to E_B \wh{\otimes}_B L^{(m)}(\Omega^{(m)}_{A/R,q} (\bullet))_B
\]
is a quasi-isomorphism.
As in the proof of Proposition \ref{mqpl}, we may assume that $B=A$. Then by Theorem \ref{qjfp}, we see that 
\[
A \to L\Omega^{(m)}_{A/R,q}  (\bullet)
\]
is a quasi-isomorphism. By the complete tensor product of this map with \[ E_A \in {\bf FMod}^{\wedge}_{(p,(p)_{q^{p^m}})}(A), \] we get the desired quasi-isomorphism.
\end{proof}

\begin{cor}
\label{jcgs}
If $E \in {\bf CR}((\ol{A}/R)_{m\text{-}q^{p^m}\text{-crys}},\cO^{(m)}_{\ol{A}/R})$, then we have a quasi-isomorphism
\[
 R \Gamma ((\ol{A}/R)_{m\text{-}q^{p^m}\text{-crys}}, E) \xrightarrow{\simeq} E_A \wh{\otimes}_A \Omega^{(m)}_{A/R,q} (\bullet).
\]
\end{cor}
\begin{proof}
We follow the notations in Section \ref{mqanc}. Then we have
\begin{align*}
R \Gamma ((\ol{A}/R)_{m\text{-}q^{p^m}\text{-crys}}, E) &\simeq Re_{\ol{A}/R*} E \\
&\simeq Re_{\ol{A}/R*} ( E \wh{\otimes}_{\cO^{(m)}_{\ol{A}/R}} L^{(m)}(\Omega^{(m)}_{A/R,q} (\bullet)) ) .
\end{align*}
By Lemma \ref{elle}, this is quasi-isomorphic to \[ Re_{\ol{A}/R*} (L^{(m)}(E_A \wh{\otimes}_A \Omega^{(m)}_{A/R,q} (\bullet))). \] Then, by the  calculations as in the proof of Lemma \ref{rglm}, this is quasi-isomorphic to $E_A \wh{\otimes}_A \Omega^{(m)}_{A/R,q} (\bullet)$.
\end{proof}

\phantomsection
\bibliographystyle{alpha}
\bibliography{myrefs}
\addcontentsline{toc}{section}{References}
\nocite{*}

\end{document}